\documentclass[a4paper,12pt,oneside]{article}

\usepackage[utf8]{inputenc}
\usepackage[english]{babel}
\usepackage{amsthm}
\usepackage{amssymb}
\usepackage{amsmath}
\usepackage{hyperref}
\usepackage{mathrsfs}

\usepackage{pstricks}

 \usepackage{tikz}

\usepackage{graphicx}
\graphicspath{{images/}}

\usepackage{longtable,geometry}
\usepackage{color}

\usepackage{hyperref}

\newtheoremstyle{note}{12pt}{12pt}{}{}{\bfseries}{.}{.5em}{}
\title{\LARGE\textbf{ Non-Hyperbolic Chaotic Dynamics:\\Renormalization and More }} 
\author{Marco Martens and Liviana Palmisano}

\makeatletter
\newtheorem{theo}[equation]{Theorem}

\newtheorem{prop}[equation]{Proposition}
\numberwithin{equation}{subsection}
\newtheorem{defin}[equation]{Definition}

\newtheorem{rem}[equation]{Remark}
\newtheorem{cor}[equation]{Corollary}

\newtheorem{lem}[equation]{Lemma}

\usepackage{babel}

\newcommand{\N}{{\mathbb N}}
\newcommand{\Z}{{\mathbb Z}}
\newcommand{\R}{{\mathbb R}}
\newcommand{\Q}{{\mathbb Q}}

\newcommand{\Cinf}{{{\mathcal C}^\infty}}
\newcommand{\Cd}{{{\mathcal C}^2}}
\newcommand{\Ct}{{{\mathcal C}^3}}
\newcommand{\Cq}{{{\mathcal C}^4}}

\newcommand{\Cuno}{{{\mathcal C}^1}}

\begin{document}
\maketitle

\textcolor{blue}{}\global\long\def\sbr#1{\left[#1\right] }
\textcolor{blue}{}\global\long\def\cbr#1{\left\{  #1\right\}  }
\textcolor{blue}{}\global\long\def\rbr#1{\left(#1\right)}
\textcolor{blue}{}\global\long\def\ev#1{\mathbb{E}{#1}}
\textcolor{blue}{}\global\long\def\R{\mathbb{R}}
\textcolor{blue}{}\global\long\def\E{\mathbb{E}}
\textcolor{blue}{}\global\long\def\norm#1#2#3{\Vert#1\Vert_{#2}^{#3}}
\textcolor{blue}{}\global\long\def\pr#1{\mathbb{P}\rbr{#1}}
\textcolor{blue}{}\global\long\def\qq{\mathbb{Q}}
\textcolor{blue}{}\global\long\def\aa{\mathbb{A}}
\textcolor{blue}{}\global\long\def\ind#1{1_{#1}}
\textcolor{blue}{}\global\long\def\pp{\mathbb{P}}
\textcolor{blue}{}\global\long\def\cleq{\lesssim}
\textcolor{blue}{}\global\long\def\ceq{\eqsim}
\textcolor{blue}{}\global\long\def\Var#1{\text{Var}(#1)}
\textcolor{blue}{}\global\long\def\TDD#1{{\color{red}To\, Do(#1)}}
\textcolor{blue}{}\global\long\def\dd#1{\textnormal{d}#1}
\textcolor{blue}{}\global\long\def\eqdef{:=}
\textcolor{blue}{}\global\long\def\ddp#1#2{\left\langle #1,#2\right\rangle }
\textcolor{blue}{}\global\long\def\En{\mathcal{E}_{n}}
\textcolor{blue}{}\global\long\def\Z{\mathbb{Z}}
\textcolor{blue}{{} }

\textcolor{blue}{}\global\long\def\nC#1{\newconstant{#1}}
\textcolor{blue}{}\global\long\def\C#1{\useconstant{#1}}
\textcolor{blue}{}\global\long\def\nC#1{\newconstant{#1}\text{nC}_{#1}}
\textcolor{blue}{}\global\long\def\C#1{C_{#1}}
\textcolor{blue}{}\global\long\def\meas{\mathcal{M}}
\textcolor{blue}{}\global\long\def\cSpace{\mathcal{C}}
\textcolor{blue}{}\global\long\def\pspace{\mathcal{P}}


\begin{abstract}
%
%
%
In two-dimensional unfoldings of homoclinic tangencies, the parameter space contains codimension-1 laminations whose leaves consist of maps with invariant non-hyperbolic Cantor sets. These sets are wild and unstable in the sense of Newhouse and contain Collet–Eckmann points with dense orbits. Thus, wild and non-hyperbolic chaotic dynamics can coexist on a single invariant set, while persisting along codimension-1 manifolds.

Even more strikingly, each leaf of the lamination contains a map with infinitely many sinks accumulating on the invariant Cantor set carrying the Collet–Eckmann dynamics. In particular, the occurrence of infinitely many sinks is compatible with this type of non-hyperbolic chaotic behavior, with both phenomena organized around the same invariant Cantor set.

To analyze these phenomena, we introduce a generalized renormalization scheme for two-dimensional systems, which describes the dynamics at successive scales and reveals a finite-scale hyperbolic structure within these non-hyperbolic regimes.

\end{abstract}

\section{Introduction}
This paper addresses questions of existence and stability for invariant sets arising in non-hyperbolic dynamical systems. Given a dynamical system $( \mathcal M,f)$ where $ \mathcal M$ is a manifold and $f$ a map acting on it, one aims to identify the regions where the essential dynamics take place, that is, invariant sets capturing the long-term behaviour of typical orbits. Establishing the existence of such sets can already be a challenging task. Once their existence is established, one naturally seeks to understand whether these sets, together with their dynamical properties, persist under perturbations of the system. Are these properties robust and observable across nearby systems? In other words, are they stable? We now introduce a notion of stability adapted to this setting.

Given two dynamical systems \((\mathcal M, f)\) and \((\mathcal M, g)\) with relevant sets \(A(f)\) and \(A(g)\), we say that they are \emph{conjugate} if there exists a homeomorphism \(h : A(f) \to A(g)\) such that $
h \circ f = g \circ h.$ 
One can think of this homeomorphism as a change of coordinates. For instance, if \(A(f)\) and \(A(g)\) are Cantor sets, the conjugacy map ``rearranges'' the pieces of one Cantor set to produce the other. In this sense, the two sets share the same topological structure and are therefore topologically equivalent.

Within the space of dynamical systems, one can consider the class of maps whose relevant sets are topologically equivalent. The structure of this class encodes the stability of the corresponding sets. We say that the set \(A(f)\) is \emph{codimension-\(k\) stable} if the set of systems whose relevant sets are conjugate forms a codimension-\(k\)  submanifold of the space of dynamical systems. This submanifold is called the \emph{manifold of stability}.

The sets exhibiting the strongest form of stability are closed attracting periodic orbits and, more generally, hyperbolic invariant sets. Indeed, such sets are {codimension-\(0\) stable}, meaning that their topology does not change under small perturbations of the system. 

More subtle phenomena already arise in one-dimensional dynamics. For instance, in the class of \(\mathcal{C}^3\) circle maps, the relevant invariant set is the entire circle, but its topological structure is no longer stable under arbitrary perturbations. Instead, it persists only along codimension-1 manifolds in parameter space, and is therefore codimension-1 stable. This phenomenon follows from one of the most profound results in one-dimensional dynamics, namely Herman’s Theorem, see  \cite{Her79}, and for more modern expositions \cite{DMP, KMP}.

The stability of invariant sets in one-dimensional dynamics is by now largely understood, and no fundamentally new phenomena are expected. That said, a number of subtle and interesting questions remain open, mostly concerning finer properties of the dynamics and their dependence on parameters.

In contrast, the situation in two dimensions is considerably richer and remains far from fully understood. Several mechanisms of instability may interact in nontrivial ways, making questions of existence and stability more delicate. 

We now turn to the family of systems at the center of our study, namely \emph{unfoldings of maps with a homoclinic tangency and a transversal intersection}. The precise definition is given in Subsection~\ref{familyofunfoldings}. Such unfoldings provide a natural setting in which intricate invariant sets and coexistence phenomena arise, as they capture the breakdown of hyperbolicity and the emergence of complex dynamics.
A central example is provided by the Hénon family (see Section~\ref{sec:Henonfamily} for a precise definition).
What relevant invariant sets arise in this context, and what is known about their stability?

As expected, a finite number of attracting periodic orbits (sinks) may coexist, and these are {codimension-0 stable}, see \cite{CPT}. Indeed, sinks are hyperbolic sets. More surprisingly, the existence of infinitely many sinks has also been established (see \cite{Newhouse}), together with initial results concerning their stability (see \cite{GS}).

Beyond periodic behaviour, non periodic dynamics have likewise been shown to occur. In particular, the existence of a \emph{period-doubling Cantor attractor}, an invariant, non-hyperbolic, minimal set carrying a measure with zero Lyapunov exponent, has been proved (see \cite{CLM, GvST, LM1}). This set is known to be {codimension-1 stable}. In addition, maps exhibiting two period-doubling Cantor attractors have been constructed, and these are {codimension-2 stable} (see \cite{Pal}).

Even more intricate invariant sets, known as \emph{strange attractors}, have been shown to exist in H\'enon-like families. Strange attractors are invariant, non-hyperbolic, and chaotic sets, in the sense that they support an invariant measure with positive Lyapunov exponent. Their existence was first established in the seminal work \cite{BC}, which introduced a new and highly nontrivial construction.
Similar constructions and further developments were later given in \cite{BY, BV, BergerAst, Haz, MV, WY, WY1}, and the coexistence of multiple strange attractors was established in \cite{BP}.

In many constructions, the analysis of chaotic dynamics relies on the identification of a point satisfying the \emph{Collet–Eckmann condition}. Roughly speaking, this condition asserts that the derivative of the map grows exponentially fast along the orbit of this point (see Definition~\ref{def:CE} for a precise formulation).

A natural question then arises: what can be said about the stability of strange attractors? Because of their chaotic nature, this question is particularly challenging. It is known that, in parameter space, there exists a positive Lebesgue measure set of maps exhibiting such attractors. However, the structure of these attractors typically changes as the parameter varies. At present, we lack general tools to determine whether, and in what sense, such chaotic sets can be stable. The present paper provides a first answer to this question for a class of non-hyperbolic chaotic invariant sets. More precisely, we construct codimension-one laminations of maps carrying invariant, Newhouse unstable, wild and non-hyperbolic Cantor sets (Theorem A), show that on a sublamination these sets contain Collet–Eckmann points with dense orbits (Theorem B), and prove that each leaf also contains a map with infinitely many sinks accumulating on the corresponding Cantor set (Theorem C).

We now summarize the main results of the paper.

\paragraph{Theorem A.}
Let $\mathcal M$ be a two-dimensional manifold and let $F:\left([-t_0,t_0]\times [-a_0,a_0]\right)\times\mathcal  M\to\mathcal  M$ be an unfolding with $[-t_0,t_0]\times\left\{0\right\}$ as homoclinic tangency locus of a saddle point.  There exists a codimension-$1$ lamination $\mathfrak{L}$ in parameter space such that, each leaf $\ell\subset\mathfrak{L}$ and each map $f\in\ell$ have a set $A(f)$ satisfying the following:

\smallskip
\emph{(Structure and stability)}
\begin{itemize}
\item The set $A(f)$ in an invariant, non-hyperbolic Cantor set.
\item Every leaf of $\mathfrak L$ is the graph of a smooth function $\ell:[-t_0,t_0]\to [-a_0,a_0]$. Moreover if $t\in [-t_0,t_0]$, the closure of the set $L_t=\left\{(t,a)\in\mathfrak L\right\}$ is a Cantor set. 
\item If $f$ and $\tilde {f}$ are in the same leaf $\ell$ of the lamination $\mathfrak{L}$ then they are conjugate on their respective Cantor sets $A(f)$ and $A(\tilde f)$. 
\item The set $A(f)$ is codimension-$1$ stable and $\ell$ is the smooth manifold of stability. 
\end{itemize}

\smallskip
\emph{(Instability)}
\begin{itemize}
\item Every open neighborhood of $f\in\mathfrak{L}$ contains maps with infinitely many sinks.  
\end{itemize}

\smallskip
\emph{(Dynamical properties)}
\begin{itemize}
\item The set $A(f)$ contains a point with dense orbit.
\item The set $A(f)$ contains a saddle point together with a corresponding transversal homoclinic intersection. 
\item The set $A(f)$ has a unique invariant probability measure. This invariant measure is the point mass on the fixed point in $A(f)$ .
\end{itemize}

Theorem A is the summary of several results proved in the paper, see Lemma \ref{lem:ACantorset}, Corollary \ref{Cor:AHyperbolic}, Lemma \ref{lem:Adense}, Lemma \ref{lem:omegaO2},  Theorem \ref{Theo:uniquemeasure}, Proposition \ref{Prop:longleavesandtransversalCantorset}, Theorem \ref{Th:conjonleaves}, Theorem \ref{Theo:Acod1stable}, Remark \ref{rem:NewhouseInstability}. 

Theorem A establishes the existence of an invariant, non-hyperbolic Cantor set that is codimension-1 stable. In particular, any two maps lying on the same leaf of the lamination are conjugate on their corresponding Cantor sets, so that the dynamical structure is preserved along the leaf. This form of stability is striking in view of the dynamical properties of these sets.

The invariant sets we construct display several types of behaviour that are commonly regarded as \emph{wild}. From a statistical point of view, they are degenerate: although a generic orbit is dense, typical orbits accumulate on a singular invariant measure, namely the Dirac mass at a saddle point (see, for instance, \cite{BKNvS, HK, Lyub}). In addition, they are unstable in the sense of Newhouse, in that arbitrarily small perturbations give rise to infinitely many sinks (see \cite{BDV, Newhouse, PT}). In this sense, Theorem~A reveals a form of stability for dynamics that are otherwise highly non-hyperbolic and unstable.

These results should be compared with those of \cite{CLM}, where the period-doubling Cantor attractor was shown to be codimension-1 stable. In that case, the topology of the Cantor set likewise persists along a codimension-1 leaf; however, the constructed leaf is very short. Numerical experiments suggest that the period-doubling locus should in fact form a long global curve, but establishing this rigorously remains a challenging open problem. By contrast, our approach exploits the unstable expansion at the saddle point to construct a lamination in parameter space whose leaves have uniform size and act as manifolds of stability.

The proof of Theorem~A is based on a renormalization scheme, understood as an iterative procedure using induced maps. Roughly speaking, one studies the dynamics by restricting the map to suitable regions of phase space and considering the first return map to these regions. Iterating this construction describes the dynamics at progressively smaller scales and leads to a nested sequence of regions whose intersection is an invariant Cantor set.

Rather than working directly with the renormalization operator, we focus on the combinatorial structure it generates. At each scale, the dynamics can be encoded by a finite directed graph whose vertices correspond to elements of the dynamical partition and whose edges describe how these elements are mapped onto one another. The inverse limit of these graphs provides a combinatorial model of the invariant Cantor set, allowing effective control of its dynamical properties.

A crucial aspect of our construction is that the combinatorial structure produced by the renormalization scheme persists in parameter space. More precisely, at each scale, the same combinatorics is realized for maps in suitable regions of parameter space, which is essential for the existence of the laminations described in Theorem~A. Achieving this requires quantitative control on how certain distinguished points, notably critical values, move with the parameters. This control constitutes a central difficulty of the analysis and plays a key role in the construction.

Taken together, these features allow us to interpret the dynamics in a structured way across scales. Although the systems we consider are not uniformly hyperbolic, this approach reveals a form of hyperbolic organization at finite scales. Expansion and contraction are encoded combinatorially, even though global hyperbolicity fails. In this sense, hyperbolic dynamics remains a guiding principle in our analysis, not at the global level, but through a hierarchy of finite-scale models.

This viewpoint relies on a key structural feature, namely the use of two branches in the renormalization scheme. To the best of our knowledge, this leads to the first formulation of a generalized renormalization framework with two branches in the setting of two-dimensional systems with positive entropy. Previous renormalization schemes in this context have relied on a single branch (see, for instance, \cite{AKW, BergerZoology, CLM, CEK1, CEK2, CPLY, CPT, EKW1, EKW2, GJ1, GJ2, GJM, GvST, LM1, LM2, Pal}). By contrast, our approach is closer in spirit to the generalized renormalization frameworks developed for one-dimensional systems.

\bigskip

We next show that these invariant Cantor sets also exhibit chaotic behavior, in the sense that they contain a Collet–Eckmann point with dense orbit. Moreover, these sets remain codimension-1 stable, and the corresponding Collet–Eckmann points persist continuously along the leaves of the lamination. This shows that strong chaotic behavior, manifested by exponential growth along a dense orbit, is compatible with a robust topological structure in parameter space. Theorem~B thus represents a first step toward the study of the stability of chaotic non-hyperbolic invariant sets. 


\paragraph{Theorem B.}
The lamination $\mathfrak{L}$  (see Theorem A) satisfies the following property. For each leaf $\ell\subset\mathfrak{L} $ and for each map $f\in\ell$, the invariant non-hyperbolic Cantor set $A(f)$ contains a Collet-Eckmann point $c_f$ satisfying the following properties. 
\begin{itemize}
\item The omega-limit set $\omega(c_f)=A(f)$.
\item  The point $c_f$ depends continuously on $f\in\mathfrak{L}$.
\item  If $f$ and $\tilde {f}$ are in the same leaf $\ell$ of the lamination $\mathfrak{L}$ then there exists a conjugation $h:A(f)\to A(\tilde f)$ such that $h(c_f)=c_{\tilde f}$.
 \end{itemize}

\bigskip

Theorem B is a reformulation of Theorem \ref{Theo:CE}. 

\bigskip

In pursuing the stability of chaotic non-hyperbolic invariant sets, we uncover a striking coexistence phenomenon. We show that each leaf of the lamination contains a map with infinitely many sinks of increasing period, and that these sinks accumulate on the invariant Cantor set containing a Collet–Eckmann point. This coexistence of infinitely many sinks with chaotic dynamics is particularly surprising. The precise statement is as follows.

\paragraph{Theorem C.} 
The lamination $\mathfrak{L}$  (see Theorem A) satisfies the following property.  Each leaf $\ell\subset \mathfrak{L}$ contains a map $f_{\ell}$ with infinitely many attracting periodic points $p_g$, of arbitrarily high  period, such that,
$$
\bigcap_n\overline{\bigcup_{g\geq n}\text{Orb}(p_g)}=A(f_{\ell}).
$$ 
Moreover,  $\overline{\{f_\ell|\ell\subset \mathfrak{L}\}}=\overline{ \mathfrak{L}}.
$ 
\bigskip

Theorem C is a reformulation of Theorem \ref{Theo:CE+N}. 

\bigskip

Theorem C can be viewed in relation to the main result in \cite{PY, MPY}, in which the authors studied unfoldings of heteroclinic cycles and showed that, for most parameters, the dynamics give rise to non-uniformly hyperbolic horseshoes with no periodic attractors. In that setting, non-uniform hyperbolicity appears as a dominant phenomenon, excluding the presence of sinks. In contrast, the mechanisms developed in this paper exhibit a coexistence of non-uniform hyperbolic behavior, including Collet–Eckmann dynamics, with infinitely many sinks accumulating on the same invariant Cantor set. This coexistence suggests a different form of organization of the dynamics beyond the typical parameter regimes.

This naturally raises the question of how abundant such configurations may be. While it is often expected that maps with infinitely many sinks form a very small subset in parameter space, our construction indicates that they arise in a structured and robust way. At present, however, we do not have a clear understanding of their abundance.

One indication comes from the following observation.
For a map with infinitely many sinks, it is natural to expect that each sink arises from a local return map which, after rescaling, is Hénon-like. The creation of further sinks at smaller scales therefore requires that these return maps be parameter unstable, as stability under parameter variations would prevent the appearance of new sinks. Since parameter instability is often associated with dynamical instability of the underlying system, this suggests that the presence of infinitely many sinks may be linked to a form of dynamical instability in the underlying system. This perspective leads to the following conjecture.

\paragraph{Conjecture.} 
Let $f$ be a diffeomorphism of a compact two-dimensional manifold with infinitely many attracting periodic points $p_g$, of arbitrarily high  period. If the set
$$
A=\bigcap_n\overline{\bigcup_{g\geq n}\text{Orb}(p_g)},
$$ 
is transitive then it contains a Collet-Eckmann point with dense orbit. 

\bigskip

The proof of Theorem~C relies on an intermediate result, Theorem~\ref{Newhousepoints}, which establishes the existence of maps with infinitely many sinks in unfoldings. The coexistence of infinitely many attracting periodic orbits is known as the \emph{Newhouse phenomenon}, first discovered in~\cite{Newhouse}. Since then, this phenomenon has been observed in a wide range of settings. In particular, many constructions yield Baire sets of maps with infinitely many sinks in the space of dynamical systems; see, for instance, \cite{ABC, Berger, BDP, Bu, CCH, Da, DR, GST, GST1, KS, Newhouse, PV, Ures, VaS}. These results show that the phenomenon is abundant from a topological point of view, although its detailed dynamical structure often remains poorly understood.

In Newhouse’s original construction, the coexistence of infinitely many sinks is produced by the persistence of homoclinic tangencies, arising from the interaction of stable and unstable Cantor sets with sufficiently large thickness. This mechanism ensures that tangencies persist under perturbations, allowing the repeated creation of sinks.

Our proof of Theorem~C follows a different, more explicit strategy. The argument proceeds inductively: starting from a map with a sink and a homoclinic tangency, we show that the sink persists under small parameter variations, while the tangency gives rise to a new, explicitly constructed tangency, called a secondary tangency, together with an additional sink. This construction is intrinsically related to the renormalization scheme. Iterating the procedure produces infinitely many sinks of increasing period, accumulating on a non-hyperbolic Cantor set containing a Collet–Eckmann point.

The main difference between our approach and Newhouse’s original method~\cite{Newhouse} lies in the treatment of these secondary tangencies. In the classical construction, persistence guarantees the existence of tangencies for nearby maps, but their location depends highly discontinuously on the system, making a detailed analysis difficult. By contrast, our method, inspired by~\cite{BC, BP}, replaces persistence arguments with an explicit construction at each step. This approach provides an alternative route to the Newhouse phenomenon and clarifies how the coexistence of chaotic dynamics and stable periodic behaviour can arise within the same system. Moreover, its explicit nature allows for a precise geometric analysis in both parameter space and phase space. To the best of our knowledge, this is the first time that the asymptotic dynamics of Newhouse sinks accumulating on a chaotic non-hyperbolic set is described in such detail.

    \paragraph{Acknowledgements.}
The authors express their gratitude to Michael Benedicks for his continued support throughout the preparation of this paper and for the many invaluable discussions at both the inspirational and technical levels. The authors are also grateful to Michael Lyubich and Enrique Pujals for helping to clarify aspects of our results that we had previously overlooked, and for their valuable insights that connect the different parts of this work. The authors would also like to thank Dennis Sullivan for his interest and enthusiasm shown for these results at different stages of their development.

The first author was partially supported by the Knut and Alice Wallenberg Foundation. The second author was partially supported by the Starting Grant SFB 67575 and by the VR grant 67578, Renormalization in Dynamics.
     \newpage
\tableofcontents

\section{Preliminaries}\label{section:preliminaries}
 To define the families of maps that will concern us, we first need a suitable linearization theorem. The following well-known result, due to Sternberg (see \cite{S}), provides exactly this.
\begin{theo}\label{Ctlinearization}
Given $\left(\mu, \lambda\right)\in\R^{2}$, there exists $k\in\N$ such that the following holds. 
Let $\mathcal M$ be a two dimensional $\Cinf$ manifold and let $f:\mathcal M\to\mathcal  M$ be a diffeomorphism with saddle point $p\in\mathcal M$ having unstable eigenvalue $|\mu|>1$ and stable eigenvalue $|\lambda|$. If
\begin{equation}\label{nonresonance}
\lambda\neq\mu^{k_1}\text{ and }\mu\neq \lambda^{k_2}
\end{equation}
for $k=(k_1,k_2)\in\N^2$ and $2\leq |k|=k_1+k_2 \leq k$, then $f$ is $\Cq$ linearizable.
\end{theo}
\begin{defin}
Let $\mathcal M$ be an two dimensional $\Cinf$ manifold and let $f:\mathcal M\to\mathcal  M$ be a diffeomorphism with saddle point $p\in\mathcal  M$. We say that $p$ satisfies the $\Cq$ non-resonance condition if (\ref{nonresonance}) holds.
\end{defin}

\begin{theo}\label{familydependence}
Let $\mathcal M$ be a two dimensional $\Cinf$ manifold and let $f:\mathcal M\to\mathcal M$ be a diffeomorphism with a saddle point $p\in\mathcal  M$ which satisfies the $\Cq$ non-resonance condition. Let $0\in \mathcal P\subset\R^{n}$ and let $F:\mathcal P\times \mathcal M\to\mathcal M$ be a $\Cinf$ family with $F_0=f$. Then, there exists a neighborhood $U$ of $p$ and a neighborhood $V$ of $0\in\mathcal P$ such that, for every $t\in V$, $F_t$ has a saddle point $p_t\in U$ satisfying the $\Cq$ non-resonance condition. Moreover $p_t$ is $\Cq$ linearizable in the neighborhood $U$ and the linearization depends $\Cq$ on the parameters. 
\end{theo}
The proofs of Theorem \ref{Ctlinearization} and Theorem \ref{familydependence} can be found in \cite{BrKo, IlaYak}.

\subsection{Unfoldings} \label{familyofunfoldings}
In the sequel we define a map with a strong homoclinic tangency by listing several conditions.  To summarize, a map with a strong homoclinic tangency has a non degenerate homoclinic tangency $q_1$ and a transversal homoclinic intersection $q_2$.  Conditions $(f6)$, $(f7)$ and $(f8)$ ensure that $q_1$ and the unstable local manifolds of $q_1$ and $q_2$ accumulate on the leg of the unstable manifold of the saddle point containing the transversal homoclinic intersection $q_2$. These conditions are redundant if $\mu<-1$. Moreover to implement the construction, a condition on the eigenvalues is also required, see $(f2)$. The basic properties of a map with a strong homoclinic tangency are illustrated in Figure \ref{Fig1}. 
\begin{defin}\label{stronghomtang}
Let $\mathcal M$ be a two dimensional $\Cinf$ manifold and let $f:\mathcal M\to\mathcal  M$ be a local diffeomorphism satisfying the following conditions:
\begin{itemize}
\item[$(f1)$] $f$ has a saddle point $p\in\mathcal M$ with unstable eigenvalue $|\mu|>1$ stable eigenvalue $\lambda$,
\item[$(f2)$] $|\lambda||\mu|^3<1$,
\item[$(f3)$] $p$ satisfies the $\Cq$ non-resonance condition,
\item[$(f4)$] $f$ has a non degenerate homoclinic tangency,  $q_1\in W^u(p)\cap W^s(p)$,
\item[$(f5)$] $f$ has a transversal homoclinic intersection,  $q_2\in W^u(p)\pitchfork W^s(p)$,
\item[$(f6)$] let $[p,q_2]^u\subset W^u(p)$ be the arc connecting $p$ to $q_2$, then there exist arcs  $W^u_{\text{\rm loc},n}(q_2)=[q_2, u_n]^u\subset W^u(q_2)$ such that $[p,q_2]^u\cap [q_2,u_n]^u=\left\{q_2\right\}$ and 
$$
\lim_{n\to\infty}f^n\left(W^u_{\text{\rm loc},n}(q_2)\right)=[p,q_2]^u,
$$
\item[$(f7)$] there exist neighborhoods $W^u_{\text{\rm loc},n}(q_1)\subset W^u(q_1)$ such that 
$$
\lim_{n\to\infty}f^n\left(W^u_{\text{\rm loc},n}(q_1)\right)=[p,q_2]^u,
$$
\item[$(f8)$] there exists $N,M\in\N$ such that 
$$
f^{-N}(q_1)\in [p,q_2]^u
$$
and 
$$
f^{M}(q_2)\in W^s_{\text{\rm loc}}(p).
$$
\end{itemize}
A map $f$ with these properties is called a map with a \emph{strong homoclinic tangency}, see Figure \ref{Fig1}.
\end{defin}

\begin{figure}[h]
\centering
\includegraphics[width=0.5\textwidth]{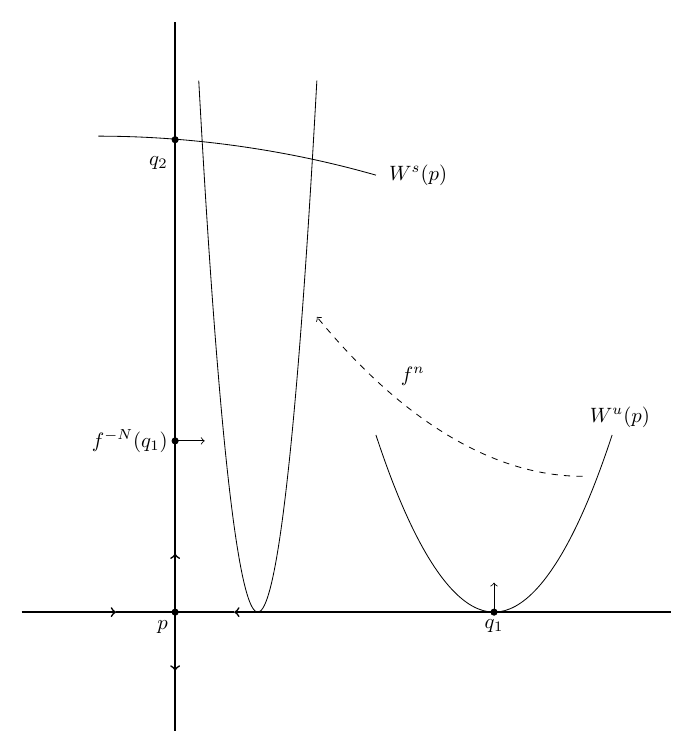}
\caption{A map with a strong homoclinic tangency}
\label{Fig1}
\end{figure}

\begin{rem} The conditions defining a map with a strong homoclinic tangency are natural, except $(f2)$. This condition requires that the contraction at the saddle is strong enough. It plays a crucial role in many fundamental places towards the paper.
\end{rem}

\begin{rem}
Observe that all conditions are open in the space of maps with an homoclinic and transversal tangency.
\end{rem}

\begin{rem} 
If the unstable eigenvalue is negative, $\mu<-1$, then $(f6)$, $(f7)$, and $(f8)$ are redundant.
\end{rem}

Following \cite{PT}, we consider now an unfolding of a map $f$ with a homoclinic tangency and a transversal intersection.
Let  $\mathcal P=[-r,r]\times [-r,r]$ with $r>0$.
Given a map $f$ with a strong homoclinic tangency, we consider a $\Cinf$ family $F:\mathcal P\times\mathcal M\to\mathcal  M$ through $f$ with the following properties:
\begin{itemize}
\item[$(F1)$] $F_{0,0}=f$,
\item[$(F2)$] $F_{t, a}$ has a saddle point $p(t, a)$ with unstable eigenvalue $|\mu(t, a)|>1$, with stable eigenvalue $\lambda(t,a)$, and
$$
\frac{\partial \mu}{\partial t}\ne 0,
$$
\item[$(F3)$] let 
$\mu_{\text{max}}=\max_{(t,a)}|\mu(t,a)|$,
  $\lambda_{\text{max}}=\max_{(t,a)}|\lambda(t,a)|$ and assume 
  \begin{equation}\label{ourcond}
  \lambda_{\text{max}}\mu_{\text{max}}^3<1,
  \end{equation}
\item[$(F4)$] there exists a $\Cd$ function $[-r,r]\ni t\mapsto q_1(t)\in W^u(p(t,0))\cap W^s(p(t,0))$ such that $q_1(t)$ is a non degenerate homoclinic tangency of $F_{t,0}$,
\item[$(F5)$] there exists a $\Cd$ function $[-r,r]^2\ni (t,a)\mapsto q_2(t,a)\in W^u(p(t,a))\cap W^s(p(t,a))$ such that $q_2(t,a)$ is a transversal homoclinic intersection of $F_{t,a}$.
\end{itemize}
\begin{rem}\label{rem:dmudtpositive}
Without loss of generality, we may assume that 
$$
\frac{\partial \mu}{\partial t}>0, \mu(t,a)>0 \text{ and } \lambda(t,a)>0.
$$
\end{rem}
Observe that, by $(F3)$ and by continuity, for $t_0$ small enough, \begin{equation}\label{thetacond0}
0<\frac{\log\mu_{\text{max}}}{\log{\frac{1}{\lambda_{\text{max}}}}}<\frac{3}{2}\frac{\log\mu_{\text{min}}}{\log{\frac{1}{\lambda_{\text{min}}}}}<\frac{1}{2}.
\end{equation}
\subsubsection{Rescaled Families}\label{RescaledFamilies}
According to Theorem \ref{familydependence} we may make a change of coordinates to ensure  that the family $F$ is $\Cq$ and for all $(t,a)\in [-t_0,t_0]\times [-a_0,a_0]$ with $0<t_0,a_0<r$ and by rescaling we can assume that $F_{t,a}$ is linear on the ball $[-2,2]^2$, namely 
\begin{equation}
\label{Flinear}
F_{t,a}=\left(\begin{matrix}
\lambda(t,a)&0\\
0&\mu(t,a)\\
\end{matrix}\right).
\end{equation}

Moreover the saddle point $p(t,a)=(0,0)$ and the local stable and unstable manifolds satisfy:
\begin{itemize}
\item[-] $W^s_{\text{loc}}(0)=[-2,2]\times \left\{0\right\}$,
\item[-] $W^u_{\text{loc}}(0)=\left\{0\right\}\times [-2,2]$ ,
\item[-] $q_1(t)\in W^s_{\text{loc}}(0)$,
\item[-] $q_2(t,a)\in \left\{0\right\}\times \left(\frac{1}{\mu},2\right)\subset W^u_{\text{loc}}(0)$,
\item[-] there exists $N$ such that $f^{N}(q_3(t))=q_1(t)$ where $q_3(t)=(0,1)$.

\end{itemize}
\subsubsection{Critical Points and Critical Values}
In this subsection we aim to introduce the concept of critical points and critical values in our setting.  This will lead to the formal definition of unfoldings. 

In the next lemma we prove that $q_3$ is contained in a curve of points whose vertical expanding tangent vectors are mapped by $DF^{N}$ to horizontal contracting ones. Let $(x,y)$ be in a neighborhood of $q_3$ and consider the point $$(X_{t,a}(x,y),Y_{t,a}(x,y))=F^N_{t,a}(x,y).$$
The following lemma says that $F^N_{t,a}(x,y)$ produces an unfolding in the sense of \cite{PT}. The formal definition of unfoldings is given in Definition \ref{unfolding} below.
\begin{lem}\label{functionc}

There exist $x_0, a_0>0$, a $\Cd$ function $c:[-x_0,x_0]\times [-t_0,t_0]\times  [-a_0,a_0]\to\R $ and a positive constant $Q$ such that 
\begin{equation}\label{partialYpartialt}
\frac{\partial Y_{t,a}}{\partial y}\left(x, c(x,t,a)\right)=0,
\end{equation}
and 
$$\frac{\partial^2 Y_{t,a}}{\partial y^2}\left(x, c(x,t,a)\right)\geq Q.$$
Moreover, 
\begin{eqnarray}\label{gamma}
|c(x,t,a)-c(0,t,a)|=O\left(|x|\right).
\end{eqnarray}
\end{lem}
\begin{proof}
Let $\Phi: [-\frac{1}{2},\frac{3}{2}]\times\left( [-1,1]\times [-t_0,t_0]\times  [-a_0,a_0]\right)\to\R$ be defined by 
$$\Phi(y, x, t,a)=\frac{\partial Y_{t,a}}{\partial y}\left(x, y\right).$$ Observe that $\Phi$ is $\Cd$. Let $q_3(t)=F^{-N}_{t,0}(q_1(t))$. Because $q_1(t)$ is an homoclinic tangency, see $(F4)$, we have $$\Phi (q_3(t),0, t,0)=0,$$ and because $q_1(t)$ is a non degenerate tangency, we get $$\frac{\partial\Phi}{\partial y}(q_3(t),0, t,0)>0.$$ For every $t\in[-t_0,t_0]$ there exist, by the implicit function theorem, $\epsilon >0$ and a unique $\Cd$ function $c:[-\epsilon ,\epsilon]\times [-t_0-\epsilon,t_0+\epsilon]\times  [-\epsilon,\epsilon]\to\R $ locally satisfying the requirements of the lemma. These local functions extend to global ones because of the compactness of the interval $[-t_0,t_0]$ and the local uniqueness.
\end{proof}
\begin{rem}\label{c0=1}
Without lose of generality, by a smooth change of coordinates in the $y$ direction, we may assume that $c(0,t,a)=1$.
\end{rem}

\begin{defin}\label{criticalpoints}
Let $\left(t,a\right)\in [-t_0,t_0]\times  [-a_0,a_0]$. We call the point $$c_{t,a}=\left(0, c(0,t,a)\right),$$ the {primary critical point} and 
 $$z_{t,a}=F_{t,a}^N\left(c_{t,a}\right)=(z_x(t,a),z_y(t,a)),$$ the {primary critical value} of $F_{t,a}$.
\end{defin}
Observe that, near the saddle point, vertical vectors are expanding and horizontal ones are contracting. The critical points are defined to have the property that the expanding vertical vector is sent to the contracting horizontal one under $DF^N$. Let us briefly discuss the concept of critical points for dissipative maps in order to compare our definition with the ones previously used. 

Let us start by recalling that critical points are fundamental in the study of one-dimensional dynamics and they are easy to detect: those are the points where the map is not locally a diffeomorphism. This definition has no meaning for diffeomorphism of higher dimensional manifolds. The formal definition in this setting is given in \cite{PRH}, where the authors define the critical points as homoclinic tangencies, i.e. a tangency between the stable and unstable manifold of a saddle point. One can also, which is the basis of the corresponding construction in \cite{BC}, define a critical point as a tangency between a local stable manifold and an unstable manifold of a periodic point.

Homoclinic tangencies play a crucial topological role in general. However, they are difficult to detect. That is why in varies studies, starting with \cite{BC}, there are notions of approximate critical points which share with homoclinic tangencies the property that expanding vectors are mapped into contracting ones.  In particular, in \cite{BC}, critical points are rather tangencies between the unstable manifold of the saddle point and an approximate local stable manifold, not necessarily of the saddle point.
In our situation, similarly the tangent vector of the unstable manifold at the critical point $c_{t,a}$ is mapped into the contracting horizontal vector at the critical value $z_{t,a}$.

The crucial fact, which comes from the fundamental property of the approximate critical points (and homoclinic tangencies), is that  in a neighborhood, the return map is an H\'enon-like map. The local 
H\'enon behavior is what  is important and allows the analysis. From a technical point of view, the notion of critical point itself is less important; the local H\'enon behavior is all what is needed.

Another instance where critical points occur but play only a secondary role is in the context of strongly dissipative H\'enon maps at the boundary of chaos. These maps have a period doubling Cantor attractor, see \cite{CLM}. The Cantor attractors are studied using renormalization zooming in to the so-called {\it tip} of the Cantor set as in \cite{CLM}. Indeed, return maps to neighborhoods of the tip are 
H\'enon-like maps. A posterior one shows that the stable manifold at the tip is tangent to the direction of the neutral Lyapunov exponent. The tip plays the role of a critical point. However, this fact did not play any role during the renormalization analysis. This is another instance where, from a technical point of view, the notion of critical point is not that important. A similar definition of critical point in a more general setting is given in \cite{CPLY}. 

%
%
%
%
%
%



\begin{defin}\label{unfolding}
A family $F_{t,a}$ is called an {unfoldings} of $f$ if it can be reparametrized such that
\begin{itemize}
\item[$(P1)$] $z_y(t,0)=0$,
\item[$(P2)$] $\frac{\partial z_y(t,0)}{\partial a}\neq 0.$
\end{itemize}
\end{defin}

\begin{figure}[h]
\centering
\includegraphics[width=.9\textwidth]{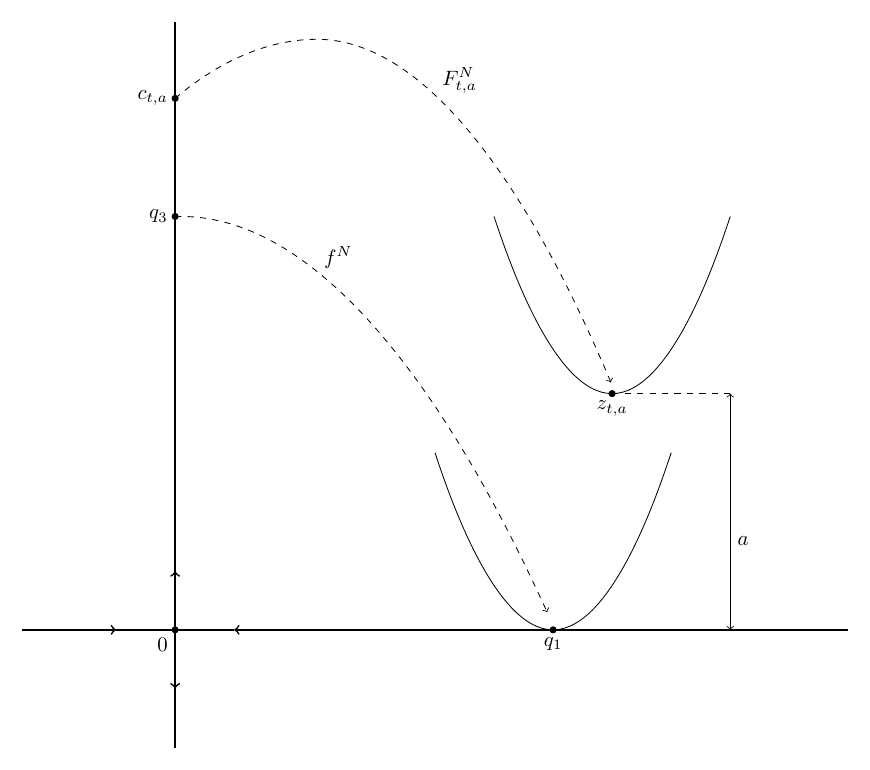}
\caption{Unfolding}
\label{Fig2}
\end{figure}
\begin{rem}\label{highzya}
Without lose of generality by a suitable coordinate change we may assume that if $F$ is an unfolding then $z_y(t, a)=a$, the primary critical value is  at height $a$, $z_x(t, a)=1$ and the primary critical point $c_{t,a}=(0,1)$, see Figure \ref{Fig2}. 
\end{rem}
\begin{rem}
A generic two-dimensional family trough f can locally be reparametrized to become an unfoldings.
\end{rem}
\begin{rem}
The results and the techniques presented in this paper can be generalized to higher dimensional families, both in parameter and phase space.
\end{rem}

\subsection{Properties of unfoldings}

In this section, we outline several fundamental properties of unfoldings.

\paragraph{Estimates of the stable manifold at $q_2$.} 
Consider the transversal homoclinic intersection $q_2$ defined in $(f5)$. Let $W=W^s_{\text{loc}}(q_2)$ and for all $n\in\N$ let $W_n=F^{-n}_{t,a}(W)$. Because $W\pitchfork W^u(p)$ we can apply the $\lambda$-Lemma which implies that $W_n$ converges to $W^s_{\text{loc}}(0)$. In particular $W_n$ is the graph of a function which will also be denoted by $W_n$.  Moreover, because 
\begin{equation}\label{wnformule}
W_n(x)=\mu^{-n}W\left(\lambda^nx\right),
\end{equation}
 then
\begin{equation}\label{wn}
\frac{1}{2\mu}\mu^{-n}\leq |W_n|\leq 2 \mu^{-n},
\end{equation}
\begin{equation}\label{dwn}
\left|\frac{\partial W_n}{\partial x}\right|=O\left(\frac{\lambda}{\mu}\right)^{n},
\end{equation}
and
\begin{equation}\label{ddwn}
\left|\frac{\partial^2 W_n}{\partial x^2}\right|=O\left(\frac{\lambda^2}{\mu}\right)^{n}.
\end{equation}
We estimate now the speed of the stable manifold $W_n$ as the graph of a function  in the phase space. 
\begin{lem}\label{partialwn}
$$\frac{\partial W_n}{\partial t}=-\frac{n}{\mu}\frac{\partial\mu}{\partial t}W_n+\left(\frac{\lambda}{\mu}\right)^n\frac{n}{\lambda}\frac{\partial\lambda}{\partial t}\frac{\partial W}{\partial x}x+\frac{1}{\mu^n}\frac{\partial W}{\partial t}=O\left(\frac{n}{\mu^n}\right),$$
$$\frac{\partial W_n}{\partial a}=-\frac{n}{\mu}\frac{\partial\mu}{\partial a}W_n+\left(\frac{\lambda}{\mu}\right)^n\frac{n}{\lambda}\frac{\partial\lambda}{\partial a}\frac{\partial W}{\partial x}x+\frac{1}{\mu^n}\frac{\partial W}{\partial a}=O\left(\frac{n}{\mu^n}\right).$$
\end{lem}
\begin{proof}
Fix a point $x\in [-2,2]$ and a parameter $(t,a)\in [-r_0,r_0]^2$. We denote by $W_n$ the manifold corresponding to $(t,a)$ and by $W_n+\Delta W_n$ the manifold to $(t,a+\Delta a)$. Then, since the maps $F^n_{t,a}$ are linear, by differentiating (\ref{wnformule}) we obtain
$$
\left(\mu+\frac{\partial\mu}{\partial a}\Delta a\right)^n\left(W_n+\Delta W_n\right)(x)=\mu^n W_n(x)+\frac{\partial W}{\partial x}\left(\left(\left(\lambda+\frac{\partial\lambda}{\partial a}\Delta a\right)^n-\lambda^n\right)x\right)+\frac{\partial W}{\partial a}\Delta a.
$$

Similarly, one gets the same bound for ${\partial W_n}/{\partial t}$.
\end{proof}

\bigskip 
\paragraph{The $\Gamma$ curve and the $a_n$ curve.}
Fix an unfolding $F$ and for each $\left(t, a\right)\in  [-t_0,t_0]\times  [-a_0,a_0]$ let $$\Gamma_{t, a}=\left\{(x,c(x,t, a))| x\in[-x_0,x_0]\right\}.$$ 
In the next lemma we build a curve $a_n$ of points, in the parameter space, whose corresponding critical values are mapped after $n$ steps  into $\Gamma_{t, a}$.

\begin{lem}\label{Imtang}
For $n$ large enough, there exists a $\Cd$ function $a_n:[-t_0,t_0]\to (0,a_0]$ such that $$F^n_{t, a_n\left(t\right)}\left(z_{\left(t, a_n\left(t\right)\right)}\right)\in\Gamma_{\left(t, a_n\left(t\right)\right)}.$$ Moreover
\begin{equation}\label{dandt}
\frac{d a_n}{d t}=-n\frac{\partial\mu}{\partial t}\frac{1}{\mu^{n+1}}\left[1+O\left(\lambda^n\right)\right].
\end{equation}
\end{lem}
\begin{proof}
Let $\Gamma=\text{graph}(c)$, namely, $$\Gamma=\left\{\left(x, c(x,t,a), t,a\right)| \left( x,t,a\right)\in [-x_0,x_0]\times [-t_0,t_0]\times [-a_0,a_0] \right\}.$$ 
Then $\Gamma$ is a $\Cd$ codimension-$1$ manifold transversal to $W^u(0)$. For $n\geq 0$ let
\begin{equation}\label{defgamma}
\Gamma_n=\left\{\left(x,y,t,a\right)| F^n_{t,a}\left( x,y\right)\in\Gamma\right\},
\end{equation}
and the limit of the $\Gamma_n$:s as $n\to\infty$ is, $$\Gamma_{\infty}=\left\{\left(x,0,t,a\right)| \left( x,t,a\right)\in [-x_0,x_0]\times [-t_0,t_0]\times [-a_0,a_0]  \right\}.$$
This follows since, for each $(t,a)\in [-t_0,t_0]\times [-a_0,a_0] $, $F_{t,a}$ is linear, $\Gamma_n$ converges to $\Gamma_{\infty}$ in the $\Cd$ topology. Namely for large $n$, $\Gamma_n$ is a graph of a function also denoted by $\Gamma_n$ and 
$$
\left\|\Gamma_n-\Gamma_{\infty}\right\|_{\Cd}\to 0.
$$
Let $z:[-t_0,t_0]\times [-a_0,a_0]\to\R^2\times [-t_0,t_0]\times [-a_0,a_0]$ be the $\Cd$ critical value function defined as
$$z(t,a)=\left(z_x(t,a), z_y(t,a),t,a\right)=\left(z_x(t,a), a,t,a\right),$$
where $z_{t,a}=\left(z_x(t,a),z_y(t,a)\right)$.
Observe that 

\begin{equation}\label{dz}
\frac{\partial z_y}{\partial a}=1\text{ and }z_y(t,0)=0. 
\end{equation}
Let $Z=\text{Image}(z)$. Because of (\ref{dz}), $Z$ is a manifold transversal to $\Gamma_{\infty}$. Hence, there exists $n_1>0$ such that, for all $n\geq n_1$, $Z$ is transversal to $\Gamma_n$. As consequence, for all $n\geq n_1$,
$$A_n=z^{-1}\left(\Gamma_n\right)$$ is a $\Cd$ codimension-$1$ manifold. We define $a_n: t\mapsto a_n(t)$ as a function whose graph is $A_n$. 
Observe that, by Lemma \ref{functionc} and Remark \ref{c0=1}, the $\Cd$ function $c$ satisfies, $c(0,t,a)=1$. Hence $\partial c/\partial t=O(x)$ and   $\partial c/\partial a=O(x)$. 
Recall that
\begin{equation}\label{munan}
\mu^n a_n=c\left(\lambda^n z_x, t, a_n\right).
\end{equation}
By differentiating (\ref{munan}) and using $\partial c/\partial a, \partial c/\partial t=O(x)$ we have
$$
\mu^n\frac{d a_n}{dt}+na_n\mu^{n-1}\frac{\partial\mu}{\partial t}=O\left(\lambda^n\right).
$$
The lemma follows.
\end{proof}

A proof similar to that of Lemma \ref{wnformule} gives the following.
\begin{lem}\label{partialgamman}
Let $\Gamma_n$ be as in (\ref{defgamma}), then 
$$\frac{\partial\Gamma_n}{\partial t}=-\frac{n}{\mu}\frac{\partial\mu}{\partial t}\Gamma_n+\left(\frac{\lambda}{\mu}\right)^n\frac{n}{\lambda}\frac{\partial\lambda}{\partial t}\frac{\partial \Gamma}{\partial x}x+\frac{1}{\mu^n}\frac{\partial \Gamma}{\partial t}=O\left(\frac{n}{\mu^n}\right),$$
$$\frac{\partial \Gamma_n}{\partial a}=-\frac{n}{\mu}\frac{\partial\mu}{\partial a}\Gamma_n+\left(\frac{\lambda}{\mu}\right)^n\frac{n}{\lambda}\frac{\partial\lambda}{\partial a}\frac{\partial \Gamma}{\partial x}x+\frac{1}{\mu^n}\frac{\partial \Gamma}{\partial a}=O\left(\frac{n}{\mu^n}\right).$$
\end{lem}
\paragraph{Form of the map $F^N$.}
For unfoldings the following holds.
\begin{lem}\label{Nderivatives}
There exist $x'_0<x_0$, $a'_0<a_0$, $b>0$ and $Q>0$, such that, for all $(t,a)\in [-t_0,t_0]\times [-a'_0,a'_0] $ and for every $(x,y)\in\Gamma_{t,a}$ with $|x|<| x'_0|$ the following holds. There exist $\Ct$ functions  $A_{x,y}$, $B_{x,y}\neq 0$ and $C_{x,y}$ such that 
$F^N_{t,a}$ in coordinates centered in $(x,y)$ and $F^N_{t,a}(x,y)$ has the form
\begin{eqnarray}\label{NstepsTaylorformula}
\nonumber
F^N_{t,a}\left(\begin{matrix}
\Delta x\\\Delta y
\end{matrix}\right)&=&\left(\begin{matrix}
A_{x,y}& B_{x,y}\\C_{x,y} & 0
\end{matrix}\right)\left(\begin{matrix}
\Delta x\\\Delta y
\end{matrix}\right)\\&+&\left(\begin{matrix}
O_{1,1}\Delta x^2+O_{1,2}\Delta x\Delta y+O_{1,3}\Delta y^2\\ Q_{x,y}\Delta y^2+O_{2,1}\Delta x^2+O_{2,2}\Delta x\Delta y+O_{2,3} \Delta y^3
\end{matrix}\right),
\end{eqnarray} 
where $Q_{x,y}>Q$, $|B_{x,y}|>b>0$ and the $\Cd$ functions $O_{i,j}$ are uniformly bounded.
\end{lem}
\begin{proof}
The lemma gives the Taylor expansions of $F^N_{t,a}$ when $(x,y)\in\Gamma_{t,a}$. It follows immediately from Lemma \ref{functionc} where we state that a vertical curve trough $(x,y)$ in $\Gamma_{t,a}$ is mapped to a curve with a non degenerate horizontal tangency. In particular $\partial Y_{t,a}/\partial y=0$ by (\ref{partialYpartialt}) and the horizontal tangency, $$DF^N_{t,a}(x,y)\left(\begin{matrix}
0\\1
\end{matrix}\right)=\left(\begin{matrix}
B_{x,y}\\0
\end{matrix}\right)$$ is not degenerate for all $(x,y)\in\Gamma_{t,a}$. This is a consequence of the following argument. Let $t\in[-t_0,t_0]$. Because $F^N_{t,0}(q_3(t))=q_1(t)$ is a non degenerate homoclinic tangency, we know that $B_{q_{3}(t)}\neq 0$ and $Q_{q_3(t)}>0$. By taking $|a|<|a'_0|$, $|x|<|x'_0|$ small enough, the lower bounds on $Q_{x,y}$ and $B_{x,y}$ follow.
\end{proof}

\subsection{Choice of $\theta$}
In this subsection, we introduce a number $\theta$ that is going to play a fundamental role in the construction of our combinatorics. 

The condition $\lambda_{\text{max}}\mu_{\text{max}}^3<1$, see $(F3)$, allows us to choose $\theta\in(0,\frac{1}{2})$ such that 
\begin{equation}\label{thetacond}
1<\lambda_{\text{min}}^{2\theta}\mu_{\text{min}}^3 \text{ and } \lambda_{\text{max}}^{3\theta}\mu^4_{\text{max}}< 1.
\end{equation}
We choose any $\theta$ satisfying 
\begin{equation}\label{thetacond1}
0<\theta_0=\frac{4}{3}\frac{\log\mu_{\text{max}}}{\log{\frac{1}{\lambda_{\text{max}}}}}<\theta<\frac{3}{2}\frac{\log\mu_{\text{min}}}{\log{\frac{1}{\lambda_{\text{min}}}}}=\theta_1<\frac{1}{2},
\end{equation}
where we used $(F3)$, the initial condition $\lambda_{\text{max}}\mu_{\text{max}}^3<1$ and (\ref{thetacond0}).

Moreover, by adjusting the parameter domain, we can choose $\theta$ such that
\begin{equation}\label{eq:thetaconddlambdadmu}
  \left|\frac{\theta}{\lambda}\frac{\partial\lambda}{\partial t}+\frac{1}{\mu}\frac{\partial\mu}{\partial t}\right|>w> 0.
\end{equation}
\section{The H\'enon family} \label{sec:Henonfamily}
Consider the real H\' enon family $F:\mathbb{R}^2\times\mathbb{R}^2\to\mathbb{R}^2$, 
$$
F_{a,b}\left(\begin{matrix}
x\\y
\end{matrix}
\right)=\left(\begin{matrix}
a-x^2-by\\x
\end{matrix}
\right),
$$
a two parameter family. In this section we are going to prove that the H\' enon family is an unfolding of a map with a strong homoclinic tangency and a transversal intersection. 
\begin{prop}
The real H\' enon family is an unfolding. 
\end{prop}
\begin{proof}
Consider the map $f(x)=2-x^2$. Then $x=-2$ is an expanding fixed point and $f^2(0)=-2$ where $0$ is the critical point. For given $b>0$ and $a$ large enough, the H\' enon map $F_{a,b}$ has an horse-shoe. By decreasing $a$ to $a(b)$, one arrives at the first homoclinic tangency. Hence, there exists an analytic curve $b\mapsto a(b)$ with $a(0)=2$ such that the parameter $(a(b),b)$ corresponds to a H\' enon map with an homoclinic tangency of the saddle point $p(b)$  which is a continuation of $p(0)=(-2,-2)$. For all $b$ positive and small enough $F_{a(b),b}$ has a strong homoclinic tangency and the H\' enon family $F_{a,b}$ restricted to a small neighborhood of $(a(b),b)$ is an unfoldings. 
 \end{proof}

\section{Existence of secondary tangencies}
Given an unfolding we create in this section curves corresponding to maps with a so-called secondary tangency. These curves are graphs over the $t$-axes accumulating on the primary tangency curve at $a=0$. Moreover they are fully contained in a well defined strip in parameter space which is defined below. The restriction of the initial unfolding to a neighborhood of any of these curves is again an unfolding of the new secondary tangency.

Remember the choice of $\theta$ in (\ref{thetacond}) and (\ref{thetacond1}).
 Denote by $C$ the entry of the matrix in Lemma \ref{Nderivatives} corresponding to the critical point $c$, see Definition \ref{criticalpoints}. Observe that $C$, $c$, $\lambda$ and $\mu$ all depend on the parameters $(t,a)$. For simplicity of notation we omit this dependence. Let 
 \begin{equation}\label{eq:n0*def}
 n_0^*=\left[\frac{\log\left(\lambda^{2\theta}\mu^3\right)}{2\log \mu}n\right].
 \end{equation}
 
  Choose $\epsilon>0$, $T$ a positive large constant  and define, for large $n$,  
\begin{equation}\label{eq:Bn}
\mathcal{B}_n=\left\{(t,a)\in [-t_0,t_0]\times [-a_0,a_0] \left|\right. \left(\frac{\epsilon}{\mu^{\frac{n_0^*}{2}}}-C\right)\lambda^{\theta n}\leq(a-a_n(t))\leq \frac{T}{\mu^{n+\frac{n_0^*}{2}}}\right\}.
\end{equation}

The strip $\mathcal{B}_n$ is built to contain the curves of secondary tangencies. The exact choice of $n_0$ is made in \eqref{nminusn0}. Denote by $E=C-\frac{\epsilon}{\mu^{{n_0^*}/{2}}}$. 
\begin{rem}\label{muzeroothers}
Observe that, if $(t,a)\in\mathcal B_n$, then
$$
\left(\frac{\mu(t,a)}{\mu(t,a_n(t))}\right)^n=1+O\left(n(a-a_n(t))\right),
$$
and 
$$
\left(\frac{\lambda(t,a)}{\lambda(t,a_n(t))}\right)^n=1+O\left(n(a-a_n(t))\right).
$$
\end{rem}
\subsection{The combinatorics}\label{subsection:combinatorics}
Fix $(t,a)\in\mathcal B_n$ and in the sequel when there is no possibility of confusion, we will suppress this choice in the notation, for example, $z=z_{t,a}$, $c=c_{t,a}$ (see Definition \ref{criticalpoints}), $\lambda=\lambda(t,a)$ and $\mu=\mu(t,a)$. 
\paragraph{The first loop: $N+\theta n+N$ iterates.}\mbox{} \\

By the initial conditions on the family, for $n$ large enough and $W^u_{\text{loc}}(z)$ small enough, $F^{\theta n}_{t,a}\left(W^u_{\text{loc}}(z)\right)$ intersects $\Gamma$ in exactly two points.  Choose one of these points
$$c'\in F^{\theta n}_{t,a}\left(W^u_{\text{loc}}(z)\right)\cap \Gamma,$$ and a local unstable manifold $W^u_{\text{loc}}(c')$ of vertical size $L^{\frac{1}{2}}\lambda^{\frac{\theta n}{2}}$ where $L$ is a large constant that will be chosen later.   The curve $W^u_{\text{\rm loc}}(c')$, is a graph over the $y$-axes. In the following lemma we give an expression for the curve $W^u_{\text{\rm loc}}(c')$.

 \begin{lem}\label{shapeatcprime}
In a coordinate system centered in $c'$ the $\Cq$ curve $W^u_{\text{\rm loc}}(c')$ is given by 
$$
\Delta x=O\left(\left(\frac{\lambda^2}{\mu}\right)^{\frac{\theta n}{2}}\right)\Delta y+O\left(\left(\frac{\lambda^2}{\mu}\right)^{\frac{\theta n}{2}}\right)(\Delta y)^2,
$$
where $0\leq\Delta y<L^{\frac{1}{2}}\lambda^{\frac{\theta n}{2}}$, and $\Delta x=O\left(\left({\lambda^3}/{\mu}\right)^{\frac{\theta n}{2}}\right)$. 
 \end{lem}
\begin{proof}
Use coordinates centered at the critical point $c'$  of the parameter $(t,a)$ and let $(\Delta x_{c'},\Delta y_{c'})\in W^u_{\text{loc}}(c')$.  Observe that $\Delta x_{c'}$ is a function of $\Delta y_{c'}$. 
First we are going to prove that 
\begin{eqnarray}\label{eq:deltaxest1}
\Delta x_{c'}=O\left(\left(\frac{\lambda^2}{\mu}\right)^{\frac{\theta n}{2}}\right),
\end{eqnarray}
 
\begin{eqnarray}\label{slopeespr}
\frac{d\Delta x_{c'}}{d\Delta y_{c'}}=O\left(\left(\frac{\lambda^2}{\mu}\right)^{\frac{\theta n}{2}}\right),
\end{eqnarray}
and

\begin{eqnarray}\label{slopeesprsecondderivative}
\frac{d^2\Delta x_{c'}}{d \left(\Delta y_{c'}\right)^2}=O\left(\left(\frac{\lambda^2}{\mu}\right)^{\frac{\theta n}{2}}\right).
\end{eqnarray}

Because the first coordinate of $z$ is equal to $1$, see Remark \ref{highzya} , we have that 
\begin{equation}\label{eq:yccordinatec'}
c'_y=1+O\left(\lambda^{\theta n}\right) 
\end{equation}
and since $(t,a)\in\mathcal B_n$, then
$$
z_y\leq a_n(t)+\frac{T}{\mu^{n+n_0^*/2}}\leq\frac{1}{\mu^n}\left(1+O\left(\lambda^n\right)\right)+\frac{T}{\mu^{n+n_0^*/2}},
$$
where we used \eqref{munan}. Let $m$ be the highest point of $W^u_{\text{loc}}(c')$. Then 
$$
m_y=1+\left(L\lambda^{\theta n}\right)^{1/2}+O\left(\lambda^{\theta n}\right)=1+O\left(\lambda^{\theta n/2}\right).
$$
Let $Y$ be the vertical distance from $F^{-\theta n}(m)$ and $z$ and $X$ the horizontal one. Then, since $(t,a)\in\mathcal B_n$,
\begin{eqnarray}\label{eq:estthetaover2}\nonumber
Y&=&\frac{1}{\mu^{\theta n}}\left(1+O\left(\lambda^{\theta n/2}\right)\right)-z_y\\
&\leq&\frac{1}{\mu^{\theta n}}\left(1+O\left(\lambda^{\theta n/2}\right)\right)-\left[\frac{1}{\mu^{n}}\left(1+O\left(\lambda^{\theta n}\right)\right)-E\lambda^{\theta n}\right]=O\left(\frac{1}{\mu^{\theta n}}\right),
\end{eqnarray}
and since $W^u_{\text{loc}}(z)$ is essentially a parabola, $X\leq O\left({1}/{\mu^{\theta n/2}}\right)$. From this and from the fact that $z_x=1$, see Remark \ref{highzya}, formula \eqref{eq:deltaxest1} follows. Observe that, the estimate on $X$ also gives that 
\begin{equation}\label{eq:xcoordinatec'}
c'_x=\lambda^{\theta n}+O\left(\left(\frac{\lambda^2}{\mu}\right)^{\frac{\theta n}{2}}\right).
\end{equation}

To prove \eqref{slopeespr}, use coordinates $\left(\Delta x_{z},\Delta y_{z}\right)\in W^u_{\text{\rm loc}}(z)$ centered around $z$ and coordinates $\left(\Delta x_{F^{\theta n}(z)},\Delta y_{F^{\theta n}(z)}\right)\in W^u_{\text{\rm loc}}(c')$ centered around $F^{\theta n}(z)$, i.e. $F^{\theta n}\left(\Delta x_{z},\Delta y_{z}\right)=\left(\Delta x_{F^{\theta n}(z)},\Delta y_{F^{\theta n}(z)}\right)$. 
Let $(u, v)$ be a tangent vector to $W^u_{\text{loc}}(z)$ at the point $\left(\Delta x_{z}, \Delta y_{z}\right)$. Then, using the same estimates as in \eqref{eq:estthetaover2}, there exist uniform constants $K_1,K_2$ such that 
\begin{equation}\label{eq:compdeltaz1}
{K_1^{-1}}{\mu^{-\theta n}}\leq \Delta y_{z}\leq K_1{\mu^{-\theta n}},
\end{equation} 
and because of the quadratic behavior of $W^u_{\text{loc}}(z)$,  $$K_2^{-1}{\mu^{-\frac{\theta n}{2}}}\leq |\Delta x_{z}|\leq K_2{\mu^{-\frac{\theta n}{2}}}.$$ In particular $$K_2^{-1}{\mu^{-\frac{\theta n}{2}}}|u|\leq |v|\leq K_2{\mu^{-\frac{\theta n}{2}}}|u|.$$ Observe that $$\frac{d\Delta x_{c'}}{d\Delta y_{c'}}=O\left(\frac{\lambda^{\theta n}}{\mu^{\theta n}}\frac{u}{v}\right).$$ The estimate for the slope follows. 
This estimate also give a better bound for $\Delta x$, namely $\Delta x=O\left(\left({\lambda^3}/{\mu}\right)^{\frac{\theta n}{2}}\right)$.

We are left to prove \eqref{slopeesprsecondderivative}. Since $W^u_{\text{\rm loc}}(z)$ is the image of a straight line segment, there is a non zero uniformly bounded $\Cd$ function $R$ such that
$$\Delta x_{z}=R\left(\Delta y_{z}\right)\left(\Delta y_{z}\right)^{\frac{1}{2}},$$ and by \eqref{eq:compdeltaz1}, there exists a uniform constant $K_1$, such that
$${K_1^{-1}}{\mu^{-\theta n}}\leq\Delta y_{z}.$$
In particular,
$$
\frac{d^2\Delta x_{z}}{d \left(\Delta y_{z}\right)^2}=O\left(\left(\Delta y_{z}\right)^{-\frac{3}{2}}\right)=O\left(\mu^{\frac{3}{2}\theta n}\right).
$$
By applying the linear map $F^{\theta n}$ we have
$$
\frac{d^2\Delta x_{c'}}{d \left(\Delta y_{c'}\right)^2}=\frac{d^2\Delta x_{F^{\theta n}(z)}}{d \left(\Delta y_{F^{\theta n}(z)}\right)^2}=O\left(\mu^{\frac{3}{2}\theta n}\left(\frac{\lambda}{\mu^2}\right)^{\theta n}\right)=O\left(\left(\frac{\lambda^2}{\mu}\right)^{\frac{\theta n}{2}}\right).
$$
Finally, using \eqref{slopeespr} and \eqref{slopeesprsecondderivative} the curve $W^u_{\text{\rm loc}}(c')$ in coordinates centered around the point $c'$ is essentially a straight line given by
\begin{equation}\label{straightline}
\Delta x_{c'}=O\left(\left(\frac{\lambda^2}{\mu}\right)^{\frac{\theta n}{2}}\right)\Delta y_{c'}+O\left(\left(\frac{\lambda^2}{\mu}\right)^{\frac{\theta n}{2}}\right)(\Delta y_{c'})^2.
 \end{equation}

\end{proof}
Let $z^{(1)}=(z_{x}^{(1)},z_{y}^{(1)})$ be the lowest point of $ F^N\left (W^u_{\text{loc}}(c')\right)$, see Figure \ref{Fig4a}. 
\begin{figure}
\centering
\includegraphics[width=0.9\textwidth]{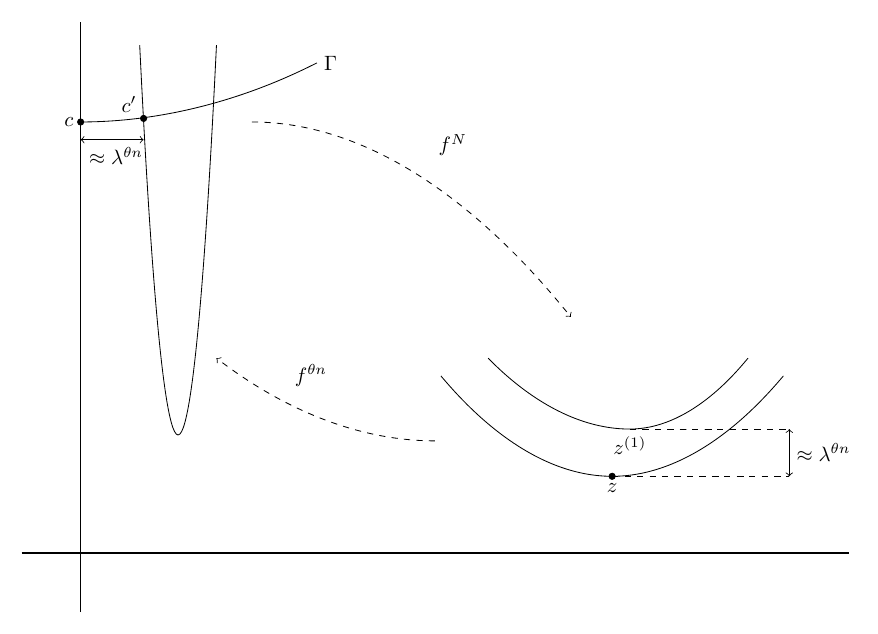}
\caption{$N+\theta n+N$ iterates}
\label{Fig4a}
\end{figure}
\begin{lem}\label{z1distz}
For all $(t,a)\in\mathcal B_n$, 
\begin{eqnarray*}
    z_{y}^{(1)}(t,a)-z_{y}(t,a)= C\lambda^{\theta n}+O\left(\left(\frac{\lambda^2}{\mu}\right)^{\frac{\theta n}{2}}\right).
    \end{eqnarray*}
    In particular, the limit $\lim_{n\to\infty} { |z_{y}^{(1)}(t,a)-z_{y}(t,a)|}/{\lambda^{\theta n}}=C$. Moreover if $(x,y)\in W^u_{\text{loc}}(z^{(1)})$, 
\begin{eqnarray*}
C\lambda^{\theta n}+O\left(\left(\frac{\lambda^2}{\mu}\right)^{\frac{\theta n}{2}}\right)\leq y-z_{y}(t,a)\leq (C+QL)\lambda^{\theta n}+O\left(\left(\frac{\lambda^2}{\mu}\right)^{\frac{\theta n}{2}}\right).
    \end{eqnarray*}

\end{lem}
\begin{proof}
We need to determine the the point $c^{(1)}$ in $ W^u_{\text{loc}}(c')$ which is mapped to $z^{(1)}$.
Use coordinates centered at $c'$  of the parameter $(t,a)$ and denote by $v'=F^N(c')$. Let $(\Delta x_{c'},\Delta y_{c'})\in W^u_{\text{loc}}(c')$ with $(\Delta  x_{v'},\Delta y_{v'})\in F_{t,a}^N\left(W^u_{\text{loc}}(c')\right)$ centered in $v'$. By Lemma \ref{Nderivatives} we have
\begin{eqnarray}\label{eq:fNonycoordinate}
\Delta y_{v'}&=&C_{c'}\Delta x_{c'}+Q_{c'}(\Delta y_{c'})^2+O\left(\Delta x_{c'}^2+\Delta x_{c'} \Delta y_{c'}+\Delta y_{c'}^3\right).
\end{eqnarray}

The point $c^{(1)}$ is given by 
$$
0=\frac{d\Delta y_{v'}}{d\Delta y_{c'}}.
$$
By taking the derivate of \eqref{eq:fNonycoordinate}, we obtain
\begin{eqnarray*}
0=\frac{d\Delta y_{v'}}{d\Delta y_{c'}}&=&C_{c'}\frac{d\Delta x_{c'}}{d\Delta y_{c'}}+2Q_{c'}\Delta y_{c'}+O\left(2\Delta x_{c'}\frac{d\Delta x_{c'}}{d\Delta y_{c'}}+\frac{d\Delta x_{c'}}{d\Delta y_{c'}}\Delta y_{c'}+\Delta x_{c'}+3\Delta y_{c'}^2\right)\\&=&C_{c'}\frac{d\Delta x_{c'}}{d\Delta y_{c'}}\left(1+o(1)\right)+2Q_{c'}\Delta y_{c'}\left(1+o(1)\right)+O\left(\Delta x_{c'}\right),
\end{eqnarray*}
where we used Lemma \ref{shapeatcprime}. 
Using \eqref{eq:deltaxest1}, \eqref{slopeespr} we get 
\begin{equation}\label{eq:orderofdeltayc'}
\Delta y_{c'}=O\left(\left(\frac{\lambda^2}{\mu}\right)^{\frac{\theta n}{2}}\right).
\end{equation}
Use now the point $c^{(1)}$ with coordinates centered at $c$. Let $c^{(1)}=(\Delta x_{c},\Delta y_{c})\in W^u_{\text{loc}}(c')$ with $z^{(1)}=(\Delta  x_{z},\Delta y_{z})\in F_{t,a}^N\left(W^u_{\text{loc}}(c')\right)$ centered in $z$. 
Observe that $\Delta x_{c}=\Delta x_{c'}+c'_x$ and $\Delta y_{c}=\Delta y_{c'}+c'_y-1$. Hence, by \eqref{eq:deltaxest1}, \eqref{eq:xcoordinatec'} and  \eqref{eq:yccordinatec'}, 
\begin{eqnarray}\label{eq:deltaxest}
\Delta x_c=\lambda^{\theta n}+O\left(\left(\frac{\lambda^2}{\mu}\right)^{\frac{\theta n}{2}}\right) \text{ and } \Delta y_c=O\left(\lambda^{{\theta n}}\right).
\end{eqnarray} 
Moreover, by Lemma \ref{Nderivatives} we have
\begin{eqnarray}\label{eq:fNonycoordinate1}
\Delta y_{z}&=&C\Delta x_{c}+Q(\Delta y_{c})^2+O\left(\Delta x_{c}^2+\Delta x_{c} \Delta y_{c}+\Delta y_{c}^3\right).
\end{eqnarray}
Using \eqref{eq:deltaxest} in \eqref{eq:fNonycoordinate1} we have 
$$
  z_{y}^{(1)}(t,a)-z_{y}(t,a)= \Delta y_z=C\lambda^{\theta n}+O\left(\left(\frac{\lambda^2}{\mu}\right)^{\frac{\theta n}{2}}\right).
$$
If $(x,y)\in W^u_{\text{loc}}(z^{(1)})$, using again \eqref{eq:fNonycoordinate1}, \eqref{eq:deltaxest} and the fact that $0\leq |\Delta y_{c}|\leq (L\lambda^{\theta n})^{1/2}+O\left(\lambda^{\theta n}\right)$, we get
\begin{eqnarray*}
C\lambda^{\theta n}+O\left(\left(\frac{\lambda^2}{\mu}\right)^{\frac{\theta n}{2}}\right)\leq y-z_{y}(t,a)\leq (C+QL)\lambda^{\theta n}+O\left(\left(\frac{\lambda^2}{\mu}\right)^{\frac{\theta n}{2}}\right).
    \end{eqnarray*}
Observe that, in the last two inequalities, for the estimate of the order terms we used $(F3)$.
\end{proof}

The curve $W^u_{\text{\rm loc}}(z^{(1)})$ is a graph over the $x$-axes. In the following lemma we study the expression for the curve $W^u_{\text{\rm loc}}(z^{(1)})$.
 \begin{lem}\label{shapeofthecurveatz1}
In a coordinate system centered in $z^{(1)}$ the $\Cq$ curve $W^u_{\text{\rm loc}}(z^{(1)})$ is given by 
$$
\Delta y=\frac{Q}{B^2}\left(1+O\left(\lambda^{\theta n}\right)\right)\Delta x^2+O\left(\Delta x^3\right),
$$
where $0\leq\Delta y<L\lambda^{\theta n}$, and $\Delta x=O\left(\lambda^{\frac{\theta n}{2}}\right)$. 
 \end{lem}
 \begin{proof} 
 Denote by $v'=(v'_x,v'_y)=F^N(c')$ and consider coordinates $(\Delta x_{c'},\Delta y_{c'})\in W^u_{\text{loc}}(c')$ centered at $c'$ with $(\Delta  x_{v'},\Delta y_{v'})\in F_{t,a}^N\left(W^u_{\text{loc}}(c')\right)$ centered in $v'$. Consider also coordinates $(\Delta x_{c^{(1)}},\Delta y_{c^{(1)}})\in W^u_{\text{loc}}(c')$ centered at $c^{(1)}=F^{-N}\left(z^{(1)}\right)$ with $(\Delta  x_{z^{(1)}},\Delta y_{z^{(1)}})\in\left(W^u_{\text{loc}}(z^{(1)})\right)$ centered at $z^{(1)}=(z^{(1)}_x, z^{(1)}_y)$. Denote by  $s=\left({\lambda^2}/{\mu}\right)^{\frac{\theta n}{2}}$ and observe that by Lemma \ref{shapeatcprime}, 
  \begin{eqnarray}\label{eq:deltaxc'anddeltayc'}
 \Delta  x_{c'}&=&O\left(s\right) \Delta  y_{c'}\left[1+ \Delta  y_{c'}\right],
 \end{eqnarray}
 and by construction, 
 \begin{eqnarray}\label{eq:deltaxv'deltaz1}
 \Delta  x_{v'}&=&z^{(1)}_x-v'_x+\Delta x_{z^{(1)}}=\Delta x_{z^{(1)}}+O\left(s\right),\\\label{eq:deltayv'deltaz1}
 \Delta  y_{v'}&=&z^{(1)}_y-v'_y+\Delta y_{z^{(1)}}=\Delta y_{z^{(1)}}+O\left(s^2\right),
 \end{eqnarray}
  where we also used \eqref{eq:orderofdeltayc'}.
 Applyng in order \eqref{eq:deltaxv'deltaz1}, Lemma \ref{Nderivatives} and \eqref{eq:deltaxc'anddeltayc'} we get the following series of equalities,
  \begin{eqnarray*}
\Delta x_{z^{(1)}} &=&\Delta  x_{v'}+O\left(s\right)\\
&=&A_{c'}\Delta  x_{c'}+B_{c'}\Delta  y_{c'}+O\left(\Delta  x_{c'}^2+\Delta  x_{c'}\Delta  y_{c'}+\Delta  y_{c'}^2\right)+O\left(s\right)\\
&=&B_{c'}\Delta  y_{c'}\left[1+O\left(s\right)+O\left(\Delta  y_{c'}\right)\right]+O(s).
 \end{eqnarray*}
 In particular,
 \begin{eqnarray}\label{eq:deltayc'deltaxz1plusorder}
\Delta y_{c'} &=&\frac{\Delta x_{z^{(1)}}+O(s)+O\left(\Delta  y_{c'}^2\right)}{B_{c'}\left(1+O(s)\right)}\\
&=&\frac{\Delta x_{z^{(1)}}}{B_{c'}\left(1+O(s)\right)}+O\left(\lambda^{{\theta n}}\right),
 \end{eqnarray}
where we used that $\Delta  y_{c'}=O\left(\lambda^{\frac{\theta n}{2}}\right)$, see Lemma \ref{shapeatcprime}. Applying in order \eqref{eq:deltayv'deltaz1}, Lemma \ref{Nderivatives}, \eqref{eq:deltaxc'anddeltayc'} and \eqref{eq:deltayc'deltaxz1plusorder} we get 
 \begin{eqnarray*}
\Delta y_{z^{(1)}} &=&\Delta  y_{v'}+O\left(s^2\right)\\
&=&C_{c'}\Delta  x_{c'}+Q_{c'}\Delta  y_{c'}^2+O\left(\Delta  x_{c'}^2+\Delta  x_{c'}\Delta  y_{c'}+\Delta  y_{c'}^3\right)+O\left(s^2\right)\\
&=&C_{c'}O(s)\Delta  y_{c'}\left(1+O\left(\lambda^{\frac{\theta n}{2}}\right)\right)+Q_{c'}\Delta  y_{c'}^2+O\left(s\Delta  y_{c'}^2+\Delta  y_{c'}^3\right)+O\left(s^2\right)\\
&=&\left[O\left(s^2\right)+O\left(s\lambda^{{\theta n}}\right)+O\left(\lambda^{{2\theta n}}\right)\right]+\left[O\left(s\right)+O\left(\lambda^{{\theta n}}\right)\right]\Delta x_{z^{(1)}}\\&+&\left[\frac{Q_{c'}\left(1+O(s)\right)}{B_{c'}^2}+O\left(\lambda^{{\theta n}}\right)\right]\Delta x_{z^{(1)}}^2+\left[O\left(1\right)\right]\Delta x_{z^{(1)}}^3.
 \end{eqnarray*}
 Observe, that this polynomial expression has a special form, the constant and linear term are zero.  This is because $z^(1)$ is the lowest point of the curve $W^u_{\text{loc}}(z^{(1)})$. Hence, 
  \begin{eqnarray*}
\Delta y_{z^{(1)}} &=&\frac{Q_{c'}}{B_{c'}^2}\left(1+O\left(\lambda^{{\theta n}}\right)\right)\Delta x_{z^{(1)}}^2+O\left(\Delta x_{z^{(1)}}^3\right)\\
&=&\frac{Q}{B^2}\left(1+O\left(\lambda^{{\theta n}}\right)\right)\Delta x_{z^{(1)}}^2+O\left(\Delta x_{z^{(1)}}^3\right),
 \end{eqnarray*}
 where we used that, by construction, $c$ and $c'$ are at distance to each other proportional to $\lambda^{{\theta n}}$.
 \end{proof}

\begin{figure}
\centering
\includegraphics[width=0.9\textwidth]{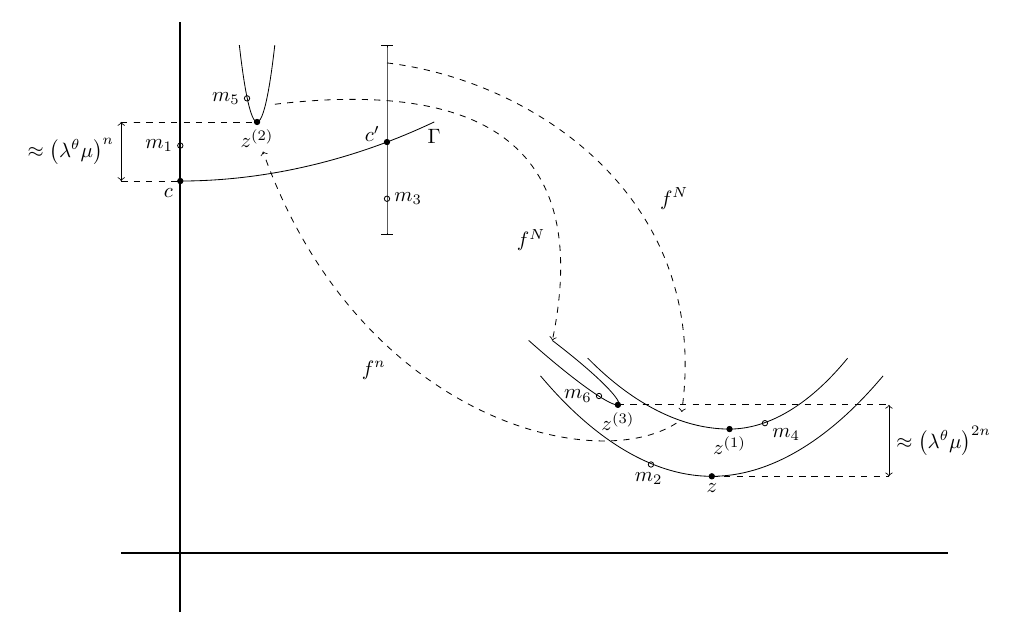}
\caption{$2N+n$ iterates}
\label{Fig4b}
\end{figure}
\paragraph{The second loop: $N+\theta n+N+n+N$ iterates.} \mbox{} \\

Let $(t,a)\in\mathcal B_n$ and let $z^{(2)}=(z_{x}^{(2)},z_{y}^{(2)})=F_{t,a}^n(z^{(1)})$ and take a local unstable manifold $W^u_{\text{loc}}(z^{(2)})$ of vertical size $L\left(\lambda^{\theta}\mu\right)^n$ with $L$ a large constant which is going to be chosen later, see Figure \ref{Fig4b}. The following lemma gives information on the size and the position of $W^u_{\text{\rm loc}}(z^{(2)})$. 

 \begin{lem}\label{ymaxymin}
 For all $(t,a)\in\mathcal B_n$, 
\begin{eqnarray*}
    z_{y}^{(2)}(t,a)-\mu^n_0 z_{n,y}= C\left(\lambda^{\theta}\mu\right)^n +\left(\mu^n+O(n)\right)\left(a-a_n(t)\right)+O\left(\left({\lambda^{\theta}}\mu^{1-\frac{\theta}{2}}\right)^n\right),
    \end{eqnarray*}
    where $z_n=z(t,a_n(t))$ and $\mu_0=\mu(t,a_n(t))$.
  Moreover if $(x,y)\in W^u_{\text{loc}}(z^{(2)})$, 
  \begin{eqnarray*}
& C\left(\lambda^{\theta}\mu\right)^n+\left(\mu^n+O(n)\right)\left(a-a_n(t)\right)+O\left(\left({\lambda^{\theta}}\mu^{1-\frac{\theta}{2}}\right)^n\right)\\&\leq y- \mu_0^n z_{n,y}\leq \\&(C+QL)\left(\lambda^{\theta}\mu\right)^n+\left(\mu^n+O(n)\right)\left(a-a_n(t)\right)+O\left(\left({\lambda^{\theta}}\mu^{1-\frac{\theta}{2}}\right)^n\right),  \end{eqnarray*} 
and $$x=\lambda^n+o\left(\lambda^n\right).$$
 \end{lem}
 \begin{proof}
From Lemma \ref{z1distz}, \eqref{dz}, \eqref{munan}, Remark \ref{muzeroothers} and from the fact that we are in the domain of linearization,
 \begin{eqnarray*}
    z_{y}^{(2)}(t,a)-\mu^n_0 z_{n,y}&=& z_{y}^{(2)}(t,a)-\mu^n z_{y}+\mu^n z_{y}-\mu^n_0 z_{n,y}\\&=&C\left(\lambda^{\theta}\mu\right)^n +O\left(\left({\lambda^{\theta}}\mu^{1-\frac{\theta}{2}}\right)^n\right)+z_{n,y}\left(\mu^n-\mu^n_0\right)+\mu^n(a-a_n(t))\\
    &=&C\left(\lambda^{\theta}\mu\right)^n +O\left(\left({\lambda^{\theta}}\mu^{1-\frac{\theta}{2}}\right)^n\right)+O\left(\left(\frac{\mu}{\mu_0}\right)^n-1\right)+\mu^n(a-a_n(t))\\&=& C\left(\lambda^{\theta}\mu\right)^n +\left(\mu^n+O(n)\right)\left(a-a_n(t)\right)+O\left(\left({\lambda^{\theta}}\mu^{1-\frac{\theta}{2}}\right)^n\right).
    \end{eqnarray*}
Repeat the same estimates using the second inequalities from Lemma \ref{z1distz} to get the second inequalities in this lemma. For the last one, it is enough to use again Lemma \ref{z1distz} and the fact that we are in the linearization domain. 
 \end{proof}
The curve $W^u_{\text{\rm loc}}(z^{(2)})$, is a graph over the $y$-axes. In the following lemma we give an expression for the curve $W^u_{\text{\rm loc}}(z^{(2)})$. \begin{lem}
\label{shape}
In a coordinate system centered in $z^{(2)}$ the $\Cq$ curve $W^u_{\text{\rm loc}}(z^{(2)})$ is given by 
$$
\Delta y=\frac{Q}{B^2}\left(1+O\left(\lambda^{\theta n}\right)\right)\left(\frac{\mu}{\lambda^2}\right)^n\Delta x^2+O\left(\Delta x^3\left(\frac{\mu}{\lambda^3}\right)^n\right),
$$
where $0\leq\Delta y<L\left(\lambda^{\theta}\mu\right)^n$ and $\Delta x=O\left(\lambda^{2+\theta}\right)^{\frac{n}{2}}$. 
 \end{lem}
 \begin{proof}
The lemma follows from the fact that the curve $W^u_{\text{loc}}(z^{(1)})$ is a graph over the $x$-axes and $W^u_{\text{loc}}(z^{(2)})$ is obtained by applying the linear map $F^n_{t,a}$ to the inequality defining the shape of $W^u_{\text{loc}}(z^{(1)})$ in Lemma \ref{shapeofthecurveatz1}.
   \end{proof}

Let $z^{(3)}=(z_{x}^{(3)},z_{y}^{(3)})$ be the lowest point of $F^N_{t,a}\left(W^u_{\text{loc}}(z^{(2)})\right)$. At this moment we have three critical values, $z$, $z^{(1)}$ and $z^{(3)}$. The critical value $z^{(1)}$ is too closely connected to $z$ to use it to create a secondary tangency. In the next lemma we prove that the new critical value $z^{(3)}$ is far enough and independent enough from the critical value $z$ that we can use it to create the secondary tangency.  This independent behavior of $z^{(3)}$ is the key of our constructions. 
 \begin{lem}\label{hnbounds}
  For all $(t,a)\in\mathcal B_n$, 
  \begin{eqnarray*}
z_{y}^{(3)}(t,a)-z_{y}(t,a)= QH^2\left[1+O\left(H\right)\right].
    \end{eqnarray*}
where $H=C\left(\lambda^{\theta}\mu\right)^n +\left(\mu^n+O(n)\right)\left(a-a_n(t)\right)$. In particular,
 \begin{eqnarray*}
  z_{y}^{(3)}(t,a_n(t))-z_{y}(t,a_n(t))= QC^2\left(\lambda^{\theta}\mu\right)^{2n}\left[1+O\left(\left(\lambda^{\theta}\mu\right)^{n}\right)\right].
    \end{eqnarray*}
\end{lem}
\begin{proof}
Use coordinates $(\Delta x_c,\Delta y_c)\in W^u_{\text{loc}}(z^{(2)})$ centered at $c$. We are interested in the special point $m_5=(\Delta x_c, \Delta y_c)\in W^u_{\text{loc}}(z^{(2)})$ such that $F^N(m_5)=z^{(3)}$. In particular, if $F^N(m_5)=(\Delta x_z,\Delta y_z)\in W^u_{\text{loc}}(z^{(3)})$ centered at $z$. Along the proof we will need as well coordinates $(\Delta x_{z^{(2)}},\Delta y_{z^{(2)}})\in W^u_{\text{loc}}(z^{(2)})$ centered at $z^{(2)}$. By construction and by Lemma \ref{shape} we get 
\begin{eqnarray}\label{eqfordeltaxcdeltaxz2}
\Delta x_c=z^{(2)}_x-m_{5,x}+\Delta x_{z^{(2)}}=O\left(\left(\lambda^{2+\theta}\right)^{\frac{n}{2}}\right)+\Delta x_{z^{(2)}},
\end{eqnarray}
and
\begin{eqnarray}\label{eqfordeltaycdeltavz2}
\Delta y_c=z_{y}^{(2)}(t,a)-\mu^n_0 z_{n,y}+\Delta y_{z^{(2)}}=H+\Delta y_{z^{(2)}}.
\end{eqnarray}
Observe that, by Lemma \ref{ymaxymin} and by the definition of the strip $\mathcal B_n$, $H> \left(C-E\right)\left(\lambda^{\theta}\mu\right)^n+O\left(\left({\lambda^{\theta}}\mu^{1-\frac{\theta}{2}}\right)^n\right)>O\left(\lambda^{2+\theta}\right)^{\frac{n}{2}}$.
Use in order, Lemma \ref{Nderivatives}, \eqref{eqfordeltaxcdeltaxz2}, \eqref{eqfordeltaycdeltavz2}, \eqref{thetacond}, \eqref{thetacond1} and the fact that $\Delta x=O\left(\lambda^{2+\theta}\right)^{\frac{n}{2}}$ (see Lemma \ref{shape}) to get
\begin{eqnarray}\label{eq:deltayzindeltaz2}\nonumber
\Delta y_z&=&C\Delta x_c+Q\Delta y_c^2+O\left(\Delta x_c^2+ \Delta x_c\Delta y_c+\Delta y_c^3\right)\\
&=&C\Delta x_{z^{(2)}}+Q\Delta y_{z^{(2)}}^2+2QH\Delta y_{z^{(2)}}+QH^2+O\left(\lambda^{2+\theta}\right)^{\frac{n}{2}}\\\nonumber&+&O\left(\Delta x_{z^{(2)}}^2+\Delta x_{z^{(2)}}H+\Delta x_{z^{(2)}}\Delta y_{z^{(2)}}+\Delta y_{z^{(2)}}^3+\Delta y_{z^{(2)}}^2H+\Delta y_{z^{(2)}}H^2+H^3\right)\\\label{eq:deltayzindeltaz21}&=&C\Delta x_{z^{(2)}}\left[1+O(H)\right]+Q\Delta y_{z^{(2)}}^2\left[1+O(H)\right]+2QH\Delta y_{z^{(2)}}\left[1+O(H)\right]\\\nonumber &+&QH^2\left[1+O(H)\right].
\end{eqnarray}
Observe that the special point $m_5$ we are looking for will satisfy the equation $d\Delta y_z/d\Delta x_{z^{(2)}}=0$. Hence, by differentiating \eqref{eq:deltayzindeltaz2},
\begin{eqnarray}\label{eq:ddeltatildeyddeltax}
0=\frac{d\Delta y_z}{d\Delta x_{z^{(2)}}}&=&C+2Q\Delta y_{z^{(2)}}\frac{d\Delta y_{z^{(2)}}}{d\Delta x_{z^{(2)}}}+2QH\frac{d\Delta y_{z^{(2)}}}{d\Delta x_{z^{(2)}}}\\\nonumber&+&O\left(\Delta x_{z^{(2)}}+ \Delta y_{z^{(2)}}+\Delta x_{z^{(2)}}\frac{d\Delta y_{z^{(2)}}}{\Delta x_{z^{(2)}}}+\Delta y_{z^{(2)}}^2\frac{d\Delta y_{z^{(2)}}}{d\Delta x_{z^{(2)}}}\right.\\\nonumber&+&\left.\Delta y_{z^{(2)}}\frac{d\Delta y_{z^{(2)}}}{d\Delta x_{z^{(2)}}}H+\frac{d\Delta y_{z^{(2)}}}{d\Delta x_{z^{(2)}}}H^2+H\right).
\end{eqnarray}

By Lemma \ref{shape}, 
\begin{eqnarray}\label{eqfordeltay}\nonumber
\Delta y_{z^{(2)}}&=&\frac{Q}{B^2}\left(1+O\left(\lambda^{\theta n}\right)\right)\left(\frac{\mu}{\lambda^2}\right)^n\Delta x_{z^{(2)}}^2+O\left(\Delta x_{z^{(2)}}^3\left(\frac{\mu}{\lambda^3}\right)^n\right)\\&=&\frac{Q}{B^2}\left(\frac{\mu}{\lambda^2}\right)^n\Delta x_{z^{(2)}}^2\left[1+O\left(\lambda^{\frac{\theta n}{2}}\right)\right]
\end{eqnarray}
and 
\begin{eqnarray*}\label{eqfordeltayderivative}\nonumber
\frac{d\Delta y_{z^{(2)}}}{d\Delta x_{z^{(2)}}}&=&2\frac{Q}{B^2}\left(1+O\left(\lambda^{\theta n}\right)\right)\left(\frac{\mu}{\lambda^2}\right)^n\Delta x_{z^{(2)}}+O\left(\Delta x_{z^{(2)}}^2\left(\frac{\mu}{\lambda^3}\right)^n\right)\\&=&2\frac{Q}{B^2}\left(\frac{\mu}{\lambda^2}\right)^n\Delta x_{z^{(2)}}\left[1+O\left(\lambda^{\frac{\theta n}{2}}\right)\right].
\end{eqnarray*}
Using the last two inequalities in \eqref{eq:ddeltatildeyddeltax} we get
\begin{eqnarray}\label{eq:ddeltatildeyddeltax1}\nonumber
0=\frac{d\Delta y_z}{d\Delta x_{z^{(2)}}}&=&2\frac{Q^3}{B^4}\left(\frac{\mu}{\lambda^2}\right)^{2n}\Delta x_{z^{(2)}}^3\left[1+O\left(H\right)\right]+4\frac{Q^2}{B^2}H\left(\frac{\mu}{\lambda^2}\right)^{n}\Delta x_{z^{(2)}}\left[1+O\left(H\right)\right]\\&+&C\left[1+O(H)\right]
\end{eqnarray}
where we also use Lemma \ref{shape} and \eqref{thetacond}. By \eqref{eq:ddeltatildeyddeltax1} it follows that 
\begin{equation}\label{eq:deltaz2inorder}
\Delta x_{z^{(2)}} =O\left(\left(\frac{\lambda^{2-\theta}}{\mu^2}\right)^n\right)
\end{equation}
and by \eqref{eqfordeltay},
\begin{equation}\label{eq:deltay2inorder}
\Delta y_{z^{(2)}}=O\left(\left(\frac{\lambda^{2-2\theta}}{\mu^3}\right)^n\right).
\end{equation}
Use the last two inequalities in \eqref{eq:deltayzindeltaz21} to get
$$
\Delta y_z=QH^2\left[1+O(H)\right],
$$
where we also used the lower bound for $H$, \eqref{thetacond} and \eqref{thetacond1}. The lemma is proved.
Moreover by using \eqref{eq:deltaz2inorder} in \eqref{eq:ddeltatildeyddeltax1} we get an exact estimate for $\Delta x_{z^{(2)}}$. Namely,
\begin{eqnarray*}
0&=&4\frac{Q^2}{B^2}H\left(\frac{\mu}{\lambda^2}\right)^{n}\Delta x_{z^{(2)}}\left[1+O\left(\left(\frac{\mu}{\lambda^2}\right)^{n}\frac{\Delta x_{z^{(2)}}^2}{H}\right)+O(H)\right]+C\left[1+O(H)\right]
\\&=&4\frac{Q^2}{B^2}H\left(\frac{\mu}{\lambda^2}\right)^{2n}\Delta x_{z^{(2)}}\left[1+O(H)\right]+C\left[1+O(H)\right]
\end{eqnarray*}
and 
\begin{equation}\label{eq:deltaxz2exact}
\Delta x_{z^{(2)}}=-\frac{CB^2}{4Q^2H}\left(\frac{\lambda^2}{\mu}\right)^{n}\left[1+O(H)\right] .
\end{equation}
\end{proof}

The curve $W^u_{\text{\rm loc}}(z^{(3)})$, is a graph over the $y$-axes. In the following lemma we give an expression for the curve $W^u_{\text{\rm loc}}(z^{(3)})$. \begin{lem}
\label{curvaturez3} 
In a coordinate system centered in $z^{(3)}$ the $\Cq$ curve $W^u_{\text{\rm loc}}(z^{(3)})$ is given by 
$$
\Delta y=\frac{8Q^4}{C^2B^4}{H^3}\left(\frac{\mu}{\lambda^2}\right)^n\Delta x^2+O\left(\Delta x^3\right),
$$
where $H=C\left(\lambda^{\theta}\mu\right)^n +\left(\mu^n+O(n)\right)\left(a-a_n(t)\right)$. 
 \end{lem}
\begin{proof}
In order to be able to use Lemma \ref{Nderivatives}, use coordinates $(\Delta x_c,\Delta y_c)\in W^u_{\text{loc}}(z^{(2)})$ centered at $c$ and coordinates $(\Delta x_z,\Delta y_z)\in W^u_{\text{loc}}(z^{(3)})$ centered at $z$. 
By Lemma \ref{Nderivatives}, we have
\begin{eqnarray}\label{eq:deltaxzafterFN}
\Delta x_z&=&F^N\left(\Delta x_c\right)=A\Delta x_c+B\Delta y_c+O\left(\Delta x_c^2+\Delta x_c\Delta y_c+\Delta y_c^2\right)\\\label{eq:deltayzafterFN}
\Delta y_z&=&F^N\left(\Delta y_c\right)=C\Delta x_c+Q\Delta y_c^2+O\left(\Delta x_c^2+\Delta x_c\Delta y_c+\Delta y_c^3\right).
\end{eqnarray}
For the aim of proving the lemma we are interested in coordinates $(\Delta x_{z^{(2)}},\Delta y_{z^{(2)}})\in W^u_{\text{loc}}(z^{(2)})$ centered at $z^{(2)}$ and coordinates $(\Delta x_{z^{(3)}},\Delta y_{z^{(3)}})\in W^u_{\text{loc}}(z^{(3)})$ centered at $z^{(3)}$. By construction, by Lemma \ref{ymaxymin}, Lemma \ref{hnbounds} and Lemma \ref{shape} we have,
\begin{eqnarray*}
\Delta x_z&=&\Delta x_{z^{(3)}}+1-{z_x^{(3)}}=\Delta x_{z^{(3)}}+O\left(H\right),\\
\Delta y_z&=&\Delta y_{z^{(3)}}+{z_y^{(3)}}-z_y=\Delta x_{z^{(3)}}+O\left(H^2\right),\\
\Delta x_c&=&\Delta x_{z^{(2)}}+O\left(\lambda^{n}\right),\\
\Delta y_c&=&\Delta y_{z^{(2)}}+H= \frac{Q}{B^2}\left(1+O\left(\lambda^{\theta n}\right)\right)\left(\frac{\mu}{\lambda^2}\right)^n\Delta x_{z^{(2)}}^2+H+O\left(\left(\lambda^{\frac{3}{2}\theta}\mu\right)^n\right).
\end{eqnarray*}
Using the equalities above in \eqref{eq:deltaxzafterFN}, \eqref{eq:deltayzafterFN} and after several order cancellations, we get
\begin{eqnarray}\label{eq:deltaxz3indeltaxz2}
\Delta x_{z^{(3)}}&=&H\left[B+O(1)\right]+ A\Delta x_{z^{(2)}}\left[1+O(H)\right]+\frac{Q}{B}\left(\frac{\mu}{\lambda^2}\right)^n\Delta x_{z^{(2)}}^2\left[1+O(H)\right]\\&+&O\left(\left(\frac{\mu}{\lambda^2}\right)^n\Delta x_{z^{(2)}}^3+\left(\frac{\mu^2}{\lambda^4}\right)^n\Delta x_{z^{(2)}}^4\right)=f\left(\Delta x_{z^{(2)}}\right),\\\label{eq:deltayz3indeltaxz2}
\Delta y_{z^{(3)}}&=&H^2\left[Q+O(1)\right]+ C\Delta x_{z^{(2)}}\left[1+O(H)\right]+2H\frac{Q^2}{B^2}\left(\frac{\mu}{\lambda^2}\right)^n\Delta x_{z^{(2)}}^2\left[1+O(H)\right]\\&+&O\left(\left(\frac{\mu}{\lambda^2}\right)^n\Delta x_{z^{(2)}}^3+\left(\frac{\mu^2}{\lambda^4}\right)^nH\Delta x_{z^{(2)}}^4+\left(\frac{\mu^3}{\lambda^6}\right)^n\Delta x_{z^{(2)}}^6\right)=g\left(\Delta x_{z^{(2)}}\right).
\end{eqnarray}
By Lemma \ref{shape}, 
\begin{equation}\label{eq:exactfromulafordeltayz3indeltaxz3}
\Delta y_{z^{(3)}}=\frac{C_2}{2}\Delta x_{z^{(3)}}^2+O\left(\Delta x_{z^{(3)}}^3\right)
\end{equation}
with $C_2\neq 0$. We need then to calculate the coefficient $C_2$.
Observe that 
\begin{eqnarray}\label{eq:C2formula}
\nonumber C_2=\frac{d^2\Delta y_{z^{(3)}}}{d\Delta x_{z^{(3)}}^2}&=&g''\left(\Delta x_{z^{(2)}}\right)\left(\frac{d\Delta x_{z^{(2)}}}{d\Delta x_{z^{(3)}}}\right)^2+g'\left(\Delta x_{z^{(2)}}\right)\frac{d^2\Delta x_{z^{(2)}}}{d\Delta x_{z^{(3)}}^2}\\
&=&g''\left(\Delta x_{z^{(2)}}\right)\left(\frac{d\Delta x_{z^{(2)}}}{d\Delta x_{z^{(3)}}}\right)^2,
\end{eqnarray}
where we used that $z^{(3)}$ is the lowest point of the $W^u_{\text{\rm loc}}(z^{(3)})$ and hence $g'\left(\Delta x_{z^{(2)}}\right)=0$.
By \eqref{eq:deltayz3indeltaxz2},
\begin{eqnarray*}
g'=\frac{d\Delta y_{z^{(3)}}}{d\Delta x_{z^{(2)}}}&=&C\left[1+O(H)\right]+4H\frac{Q^2}{B^2}\left(\frac{\mu}{\lambda^2}\right)^n\Delta x_{z^{(2)}}\left[1+O(H)\right]\\&+&O\left(\left(\frac{\mu}{\lambda^2}\right)^n\Delta x_{z^{(2)}}^2+\left(\frac{\mu^2}{\lambda^4}\right)^nH\Delta x_{z^{(2)}}^3+\left(\frac{\mu^3}{\lambda^6}\right)^n\Delta x_{z^{(2)}}^5\right),
\end{eqnarray*}
and 
\begin{eqnarray*}
g''=\frac{d^2\Delta y_{z^{(3)}}}{d\Delta x_{z^{(2)}}^2}&=&4H\frac{Q^2}{B^2}\left(\frac{\mu}{\lambda^2}\right)^n\left[1+O(H)\right]\\&+&O\left(\left(\frac{\mu}{\lambda^2}\right)^n\Delta x_{z^{(2)}}+\left(\frac{\mu^2}{\lambda^4}\right)^nH\Delta x_{z^{(2)}}^2+\left(\frac{\mu^3}{\lambda^6}\right)^n\Delta x_{z^{(2)}}^4\right).
\end{eqnarray*}
By the previous inequality and by \eqref{eq:deltaz2inorder} after cleaning the order terms we get 
\begin{eqnarray}\label{eq:gsecondderivativeindetlaxz2}
g''\left(\Delta x_{z^{(2)}}\right)=4H\frac{Q^2}{B^2}\left(\frac{\mu}{\lambda^2}\right)^n\left[1+O(H)\right].
\end{eqnarray}
Observe that 
\begin{equation}\label{eq:formulaforddeltaxz2indeltaxz3}
\frac{d\Delta x_{z^{(2)}}}{d\Delta x_{z^{(3)}}}=\frac{1}{f'\left(\Delta x_{z^{(2)}}\right)}
\end{equation}
and by \eqref{eq:deltaxz3indeltaxz2}, 
\begin{eqnarray*}
f'=\frac{d\Delta x_{z^{(3)}}}{d\Delta x_{z^{(2)}}}&=&A\left[1+O(H)\right]+2\frac{Q}{B}\left(\frac{\mu}{\lambda^2}\right)^n\Delta x_{z^{(2)}}\left[1+O(H)\right]\\&+&O\left(\left(\frac{\mu}{\lambda^2}\right)^n\Delta x_{z^{(2)}}^2+\left(\frac{\mu^2}{\lambda^4}\right)^n\Delta x_{z^{(2)}}^3\right).
\end{eqnarray*}
In particular, by \eqref{eq:deltaxz2exact}
\begin{eqnarray*}
f'\left(\Delta x_{z^{(2)}}\right)=-\frac{CB}{2Q}\frac{1}{H}\left[1+O(H)\right],
\end{eqnarray*}
and by \eqref{eq:formulaforddeltaxz2indeltaxz3}
\begin{equation}\label{eq:formulaforddeltaxz2indeltaxz3final}
\frac{d\Delta x_{z^{(2)}}}{d\Delta x_{z^{(3)}}}=-\frac{2Q}{CB}{H}\left[1+O(H)\right].
\end{equation}
Use \eqref{eq:gsecondderivativeindetlaxz2} and \eqref{eq:formulaforddeltaxz2indeltaxz3} in \eqref{eq:C2formula} to get 
\begin{eqnarray*}
C_2&=&\frac{16Q^4}{C^2B^4}{H^3}\left(\frac{\mu}{\lambda^2}\right)^n\left[1+O(H)\right],
\end{eqnarray*}
and by \eqref{eq:exactfromulafordeltayz3indeltaxz3},
\begin{equation*}
\Delta y_{z^{(3)}}=\frac{8Q^4}{C^2B^4}{H^3}\left(\frac{\mu}{\lambda^2}\right)^n\Delta x_{z^{(3)}}^2+O\left(\Delta x_{z^{(3)}}^3\right).
\end{equation*}
\end{proof}

\paragraph{The full cycle: $N+\theta n+N+n+N+n_0$ iterates.} 
For all $(t,a)\in\mathcal B_n$, let $n_0$ be the maximal $n\in\N$ such that 
\begin{equation}\label{nminusn0}
z^{(3)}\subset\left\{(x,y)\in [-2,2]^2| y\geq W_{n-n_0}(x)\right\}.
\end{equation}
Observe that the funtion $(t,a)\mapsto n_0$ gives information on the position of $W^u_{\text{loc}}(z^{(3)}(t,a_n(t)))$ with respect to the pull-backs of the stable manifold, see Figure \ref{Fig5}. A more precise estimate on the size of $n_0$ comes from the following lemma.
Let 
\begin{equation}\label{alphaofuse1}
\alpha=\frac{\log\left(\lambda^{2\theta}\mu^3\right)}{\log \mu}
\end{equation}
 and recall the definition of $n_0^*$ in \eqref{eq:n0*def}. By (\ref{thetacond1}), $\alpha\in \left(0,{1}\right)$. 
\begin{lem}\label{n0} For every $(t,a_n(t))\in\mathcal B_n$, the integer $n_0$ satisfies the following:
$$
n_0=[n\alpha]+O(1)>n_0^*.
$$
In particular there exists a uniform positive constant $K$ such tha
\begin{equation}\label{alphaofuse}
\frac{1}{K}\left(\lambda^{2\theta}\mu^3\right)^{n}\leq \mu^{n_0}\leq K \left(\lambda^{2\theta}\mu^3\right)^{n}.
\end{equation}
\end{lem} 
\begin{proof}
Let  $(t,a_n(t))\in\mathcal B_n$. From Lemma \ref{hnbounds}, (\ref{wn}) and the fact that $W_{n+1}(1)< z_y<W_{n}(1)$, there exists a uniform constant $K_1>0$ such that 
$$
 \frac{1}{K_1}\left(\frac{1}{\mu^n}+\left(\lambda^{\theta}\mu\right)^{2n}\right)\leq z_{y}^{(3)}\leq K_1\left(\frac{1}{\mu^n}+\left(\lambda^{\theta}\mu\right)^{2n}\right).
$$
Moreover the definition of $n_0$ implies the existence of a uniform constant $K_2>0$ such that
$$
\frac{1}{K_2}\frac{1}{\mu^{n-n_0}}\leq  z_{y}^{(3)}\leq {K_2}\frac{1}{\mu^{n-n_0}}.
$$
The two previous inequalities imply the existence of a positive constant $K$ such that
\begin{equation*}
\frac{1}{K}\leq\frac{1}{\mu^{n_0}}\left(1+\left(\lambda^{2\theta}\mu^3\right)^{n}\right)\leq K.
\end{equation*}
The lemma follows from (\ref{thetacond}).
\end{proof}
 \begin{rem}\label{rem:thetalpha}
 Observe that by combining \eqref{alphaofuse1} and \eqref{thetacond}, one gets
 \begin{equation}\label{eq:bounderyrelation}
    \lambda^{2\theta}\mu^2<\frac{1}{\mu^{\alpha}}. 
 \end{equation}
\end{rem}
\subsection{The secondary tangencies}
In this subsection we locate secondary tangencies along the curve $a_n$. We start by giving the formal definition of secondary tangency.
\begin{defin}
A tangency between $W^u_{\text{loc}}(z^{(3)})$ and $ W_{n-n_0}$ is called a \emph{secondary tangency of type} $n_0$. 
\end{defin}

In order to prove the existence of secondary tangencies, we need to estimate the variation of $z_{y}^{(3)}$ with respect to the parameters $a$ and $t$. This gives the speed of $W^u_{\text{\rm loc}}(z^{(3)})$ in the phase space. Comparing this speed with the one of the pullbacks of the stable manifold at $q_2$ allows finding the new tangencies. 

We take a point $m_1\in W^u_{\text{\rm loc}}\left(c\right)$ and consider its orbit at crucial moments, step by step, until we get to the point $F^{3N+\theta n+n}(m_1)\in W^u_{\text{\rm loc}}(z^{(3)})$.  In particular we are interested in the derivative of these points with respect to $m_1, t$ and $a$. The crucial moments are the ones defined by the combinatorics as presented in the previous section. The reader can refer to Figure \ref{Fig4b}. 

The reader might be surprised by the smaller order terms we need to keep track of in Proposition \ref{speed}. Unfortunately the leading terms will cancel and the speed of $W^u_{\text{\rm loc}}(z^{(3)})$ when varying the parameters along the curve $a_n$ will be determined by smaller order terms, see (\ref{upperineq}). Smaller order terms also turned out to be crucial in the proof of Proposition \ref{angle}.


\begin{prop}\label{speed}
For $n\ge 1$ large enough and $(t,a)\in\mathcal B_n$,
\begin{eqnarray*}
&DS_{C-E} \left\{  \frac{n}{\mu}\frac{\partial\mu}{\partial t}+  n \left(\lambda^{\theta}\mu\right)^{n}\left[C\frac{\theta }{\lambda} \frac{\partial \lambda}{\partial t}+\frac{(C-E) }{\mu}\frac{\partial\mu}{\partial t}\right] +O\left(\left(\lambda^{\theta}\mu\right)^{n}\right)\right\}\\
&\leq\frac{\partial z_{y}^{(3)}}{\partial t}\leq\\& \frac{DT}{\mu^{n_0^*/2}}\left\{\frac{n}{\mu}\frac{\partial\mu}{\partial t}+\frac{n}{\mu}\frac{\partial\mu}{\partial t}\frac{T}{\mu^{n_0^*/2}}+O\left(\frac{1}{\mu^{n_0^*/2}}\right)\right\},
\end{eqnarray*}
and 
\begin{eqnarray*}
&DS_{C-E} \left\{ \mu^n+\frac{1}{DS_{C_E}}+ \frac{n}{\mu}\frac{\partial\mu}{\partial a}+  n \left(\lambda^{\theta}\mu\right)^{n}\left[C\frac{\theta }{\lambda} \frac{\partial \lambda}{\partial a}+\frac{(C-E) }{\mu}\frac{\partial\mu}{\partial a}\right] +O\left(\left(\lambda^{\theta}\mu\right)^{n}\right)\right\}\\&\leq\frac{\partial z_{y}^{(3)}}{\partial a}\leq\\&
 DT\mu^{n-n_0^*/2}+1+\frac{DT}{\mu^{n_0^*/2}}\left\{\frac{n}{\mu}\frac{\partial\mu}{\partial a}+\frac{n}{\mu}\frac{\partial\mu}{\partial a}\frac{T}{\mu^{n_0^*/2}} +O\left(\frac{1}{\mu^{n_0^*/2}}\right)\right\},
\end{eqnarray*}
where $S_{C-E}=(C-E)\left(\lambda^{\theta}\mu\right)^n+O\left(\left(\lambda^{\theta}\mu^{1-\frac{\theta}{2}}\right)^n\right)$.
Moreover $ C, C-E,D$  are positive smooth functions in $t,a$ and depending also on $n$. These functions converge to a non zero function in $t$ as $n\to\infty$. The constant $T$ is the same as in the definition of $\mathcal B_n$. In particular, when $a=a_n(t)$,
\begin{eqnarray*}
\frac{\partial z_{y}^{(3)}}{\partial t}=
D S_C  \left\{\frac{n}{\mu}\frac{\partial\mu}{\partial t}+ C n \left(\lambda^{\theta}\mu\right)^{n}\left[\frac{\theta }{\lambda} \frac{\partial \lambda}{\partial t}+\frac{1 }{\mu}\frac{\partial\mu}{\partial t}\right] +O\left(\left(\lambda^{\theta}\mu\right)^{n}\right) \right\},
\end{eqnarray*}
and 
\begin{eqnarray*}
\frac{\partial z_{y}^{(3)}}{\partial a}&=&
 DS_{C} \left\{ \mu^n+\frac{1}{DS_{C}}+ \frac{n}{\mu}\frac{\partial\mu}{\partial a}+  Cn \left(\lambda^{\theta}\mu\right)^{n}\left[\frac{\theta }{\lambda} \frac{\partial \lambda}{\partial a}+\frac{1 }{\mu}\frac{\partial\mu}{\partial a}\right] +O\left(\left(\lambda^{\theta}\mu\right)^{n}\right)\right\},
\end{eqnarray*}
 where $S_{C}=C\left(\lambda^{\theta}\mu\right)^n+O\left(\left(\lambda^{\theta}\mu^{1-\frac{\theta}{2}}\right)^n\right)$.

\end{prop}
\begin{proof}
Choose $ t_1\in(-t_0,t_0)$. By Lemma \ref{Nderivatives} and by Remark \ref{highzya}, we get the following expression for $F_{t,a}^N(x, y)$, in coordinates $(x,y, t, a)$ centered around $(0,1, t_1, 0)=(c,t_1,0)$ and coordinates centered around $z=F_{t_1,0}^N(0, 1)$,
\begin{eqnarray}\label{FNexpr}
F_{t,a}^N(x, y)&=&\left(
\begin{matrix}
Ax + By+tE+ aF\\
Cx+Q y^2+a\left[1+x\varphi_x +y^2\varphi_y\right] 
\end{matrix}
\right),
\end{eqnarray}
where $A(x,y,t,a), B(x,y,t,a), C(x,y,t,a ), E(x,y,t,a),F(x,y,t,a), \varphi_x(x,y,t,a)$ are $\Ct$ functions and  $Q(x,y,t,a), \varphi_y(x,y,t,a)$ are $\Cd$ function over $B,Q\neq 0$. The fact that $B\neq 0$ follows from $(f4)$. For the $y$ component of $F_{t,a}^{N}(x,y)$ observe that the linear terms in $t$ and $y$ are absent because in $a=0$, the point $F^N_{t,0}(0,1)$ is a non-degenerate homoclinic tangency, see Remark \ref{highzya}. The $a$-dependence of the second component of $F^N_{t,a}$ follows from Remark \ref{highzya}.
\\
Recall that $c'\in F^{\theta n}_{t,a}\left(W^u_{\text{loc}}(z)\right)\cap\Gamma$, $z^{(1)}$ is defined as the lowest point of $F^N_{t,a}\left(W^u_{\text{loc}}(c')\right)$, $z^{(2)}$ as the lowest point of $F^n_{t,a}\left(W^u_{\text{loc}}(z^{(1)})\right)$ and $z^{(3)}$ as the lowest point of $F^N_{t,a}\left(W^u_{\text{loc}}(z^{(2)})\right)$.
We fix the points $m_6, m_5,m_4, m_3, m_2, m_1$ so that they satisfy the following, see Figure \ref{Fig4b}. 
\begin{itemize}
\item[-] $m_6\in W^u_{\text{loc}}(z^{(3)})$ sufficiently close to $z^{(3)}$.
 \item[-] $m_5=(m_{5,x},m_{5,y})\in W^u_{\text{loc}}(z^{(2)})$, $F_{t,a}^{-N}(m_6)=m_5$. By Lemma \ref{ymaxymin} and \eqref{eq:n0*def}, 

\begin{eqnarray}\label{m5} \nonumber
 (C-E)\left(\lambda^{\theta}\mu\right)^n+O\left(\left(\lambda^{\theta}\mu^{1-\frac{\theta}{2}}\right)^n\right)&\leq&\\ m_{5,y}- c_{n,y}&\leq &\\\nonumber\frac{T}{\mu^{n_0^*/2}}+O\left(\frac{n}{\mu^{n+{n_0^*}/{2}}}\right),
  \end{eqnarray}

and by (\ref{eq:deltaz2inorder})
\begin{equation}\label{m5x1}
m_{5,x}-z^{(2)}_x=O\left(\left(\frac{\lambda^{2-\theta}}{\mu^2}\right)^n\right).
\end{equation}

In particular, when $a=a_n(t)$ by Lemma \ref{ymaxymin} and \eqref{eq:deltay2inorder},
\begin{eqnarray}\label{m5an}
 m_{5,y}- c_{n,y}=C\left(\lambda^{\theta}\mu\right)^n+O\left(\left(\lambda^{\theta}\mu^{1-\frac{\theta}{2}}\right)^n\right).
  \end{eqnarray} 

\item[-] $m_4=(m_{4,x},m_{4,y})\in W^u_{\text{loc}}(z^{(1)})$, $F_{t,a}^{-n}(m_5)=m_4$. By Lemma \ref{z1distz} and the fact that $z_{n,y}$ can be estimated by $\mu^{-n}$, see Remark \ref{muzeroothers},
\begin{equation}\label{m4}
m_{4,y}=O\left(\frac{1}{\mu^n}\right),
\end{equation}
and by (\ref{m5x1}), using the linear map backward we get, 
\begin{equation}\label{m41}
m_{4,x}-z^{(1)}_x=O\left(\left(\frac{\lambda^{1-\theta}}{\mu^2}\right)^n\right).
\end{equation}
\item[-] $m_3=(m_{3,x},m_{3,y})\in W^u_{\text{loc}}(c')$, $F_{t,a}^{-N}(m_4)=m_3$. By \eqref{eq:xcoordinatec'} and (\ref{m41}),
\begin{eqnarray}\label{m3} 
m_{3}-c&=&O\left({\lambda^{\theta{n}}}\right),\\
m_{3}-c'&=&O\left(\left(\frac{\lambda^{1-\theta}}{\mu^2}\right)^n\right).
\end{eqnarray}
\item[-] $m_2=(m_{2,x},m_{2,y})\in W^u_{\text{loc}}(z)$, $F_{t,a}^{-\theta n}(m_3)=m_2$. Therefore, there exists a uniform constant $K_2>0$ such that
\begin{equation}\label{m2}
\frac{1}{K_2}\frac{1}{\mu^{\theta{n}}}\leq m_{2,y}\leq K_2\frac{1}{\mu^{\theta{n}}}.
\end{equation}
\item[-] $m_1=(m_{1,x},m_{1,y})\in W^u_{\text{loc}}(c)$, $F_{t,a}^{-N}(m_2)=m_1$. By (\ref{m2}), \eqref{munan} and Remark \ref{muzeroothers},
\begin{equation*}
\frac{1}{K_2}\frac{1}{\mu^{\theta{n}}}\leq m_{2,y}-z_y\leq K_2\frac{1}{\mu^{\theta{n}}}.
\end{equation*}
and because of the quadratic behavior, see Lemma \ref{shapeofthecurveatz1}, 
\begin{equation*}
\frac{1}{K_2}\frac{1}{\mu^{\theta{n}/2}}\leq m_{2,x}-1\leq K_2\frac{1}{\mu^{\theta{n}/2}}
\end{equation*}
where we adjusted the value of the constant $K_2$.
Hence, there exists a uniform constant $K_1>0$ such that
\begin{equation}\label{m1}
\frac{1}{K_1}\frac{1}{\mu^{\theta\frac{n}{2}}}\leq m_{1,y}- c_{y}\leq K_1\frac{1}{\mu^{\theta\frac{n}{2}}}.
\end{equation}
\end{itemize}
Let us recall that, $m_2=F^N(m_1)$, $m_3=F^{\theta n}(m_2)$,  $m_4=F^N(m_3)$,  $m_5=F^n(m_4)$, and  $m_6=F^N(m_5)$. 
Moreover, $t_1\in[-t_0,t_0]$ was chosen arbitrarly. Take $(t_1+\Delta t,a)\in\mathcal B_n$ and consider the $\Cq$ map $$(\Delta t, \Delta a, \Delta y)\longmapsto F^{3N+(\theta+1)n}_{t+\Delta t, a+\Delta a}{( m_1+(0,\Delta y))}=m_6+(\Delta m_{6,x}, \Delta m_{6,y}).$$  We are interested in the partial derivatives of $m_6$. 
Observe that for $i=1,3,5$,
\begin{eqnarray}\label{mi+1deltami+1}
d m_{i+1}&=&DF^N_{t,a}\left(m_{i}\right)d m_i+\frac{\partial F^N_{t,a}\left(m_{i}\right)}{\partial t}dt+\frac{\partial F^N_{t,a}\left(m_{i}\right)}{\partial a}da,
\end{eqnarray}
for $i=4$,
\begin{eqnarray}\label{mi+1deltami+1fn}
d m_{i+1}&=&DF^n_{t,a}\left(m_{i}\right) dm_i+\frac{\partial F^n_{t,a}\left(m_{i}\right)}{\partial t}dt+\frac{\partial F^n_{t,a}\left(m_{i}\right)}{\partial a}da,
\end{eqnarray}
and for $i=2$,
\begin{eqnarray}\label{mi+1deltami+1fthetan}
dm_{i+1}&=&DF^{\theta n}_{t,a}\left(m_{i}\right)dm_i+\frac{\partial F^{\theta n}_{t,a}\left(m_{i}\right)}{\partial t}dt+\frac{\partial F^{\theta n}_{t,a}\left(m_{i}\right)}{\partial a}da.
\end{eqnarray}
All partial derivatives in (\ref{mi+1deltami+1}) are uniformly bounded. However we will need more careful estimate for $\Delta m_{i+1,y}$. Namely, for $i=1,3,5$,
\begin{eqnarray*}
DF^N_{t,a}\left(m_{i}\right)&=&\left(\begin{matrix}
A_i&B_i\\
C_i&O\left(m_{i,x}\right)+D_i (m_{i,y}-1)
\end{matrix}\right),
\end{eqnarray*}
which follows from (\ref{FNexpr}). Similarly from (\ref{FNexpr}) one obtains
\begin{eqnarray*}
\frac{\partial m_{i+1,y}}{\partial t}&=&\frac{\partial C}{\partial t}m_{i,x}+\frac{\partial Q}{\partial t}\left(m_{i,x}-1\right)^2+a m_{i,x}\frac{\partial\varphi_x}{\partial t}+a \left(m_{i,y}-1\right)^2\frac{\partial\varphi_y}{\partial t}\\
&=&d_i \left( m_{i,y}-1\right )^2+O\left(m_{i,x}\right),
\end{eqnarray*}
where $d_i\neq 0$ is uniformly bounded away from zero. By differentiation of (\ref{FNexpr}) with respect to $\Delta a$ we obtain
\begin{eqnarray*}
\frac{\partial m_{i+1,y}}{\partial a}&=&1+O\left(\left( m_{i,y}-1\right )^2\right)+O\left(m_{i,x}\right).
\end{eqnarray*}
For $i=3$ we refine the estimate 
\begin{eqnarray*}
DF^N_{t,a}\left(m_{3}\right)&=&\left(\begin{matrix}
A_3&B_3\\
C_3&O\left(m_{3}-c'\right)
\end{matrix}\right),
\end{eqnarray*}
where we estimate $DF^N(m_3)$ with $DF^N(c')$ evaluated at a point at distance $|m_{3}-c'|$. For $i=1,5$ we get in an analogous way

\begin{eqnarray}\label{DFN}
d m_{i+1}&=&\left(\begin{matrix}
A_i&B_i\\
C_i&O\left(m_{i,x}\right)+D_i (m_{i,y}-1)
\end{matrix}\right)d m_i\\&+&\left(\begin{matrix}
O\left(d t\right)+O\left(d a\right)\\
\left[d_i (m_{i,y}-1)^2+O\left(m_{i,x}\right)\right]d t+\left[1+O\left((m_{i,y}-1)^2+m_{i,x}\right)\right]d a
\end{matrix}\right),\nonumber
\end{eqnarray}
and  
\begin{eqnarray}\label{DFN3}
d m_{4}&=&\left(\begin{matrix}
A_3&B_3\\
C_3& O\left(m_{3}-c'\right)
\end{matrix}\right)d m_3\\&+&\left(\begin{matrix}
O\left(d t\right)+O\left(d a\right)\\
\left[d_3 (m_{3,y}-1)^2+O\left(m_{3,x}\right)\right]d t+\left[1+O\left((m_{3,y}-1)^2+m_{3,x}\right)\right]d a
\end{matrix}\right),\nonumber
\end{eqnarray}
where $D_i,B_i,C_i\neq 0$ (because $q_1$ is a non degenerate tangency in general direction and $DF^N(0,1)$ is non singular). Moreover, by construction,
\begin{equation}\label{eq:C3andC}
 C_3= C(1+O\left(\lambda^{\theta_n}\right)).   
\end{equation}
From $d m_4$ to $d m_5$ we use the linear map $F^n$. From (\ref{mi+1deltami+1fn}) and using that $m_{5,y}=\mu^n m_{4,y}$ we get 
\begin{eqnarray}\label{DFn}
d m_{5}&=&\left(\begin{matrix}
\lambda^{n}&0\\
0&\mu^{n}
\end{matrix}\right)d m_4+\left(\begin{matrix}
X_4\frac{n\lambda^n}{\lambda}\left[ \frac{\partial \lambda}{\partial t} d t+\frac{\partial \lambda}{\partial a}d a\right]\\
m_{5,y}\frac{n}{\mu}\left[\frac{\partial \mu}{\partial t}d t+\frac{\partial \mu}{\partial a}d a \right]
\end{matrix}\right),
\end{eqnarray}
with $0\neq X_4=m_{4,x}\approx 1$. 
A similar formula holds going from $d m_2$ to $d m_3$ where we use the linear map $F^{\theta n}$, (\ref{mi+1deltami+1fthetan}) and the fact that $m_{3,y}=\mu^{\theta n} m_{2,y}$. Namely,
\begin{eqnarray}\label{DFthetan}
d m_{3}&=&\left(\begin{matrix}
\lambda^{\theta n}&0\\
0&\mu^{\theta n}
\end{matrix}\right)d m_2+\left(\begin{matrix}
X_2\frac{\theta n\lambda^{\theta n}}{\lambda}\left[ \frac{\partial \lambda}{\partial t} d t+\frac{\partial \lambda}{\partial a}d a\right]\\
m_{3,y}\frac{\theta n}{\mu}\left[\frac{\partial \mu}{\partial t}d t+\frac{\partial \mu}{\partial a}d a \right]
\end{matrix}\right),
\end{eqnarray}
where 
\begin{equation}\label{eq:X2}
0\neq X_2=m_{2,x}\approx 1.
\end{equation}
Observe that $m_{1,x}=0$ and recall that $\left(d m_{1,x},d m_{1,y}\right)=\left(0,d y\right)$. By (\ref{DFN}) and by (\ref{m1}),
\begin{eqnarray*}
\left(
\begin{matrix}
d m_{2,x}\\
d m_{2,y}
\end{matrix}\right)=\left(\begin{matrix}
A_1&B_1\\
C_1&D_1\tilde K_1{\mu^{\frac{-\theta{n}}{2}}}
\end{matrix}\right)\left(
\begin{matrix}
0\\
d y
\end{matrix}\right)
+\left(\begin{matrix}
O\left(d t\right)+O\left(d a\right)\\
\left[d_1 \tilde K_1^2{\mu^{{-\theta{n}}}}\right]d t+\left[1+O\left({\mu^{{-\theta{n}}}}\right)\right]d a
\end{matrix}\right),
\end{eqnarray*}
where $\tilde K_1$ is the constant which gives an equality in (\ref{m1}). As a consequence 
\begin{eqnarray*}
d m_{2,x}&=&O\left(d t\right) +O\left(d a\right)+O\left(d y\right),\\
d m_{2,y}&=&\frac{\tilde d_1}{\mu^{\theta {n}}}d t+\left[1+O\left(\frac{1}{\mu^{\theta{n}}}\right)\right]d a+\frac{\tilde D_1}{\mu^{\theta\frac{n}{2}}}d y,
\end{eqnarray*}
where $\tilde d_1=d_1 \tilde K_1^2$ and $\tilde D_1=D_1\tilde K_1$. By using the fact that $F^{\theta n}_{t,a}$ is linear and using (\ref{DFthetan}),
\begin{eqnarray*}
\left(\begin{matrix}
d m_{3,x}\\ d m_{3,y}
\end{matrix}
\right)&=&\left(\begin{matrix}
\lambda^{\theta n}&0\\
0&\mu^{\theta n}
\end{matrix}\right)\left(\begin{matrix}
d m_{2,x}\\ d m_{2,y}
\end{matrix}
\right)+\left(\begin{matrix}
X_2\frac{\theta n\lambda^{\theta n}}{\lambda}\left[ \frac{\partial \lambda}{\partial t} d t+\frac{\partial \lambda}{\partial a}d a\right]\\
\left(\left(m_{3,y}-c_{n,y}\right)+c_{n,y}\right)\frac{\theta n}{\mu}\left[\frac{\partial \mu}{\partial t}d t+\frac{\partial \mu}{\partial a}d a \right]
\end{matrix}\right),
\end{eqnarray*}
 and using $(\ref{m3})$ and Remark \ref{highzya} (recall that $c$ is fixed at height $1$, i.e. $c_{n,y}=1$)
\begin{eqnarray*} 
d m_{3,x}&=&X_2\frac{\theta n\lambda^{\theta n}}{\lambda}\left[ \frac{\partial \lambda}{\partial t} d t+\frac{\partial \lambda}{\partial a}d a\right]+O\left(\lambda^{\theta n}d y\right)\\
&+&O\left(\lambda^{\theta n}d t\right)+O\left(\lambda^{\theta n}d a\right),\\
d m_{3,y}&=&\left[\frac{\theta n}{\mu}\frac{\partial\mu}{\partial t}+\tilde d_1+O\left(n\lambda^{\theta n}\right)\right]d t\\&+&\left[\mu^{\theta n}+\frac{\theta n}{\mu}\frac{\partial\mu}{\partial a}+O\left(1\right)\right]d a+\tilde D_1{\mu^{\theta\frac{n}{2}}}d y,
\end{eqnarray*}
where $X_2\neq 0$. By (\ref{DFN3}) and (\ref{m3}) (observe that $m_{3,y}-c_{n,y},m_{3,x}=O\left(\lambda^{\theta n}\right)$), we get
\begin{eqnarray*}
\left(
\begin{matrix}
d m_{4,x}\\
d m_{4,y}
\end{matrix}\right)=\left(\begin{matrix}
A_3&B_3\\
C_3&O\left(\left(\frac{\lambda^{1-\theta}}{\mu^2}\right)^n\right)
\end{matrix}\right)\left(
\begin{matrix}
d m_{3,x}\\
d m_{3,y}
\end{matrix}\right)
+\left(\begin{matrix}
O\left(d t\right)+O\left(d a\right)\\
O\left(\lambda^{\theta n}\right)d t+\left[1+O\left(\lambda^{\theta n}\right)\right]d a
\end{matrix}\right).
\end{eqnarray*}
As consequence, using also \eqref{eq:X2} and \eqref{eq:C3andC},
\begin{eqnarray*}
d m_{4,x}&=&\left[B_{3}\frac{\theta n}{\mu}\frac{\partial\mu}{\partial t}+O\left(1\right)\right]d t\\&+&\left[B_{3}\mu^{\theta n}+B_{3}\frac{\theta n}{\mu}\frac{\partial\mu}{\partial a}+O\left(1\right)\right]d a\\&+&\left[B_{3}\tilde D_1{\mu^{\theta\frac{n}{2}}}+O\left({\lambda^{\theta{n}}}\right)\right]d y,\\
d m_{4,y}&=&\left[C\frac{\theta n\lambda^{\theta n}}{\lambda} \frac{\partial \lambda}{\partial t} + O\left(\lambda^{\theta n}\right)\right]d t \\&+&\left[1+C\frac{\theta n\lambda^{\theta n}}{\lambda} \frac{\partial \lambda}{\partial a} + O\left(\lambda^{\theta n}\right)\right]d a\\&+&O\left(\lambda ^{\theta{n}}d y \right).
\end{eqnarray*}
 By (\ref{DFn}) we get 
\begin{eqnarray*}
d m_{5,x}&=&\left[B_{3}\frac{\theta n \lambda^n}{\mu}\frac{\partial\mu}{\partial t}+X_4\frac{n\lambda^{n}}{\lambda} \frac{\partial \lambda}{\partial t}+O\left(\lambda^n\right)\right]d t\\&+&\left[B_3\left(\lambda\mu^{\theta}\right)^n+B_{3}\frac{\theta n \lambda^n}{\mu}\frac{\partial\mu}{\partial a}+X_4\frac{n\lambda^{n}}{\lambda} \frac{\partial \lambda}{\partial a}+O\left(\lambda^n\right)\right]d a\\&+&\left[B_{3}\tilde D_1\left(\lambda\mu^{\frac{\theta}{2}}\right)^n+O\left(\left({\lambda^{\theta+1}}\right)^n\right)\right]d y,\\
d m_{5,y}&\geq & \left[\frac{n}{\mu}\frac{\partial\mu}{\partial t}+n\left(\lambda^{\theta}\mu\right)^n\left[C\frac{\theta }{\lambda} \frac{\partial \lambda}{\partial t}+\frac{(C-E)}{\mu}\frac{\partial\mu}{\partial t}\right] +O\left(\left(\lambda^{\theta}\mu\right)^n\right)\right] d t\\&+&
\left[\mu^n+\frac{n}{\mu}\frac{\partial\mu}{\partial a}+n\left(\lambda^{\theta}\mu\right)^n\left[C\frac{\theta }{\lambda} \frac{\partial \lambda}{\partial a}+\frac{(C-E)}{\mu}\frac{\partial\mu}{\partial a}\right] +O\left(\left(\lambda^{\theta}\mu\right)^n\right)\right] d a\\&+&
O\left(\left(\lambda ^{\theta}\mu\right)^{{n}}d y \right),
\end{eqnarray*}
where we used that $m_{5,y}=(m_{5,y}-c_{n,y})+c_{n,y}$, the fact that the point $c$ is at height $1$ and (\ref{m5}). Similarly,
\begin{eqnarray*}
d m_{5,y}&\leq & \left[\frac{n}{\mu}\frac{\partial\mu}{\partial t}+\frac{n}{\mu}\frac{\partial\mu}{\partial t}\frac{T}{\mu^{n_0*/2}}+n\left(\lambda^{\theta}\mu\right)^n C\frac{\theta }{\lambda} \frac{\partial \lambda}{\partial t} +O\left(\left(\lambda^{\theta}\mu\right)^n\right)\right] d t\\&+&
\left[\mu^n+\frac{n}{\mu}\frac{\partial\mu}{\partial a}+\frac{n}{\mu}\frac{\partial\mu}{\partial a}\frac{T}{\mu^{n_0^*/2}}+n\left(\lambda^{\theta}\mu\right)^n C\frac{\theta }{\lambda} \frac{\partial \lambda}{\partial a}  +O\left(\left(\lambda^{\theta}\mu\right)^n\right)\right] d a\\&+&
O\left(\left(\lambda ^{\theta}\mu\right)^{{n}}d y \right).
\end{eqnarray*}
In particular, when $a=a_n$, using the same equalities as before and \ref{m5an}, 
\begin{eqnarray}\label{eq:dm5an}
d m_{5,y}&=& \left[\frac{n}{\mu}\frac{\partial\mu}{\partial t}+n\left(\lambda^{\theta}\mu\right)^nC\left[\frac{\theta }{\lambda} \frac{\partial \lambda}{\partial t}+\frac{1}{\mu}\frac{\partial\mu}{\partial t}\right] +O\left(\left(\lambda^{\theta}\mu\right)^n\right)\right] d t\\\nonumber &+&
\left[\mu^n+\frac{n}{\mu}\frac{\partial\mu}{\partial a}+n\left(\lambda^{\theta}\mu\right)^nC\left[\frac{\theta }{\lambda} \frac{\partial \lambda}{\partial a}+\frac{1}{\mu}\frac{\partial\mu}{\partial a}\right] +O\left(\left(\lambda^{\theta}\mu\right)^n\right)\right] d a\\\nonumber &+&
O\left(\left(\lambda ^{\theta}\mu\right)^{{n}}d y \right),
\end{eqnarray}
 By (\ref{DFN}), the lower bound in (\ref{m5}) and (\ref{m5x1})  we get 
 \begin{eqnarray*}
\left(
\begin{matrix}
d m_{6,x}\\
d m_{6,y}
\end{matrix}\right)&\geq&\left(\begin{matrix}
A_5&B_5\\
C_5&D_5S_{C-E}
\end{matrix}\right)\left(
\begin{matrix}
d m_{5,x}\\
d m_{5,y}
\end{matrix}\right)\\
&+&\left(\begin{matrix}
O\left(d t\right)+O\left(d a\right)\\
\left[d_5S_{C-E}^2\right]d t+\left[1+O\left(S_{C-E}^2\right)\right]d a
\end{matrix}\right),
\end{eqnarray*}
 where $S_{C-E}=(C-E)\left(\lambda^{\theta}\mu\right)^n+O\left(\left(\lambda^{\theta}\mu^{1-\frac{\theta}{2}}\right)^n\right)$.
As consequence,
 
\begin{eqnarray}\label{dm6x1}
d m_{6,x}&\geq& \left [B_5\frac{n}{\mu}\frac{\partial\mu}{\partial t} +O\left(1\right)\right] d t+
\left[B_5\mu^n+B_5\frac{n}{\mu}\frac{\partial\mu}{\partial a}+O\left(1\right)\right]d a\\&+&O\left(\left(\lambda ^{\theta}\mu\right)^{{n}}d y \right),\nonumber
\end{eqnarray}
\begin{eqnarray}\label{dm6y1}
d m_{6,y}&\geq&\nonumber\left[D_5S_{C-E}  \frac{n}{\mu}\frac{\partial\mu}{\partial t}+ D_5S_{C-E}n \left(\lambda^{\theta}\mu\right)^{n}\left[C\frac{\theta }{\lambda} \frac{\partial \lambda}{\partial t}+\frac{(C-E)}{\mu}\frac{\partial\mu}{\partial t}\right]\right.
\\ &+& 
\left.O\left(S_{C-E}\left(\lambda^{\theta}\mu\right)^{n}\right)\right]d t\\&+& 
\nonumber\left[D_5S_{C-E}\mu^n+1+D_5S_{C-E}  \frac{n}{\mu}\frac{\partial\mu}{\partial a}+ D_5S_{C-E}n \left(\lambda^{\theta}\mu\right)^{n}\left[C\frac{\theta }{\lambda} \frac{\partial \lambda}{\partial a}+\frac{(C-E)}{\mu}\frac{\partial\mu}{\partial a}\right]\right.
\\ \label{dm6yda1}&+& 
\left.O\left(S_{C-E}\left(\lambda^{\theta}\mu\right)^{n}\right)\right]d a + O\left(\left(\lambda ^{\theta}\mu\right)^{{2n}} \right)d y
\end{eqnarray}
where $D_5, B_{5}, (C-E)\ne 0$ and $S_{C-E}\mu^n>1$ (see (\ref{thetacond})). Similarly,  
 By (\ref{DFN}), the upper bound in (\ref{m5}) and (\ref{m5x1})  we get 
 \begin{eqnarray*}
\left(
\begin{matrix}
d m_{6,x}\\
d m_{6,y}
\end{matrix}\right)&\leq&\left(\begin{matrix}
A_5&B_5\\
C_5&{D_5T}/{\mu^{n_0^*/2}}+O\left({n}/{\mu^{n+n_0^*/2}}\right)\end{matrix}\right)\left(
\begin{matrix}
d m_{5,x}\\
d m_{5,y}
\end{matrix}\right)\\
&+&\left(\begin{matrix}
O\left(d t\right)+O\left(d a\right)\\
\left[{d_5T^2}/{\mu^{n_0^*}}+O\left({n}/{\mu^{n+n_0^*}}\right)\right]dt+\left[1+O\left({1}/{\mu^{n_0^*}}\right)\right]d a
\end{matrix}\right)
\end{eqnarray*}
As consequence, if $(t,a)\in\mathcal B_n$
 \begin{eqnarray}\label{dm6x1}
d m_{6,x}&\leq& \left [B_5\frac{n}{\mu}\frac{\partial\mu}{\partial t} +O\left(1\right)\right] d t+
\left[B_5\mu^n+B_5\frac{n}{\mu}\frac{\partial\mu}{\partial a}+O\left(1\right)\right]d a\\&+&O\left(\left(\lambda ^{\theta}\mu\right)^{{n}}d y \right),\nonumber
\end{eqnarray}

 \begin{eqnarray}\label{dm6y2}
d m_{6,y}&\leq&\left[\frac{D_5T}{\mu^{n_0^*/2}}\frac{n}{\mu}\frac{\partial\mu}{\partial t}+\frac{n}{\mu}\frac{\partial\mu}{\partial t}\frac{D_5T^2}{\mu^{n_0^*}}+O\left(\frac{1}{\mu^{n_0^*}}\right)\right] d t\\
\label{dm6yda2}&+&\left[D_5T\mu^{n-n_0^*/2}+1+\frac{D_5T}{\mu^{n_0^*/2}}\frac{n}{\mu}\frac{\partial\mu}{\partial a}+\frac{n}{\mu}\frac{\partial\mu}{\partial a}\frac{D_5T^2}{\mu^{n_0^*}} +O\left(\frac{1}{\mu^{n_0^*}}\right)\right] d a\\\nonumber &+& O\left(\left(\lambda ^{\theta}\mu\right)^{{2n}} \right)d y,
\end{eqnarray}
where $D_5, B_{5}, T\ne 0$ and $\mu^{n-n_0^*/2}>1$ (see \eqref{eq:n0*def} and (\ref{thetacond})).

In particular, when $a=a_n$, using the same equalities as before, \eqref{m5an} and \eqref{eq:dm5an} we get,

 \begin{eqnarray}\label{dm6xan}
d m_{6,x}&=& \left [B_5\frac{n}{\mu}\frac{\partial\mu}{\partial t} +O\left(1\right)\right] d t+
\left[B_5\mu^n+B_5\frac{n}{\mu}\frac{\partial\mu}{\partial a}+O\left(1\right)\right]d a\\&+&O\left(\left(\lambda ^{\theta}\mu\right)^{{n}}d y \right),\nonumber
\end{eqnarray}

and 

\begin{eqnarray}\label{dm6yan}
d m_{6,y}&=&\nonumber\left[D_5S_{C}  \frac{n}{\mu}\frac{\partial\mu}{\partial t}+ D_5S_{C}Cn \left(\lambda^{\theta}\mu\right)^{n}\left[\frac{\theta }{\lambda} \frac{\partial \lambda}{\partial t}+\frac{1}{\mu}\frac{\partial\mu}{\partial t}\right]\right.
\\ &+& 
\left.O\left(S_{C}\left(\lambda^{\theta}\mu\right)^{n}\right)\right]d t\\&+& 
\nonumber\left[D_5S_{C}\mu^n+1+D_5S_{C}  \frac{n}{\mu}\frac{\partial\mu}{\partial a}+ D_5S_{C}Cn \left(\lambda^{\theta}\mu\right)^{n}\left[\frac{\theta }{\lambda} \frac{\partial \lambda}{\partial a}+\frac{1}{\mu}\frac{\partial\mu}{\partial a}\right]\right.
\\ \label{dm6ydan}&+& 
\left.O\left(S_{C}\left(\lambda^{\theta}\mu\right)^{n}\right)\right]d a + O\left(\left(\lambda ^{\theta}\mu\right)^{{2n}} \right)d y
\end{eqnarray}

 where $S_{C}=C\left(\lambda^{\theta}\mu\right)^n+O\left(\left(\lambda^{\theta}\mu^{1-\frac{\theta}{2}}\right)^n\right)$, and by \eqref{eq:thetaconddlambdadmu} 
 $$
  \left|\frac{\theta}{\lambda}\frac{\partial\lambda}{\partial t}+\frac{1}{\mu}\frac{\partial\mu}{\partial t}\right|>w>0.
$$
The formula for $m_{6,y}$ holds in general for all $\Delta t,\Delta a$ and $\Delta y$. However at the point $(t_1,a)$ when $\Delta t=\Delta a=0$ we have ${\partial m_{6,y}}/{\partial y}=0$. Hence, the Taylor polynomial of second order for $\Delta m_{6,y}$ does not contain a linear term in $\Delta y$. As consequence ${\partial z^{(3)}_{y}}/{\partial t}={\partial m_{6,y}}/{\partial t}$ and ${\partial z^{(3)}_{y}}/{\partial a}={\partial m_{6,y}}/{\partial a}$. 
It is left to prove that $C,D_5$ converge. This is achieved since they are part of the derivative $DF^N$ which converges to $DF^N(0,1)$.  
The proposition follows.
\end{proof}

In the next proposition we are going to prove that secondary tangencies exist for certain parameters in $a_n$ and they are at distance of order $1/n$ to each other. The result is achieved by comparing the rate of speed of $W^u_{\text{loc}}(z^{(3)})$ and $ W_{n-n_0}$ when changing parameters in the phase space and it relies on the cancellation of the main term.

\begin{prop}\label{newtangency}
 For all $t\in\left(-t_0,t_0\right)$ and for $n$ large enough there exists $t_n\in\left(-t_0,t_0\right)$ such $t- t_n=O\left({1}/{n}\right)$ and $F_{t_n,  a_n(t_n)}$ has a secondary tangency. 

\end{prop}
\begin{proof}
Choose $t\in\left(-t_0,t_0\right)$. 
 Use the notation of Proposition \ref{speed} and notice that, by the choice of $\theta$, see \eqref{eq:thetaconddlambdadmu},
 $$\left|\frac{\theta }{\lambda} \frac{\partial \lambda}{\partial t}+\frac{1}{\mu}\frac{\partial\mu}{\partial t}\right|>w>0.
 $$
 Suppose that, 
 $$DC^2\left[\frac{\theta }{\lambda} \frac{\partial \lambda}{\partial t}+\frac{1}{\mu}\frac{\partial\mu}{\partial t}\right]=v>0,
 $$
 the negative case is handled similarly by reversing the time. Take $n$ large enough and a point $(s, a_n(s))\in A_n$ near $(t,a_n(t))\in A_n$. By Proposition \ref{speed} and Lemma \ref{Imtang}  
\begin{eqnarray}\label{upperineq}\nonumber
z^{(3)}_y\left(s,a_n(s)\right)&=&z^{(3)}_y\left(t,a_n(t)\right)+\int_0^{(s- t)}\left[\frac{\partial z^{(3)}}{\partial t}+\frac{\partial z^{(3)}}{\partial a}\frac{d a_n}{d t}\right]dt=z^{(3)}_y\left(t,a_n(t)\right)\\\nonumber &+&DS_C\left\{  \frac{n}{\mu}\frac{\partial\mu}{\partial t}+ Cn \left(\lambda^{\theta}\mu\right)^{n}\left[\frac{\theta }{\lambda} \frac{\partial \lambda}{\partial t}+\frac{1}{\mu}\frac{\partial\mu}{\partial t}\right]- \frac{n}{\mu}\frac{\partial\mu}{\partial t} +O\left(\left(\lambda^{\theta}\mu\right)^{n}\right)\right\}(s-t)
\\\nonumber &=&z^{(3)}_y\left(t,a_n(t)\right)+\left\{DC^2n \left(\lambda^{\theta}\mu\right)^{2n}\left[\frac{\theta }{\lambda} \frac{\partial \lambda}{\partial t}+\frac{1}{\mu}\frac{\partial\mu}{\partial t}\right]+O\left(\left(\lambda^{\theta}\mu\right)^{2n}\right)\right\}(s-t)
\\&\geq &z^{(3)}_y\left(t,a_n(t)\right)+n\left(\lambda^{\theta}\mu\right)^{2n}\frac{v}{2}(s-t),
\end{eqnarray}
where we used that $\lambda^{2\theta}\mu^3>1$ and $\lambda^{\theta}\mu^2>1$  see (\ref{thetacond}). Observe that we have a cancellation of the dominant terms in the partial derivatives obtained by combining Proposition \ref{speed} and Lemma \ref{Imtang}.
Let $n_0=n_0(t,a_n(t))$ be as in \eqref{nminusn0}. Hence, by (\ref{wn}), 
\begin{equation}\label{ztildeheight}
z^{(3)}_y\left(t,a_n(t)\right)\geq \frac{2}{\mu}\frac{1}{\mu^{n-n_0}}.
\end{equation}
By (\ref{upperineq}) and (\ref{ztildeheight})
\begin{eqnarray}\label{upperineq1}
z^{(3)}_y\left(s,a_n(s)\right)\geq \frac{2}{\mu}\frac{1}{\mu^{n-n_0}}+n\left(\lambda^{\theta}\mu\right)^{2n}\frac{v}{2}(s-t).
\end{eqnarray}
Choose $\kappa\geq 2$ such that $\left[\mu^{\kappa}-{2}/{\mu}\right]\geq v/4$. Then by (\ref{wn}), in a neighborhood of $(z^{(3)}\left(t,a_n(t)\right),t,a_n(t))$ 
\begin{eqnarray*}
\max W_{n-n_0-\kappa}&\leq &\frac{2}{\mu^{n-n_0-\kappa}}.
\end{eqnarray*}
From the previous inequality, (\ref{upperineq1}), we get that $W^u_{\text{loc}}(z^{(3)}( s, a_n(s)))$ is above $W_{n-n_0^*-\kappa}$ if
\begin{eqnarray*}
 n\left(\lambda^{\theta}\mu\right)^{2n}\frac{v}{2}(s-t)\geq \frac{2}{\mu^{n-n_0}}\left[\mu^{\kappa}-\frac{2}{\mu}\right]\geq 2C'\left(\lambda^{\theta}\mu\right)^{2n}\left[\mu^{\kappa}-\frac{2}{\mu}\right]\geq C'\left(\lambda^{\theta}\mu\right)^{2n}\frac{v}{2},
\end{eqnarray*}
where we also used (\ref{alphaofuse}) and $C'$ is a uniform constant.
As consequence, if 
$$
s-t\geq \frac{C'}{n},$$
then, $W^u_{\text{loc}}(z^{(3)}( s, a_n(s)))$ is above $W_{n-n_0-\kappa}$. Because $W^u_{\text{loc}}( z^{(3)}( t, a_n(t)))$ contains a point below $ W_{n-n_0-1}$, there exists a parameter between $( s, a_n(s))$ and $( t, a_n(t))$ for which a secondary homoclinic tangency of type $n_0+1$ occurs and $t-s=O\left({1}/{n}\right)$.
\end{proof}

\begin{figure}
\centering
\includegraphics[width=0.6\textwidth]{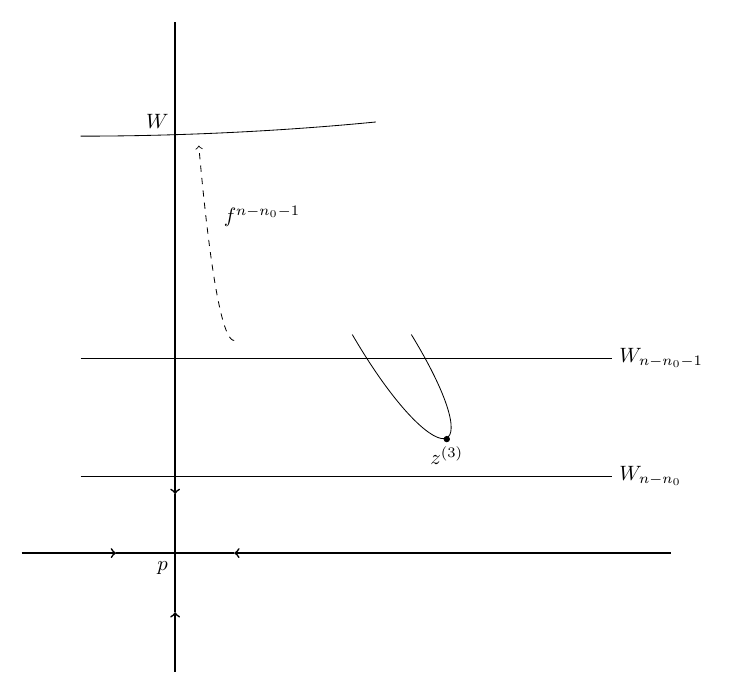}
\caption{Position of $z^{(3)}$}
\label{Fig5}
\end{figure}

The next lemma estimates the curvature of any secondary tangency point of type $n_0$ in $\mathcal{B}_n$ satisfies the following estimate. 
\begin{lem}\label{curvature}
Let $(t,a)\in\mathcal{B}_n$ be such that $W^u_{\text{\rm loc}}(z^{(3)}(t,a))$ has a secondary tangency $q_{1,n,n_0}$ of type $n_0$, then $q_{1,n,n_0}$ is non degenerate.  Namely, $W^u_{\text{\rm loc}}(q_{1,n,n_0})$ is the graph of a function and its curvature satisfies
$$
\text{\rm curv}\left(W^u_{\text{\rm loc}}( q_{1,n,n_0})\right)\geq K\left(\frac{\mu^4}{\lambda^{2-3\theta}}\right)^n,
$$
with $K>0$ a uniform positive constant.
\end{lem}
\begin{proof} Use coordinates centered at $z^{(3)}$ and let $(\Delta x_{z^{(3)}},\Delta y_{z^{(3)}})\in W^u_{\text{loc}}(z^{(3)})$ corresponding to the tangency point $q_{1,n,n_0}$. Then by \eqref{dwn}, 
$$
\frac{d\Delta y_{z^{(3)}}}{d\Delta x_{z^{(3)}}}=O\left(\left(\frac{\lambda}{\mu}\right)^{n-n_0}\right)
$$
and by differentiating the equation in Lemma \ref{curvaturez3},
\begin{equation}\label{eq:derivativeattangency}
\frac{d\Delta y_{z^{(3)}}}{d\Delta x_{z^{(3)}}}=O\left(\left(\frac{\lambda}{\mu}\right)^{n-n_0}\right)=\frac{16Q^4}{C^2B^4}{H^3}\left(\frac{\mu}{\lambda^2}\right)^n\Delta x_{z^{(3)}}+O\left(\Delta x_{z^{(3)}}^2\right).
\end{equation}
In particular,
$$
{\Delta x_{z^{(3)}}}=O\left(\left(\frac{\lambda}{\mu}\right)^{n-n_0}\left(\frac{\lambda^2}{\mu}\right)^{n}\frac{1}{H^3}\right).
$$
The curvature at the tangency point is found by differentiating \eqref{eq:derivativeattangency}. Hence,
\begin{eqnarray*}
\frac{d^2\Delta y_{z^{(3)}}}{d\Delta x_{z^{(3)}}^2}&=&\frac{16Q^4}{C^2B^4}{H^3}\left(\frac{\mu}{\lambda^2}\right)^n+O\left(\Delta x_{z^{(3)}}\right)\\&\geq& \frac{16Q^4}{C^2B^4}{H^3}\left(\frac{\mu}{\lambda^2}\right)^n\left[1-O\left(\left(\frac{\lambda}{\mu}\right)^{n-n_0}\left(\frac{\lambda^4}{\mu^2}\right)^{n}\frac{1}{H^6}\right)\right]\\&\geq& \frac{8Q^4}{C^2B^4}{H^3}\left(\frac{\mu}{\lambda^2}\right)^n.
\end{eqnarray*}
 where we used the expression of $H$ given in Lemma \ref{hnbounds} and the lower bound of the strip $\mathcal{B}_n$. The lower bound for $H$ gives the statement of the Lemma.
\end{proof}

\section{Curves of secondary tangencies}
In the previous section, we proved the existence of maps with secondary tangency points on the curve $a_n$. In the sequel, we show that secondary tangencies persists along curves of uniform size. In particular these curves are graphs of functions defined on all of the $t$ domain, see Figure \ref{Fig6}. Moreover, we prove that the curves of secondary tangencies are actually long uniform curves that are actually contained in the strip $\mathcal{B}_n$.

\begin{figure}
\centering
\includegraphics[width=1\textwidth]{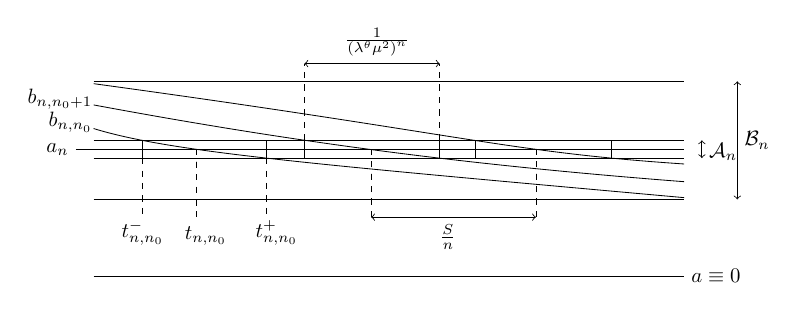}
\caption{Curves of Secondary Tangencies}
\label{Fig6}
\end{figure}

\subsection{Persistence of secondary tangencies}
The next lemma is the first step in proving the existence of local tangency curves. In particular, we prove that each parameter in $\mathcal{B}_n$ corresponding to a map having a secondary tangency of type $n_0$ is contained in a local curve. Each point on this curve corresponds to a map with a $n_0$ secondary tangency. 
\begin{lem}\label{bnlocal}
        Let $(t_{n_0}, a_{n_0})\in {\mathcal{B}_n}$ be a parameter corresponding to a map with a secondary tangency of type $n_0$. Then there exist $\epsilon>0$ and a smooth curve $$b_{n,n_0}:[t_{n_0}-\epsilon,  t_{n_0}+\epsilon]\mapsto\mathbb R$$ such that, for all $t\in[t_{n_0}-\epsilon, t_{n_0}+\epsilon]$ the parameter $(t,b_{n,n_0}(t))$ corresponds to a map with a secondary tangency of type $n_0$. 
Moreover, the slopes of the curves $b_{n,n_0}$ tend exponentially fast to zero, namely
\begin{equation*}
\left|\frac{db_{n,n_0}}{dt}\right|=O\left(\frac{n}{\left(\lambda^{\theta}\mu^2\right)^n\mu^{n_0^*/2}}\right).
\end{equation*}
In particular,
$$\frac{db_{n,n_0}}{dt}=\frac{da_n}{dt}+{n \lambda^{\theta n}}C\left[\frac{\theta }{\lambda} \frac{\partial \lambda}{\partial t}+\frac{1}{\mu}\frac{\partial\mu}{\partial t}\right]+O\left(\lambda^{\theta n}\right),$$
for all $t\in [ t_{n_0}-\epsilon,  t_{n_0}+\epsilon]$, with $( t, b_{n,n_0}( t))\in a_n$.
    \end{lem}
    \begin{proof} First we consider the case when the given point  $(t_{n_0}, a_{n_0})\in\mathcal{B}_n$ is on the curve $a_n$, i.e. $a_{n_0}=a_n(t_{n_0})$. In particular this point corresponds to a map with a secondary tangency of type $n_0$, $q_{1,n,n_0}$. 
      We construct a local function whose graph contains the point  $(t_{n_0}, a_{n_0})$. Let $m_1\in W^u_{\text{loc}}(c)$ such that $F_{t,a}^{3N+(1+\theta)n}(m_1)=q_{1,n,n_0}$. 
To describe the perturbation of $W^u_{\text{loc}}(q_{1,n,n_0})$ we define the following function by choosing coordinates centered in $q_{1,n,n_0}$.
Take some $\epsilon >0$ and consider the $\Cq$ function $
(\tilde x,\tilde y):(-\epsilon,\epsilon)^3\mapsto\R\times\R $ defined by 
$$
\left(\tilde x(\Delta y,\Delta t, \Delta a),\tilde y(\Delta y,\Delta t, \Delta a)\right)=F_{t_{n_0}+\Delta t, a_{n_0}+\Delta a}^{3N+(1+\theta) n}\left(m_1+(0,\Delta y)\right),
$$
which describes $W^u_{\text{loc}}\left(F_{t_{n_0}+\Delta t, a_{n_0}+\Delta a}^{3N+(1+\theta)n}(m_1)\right)$. Observe that the manifolds $W_{n-n_0}$ of $F_{t_{n_0}+\Delta t, a_{n_0}+\Delta a}$ are described locally, near $q_{1,n,n_0}$, as the graph of a $\Cq$ function 
$$
w_{n-n_0}:(\Delta x, \Delta t,\Delta a)\longmapsto w_{n-n_0}(\Delta x, \Delta t,\Delta a)\in\R.
$$
The curves of secondary tangencies will be constructed using the implicit function theorem. For this aim define the $\Cd$ function $\Psi:[-\epsilon,\epsilon]^3\to\R^2$ as 
$$
\Psi\left(\Delta y, \Delta t,\Delta a\right)=\left(
\begin{matrix}
\tilde y(\Delta y,\Delta t, \Delta a)-w_{n-n_0}\left(\tilde x(\Delta y, \Delta t, \Delta a),\Delta t, \Delta a\right)\\ \frac{d\tilde y}{d\tilde x}\left(\Delta y, \Delta t,\Delta a\right)-\frac{dw_{n-n_0}}{d\tilde x}\left(\tilde x(\Delta y, \Delta t, \Delta a),\Delta t, \Delta a\right)\end{matrix}\right)=\left(\begin{matrix}
\psi_1\\ \psi_2\end{matrix}\right).
$$
Observe that $\Psi^{-1}(0)$ describes locally the perturbation of the secondary tangency $q_{1,n,n_0}$ and that
\begin{equation}\label{psimap}
D\Psi({0,0,0})=\left(\begin{matrix}
0&\left(\lambda^{\theta}\mu\right)^n  \Psi_{1,2}& \left(\lambda^{\theta}\mu^2\right)^n \Psi_{1,3}\\
  \left(\frac{\mu^3}{\lambda^{2-2\theta}}\right)^n\Psi_{2,1}&  \Psi_{2,2} &  \Psi_{2,3}
\end{matrix}\right).
\end{equation}
In particular, by (\ref{dm6yan}), (\ref{dwn}), (\ref{dm6xan}), Lemma \ref{partialwn}, (\ref{alphaofuse}) (in the estimation of the order term)
\begin{eqnarray}\label{psi12}
\nonumber
\Psi_{1,2}&=&\frac{1}{\left(\lambda^{\theta}\mu\right)^n}\left[\frac{\partial m_{6,y}}{\partial t}-\frac{\partial W_{n-n_0}}{\partial x}\frac{\partial m_{6,x}}{\partial t}-\frac{\partial W_{n-n_0}}{\partial t}\right]
\\\nonumber &=&\frac{DS_C}{ \left(\lambda^{\theta}\mu\right)^{n}}\left\{\frac{n}{\mu}\frac{\partial\mu}{\partial t}+ Cn \left(\lambda^{\theta}\mu\right)^{n}\left[\frac{\theta }{\lambda} \frac{\partial \lambda}{\partial t}+\frac{1}{\mu}\frac{\partial\mu}{\partial t}\right]+ O\left(\lambda^{\theta}\mu\right)^{n}\right\}
\\\nonumber & +& 
O\left(n\left(\frac{\lambda}{\mu}\right)^{n-n_0}\frac{1}{\left(\lambda^{\theta}\mu\right)^n }\right)-\frac{1}{\left(\lambda^{\theta}\mu\right)^n }\frac{\partial W_{n-n_0}}{\partial t}
\\&=&\frac{DS_C}{ \left(\lambda^{\theta}\mu\right)^{n}}\left\{ \frac{n}{\mu}\frac{\partial\mu}{\partial t}+ Cn \left(\lambda^{\theta}\mu\right)^{n}\left[\frac{\theta }{\lambda} \frac{\partial \lambda}{\partial t}+\frac{1}{\mu}\frac{\partial\mu}{\partial t}\right]+ O\left(\lambda^{\theta}\mu\right)^{n}\right\},
\end{eqnarray}
 and by (\ref{dm6ydan}), (\ref{dwn}), (\ref{dm6xan}), Lemma \ref{partialwn} and (\ref{alphaofuse}) (in the estimation of the order term)
\begin{eqnarray}\label{psi13}
\nonumber
\Psi_{1,3}&=&\frac{1}{\left(\lambda^{\theta}\mu^2\right)^n}\left[\frac{\partial m_{6,y}}{\partial a}-\frac{\partial W_{n-n_0}}{\partial x}\frac{\partial m_{6,x}}{\partial a}-\frac{\partial W_{n-n_0}}{\partial a}\right]
\\\nonumber &=&\frac{DS_C}{\left(\lambda^{\theta}\mu^2\right)^n}\left\{\mu^n +\frac{1}{DS_C}+  \frac{n}{\mu}\frac{\partial\mu}{\partial a}+ Cn \left(\lambda^{\theta}\mu\right)^{n}\left[\frac{\theta }{\lambda} \frac{\partial \lambda}{\partial a}+\frac{1}{\mu}\frac{\partial\mu}{\partial a}\right] +O\left(\left(\lambda^{\theta}\mu\right)^{n}\right)\right\}\\\nonumber&+&O\left(\frac{\mu^n}{\left(\lambda^{\theta}\mu^2\right)^n}\left(\frac{\lambda}{\mu}\right)^{n-n_0}\right)+O\left(\frac{1}{\left(\lambda^{\theta}\mu^2\right)^n}\frac{n-n_0}{\mu^{n-n_0}}\right)
\\\nonumber&=&\frac{DS_C}{\left(\lambda^{\theta}\mu\right)^n}\left[1+O\left(\frac{1}{\left(\lambda^{\theta}\mu^2\right)^{n}}\right)\right]+O\left(\left(\lambda^{1-\alpha+\theta}\mu\right)^n\right)
\\\nonumber&=&\frac{DS_C}{\left(\lambda^{\theta}\mu\right)^n}\left[1+O\left(\frac{1}{\left(\lambda^{\theta}\mu^2\right)^{n}}\right)\right]+O\left(\left(\lambda^{\theta}\mu\right)^n\right)
\\&=&\frac{DS_C}{\left(\lambda^{\theta}\mu\right)^n}\left[1+O\left(\left(\lambda^{\theta}\mu\right)^n\right)\right],
\end{eqnarray}
where we used that $\alpha<1$ and $\lambda^{1-\alpha+\theta}\mu<\lambda^{\theta}\mu$.
Hence, by the expression for $S_C$, see Proposition \ref{speed}, $ \Psi_{1,2},\Psi_{1,3}> 0$. Moreover, by Lemma \ref{curvature}, (\ref{dm6xan}) and (\ref{ddwn}) we have $ \Psi_{2,1}\ne 0$. This implies that $D\Psi({0,0,0})$ is onto and we get that the set of secondary tangencies $\Psi^{-1}(0)$ is locally the graph of a $\Cd$ function $b$. 
Moreover the tangent space of $\Psi$ in zero satisfies $T_{(0,0,0)}\Psi^{-1}(0)=\text{Ker} D\Psi$. Hence,
$$
\left(\lambda^{\theta}\mu\right)^n  \Psi_{1,2}\Delta t+\left(\lambda^{\theta}\mu^2\right)^n\Psi_{1,3}\Delta b=0,
$$
and
\begin{equation}\label{eq:angle1}
    \frac{db}{dt}=-\frac{\Psi_{1,2}}{\Psi_{1,3}}\frac{1}{\mu^n}=-\frac{n}{\mu^{n+1}}\frac{\partial\mu}{\partial t} -{n \lambda^{\theta n}}C\left[\frac{\theta }{\lambda} \frac{\partial \lambda}{\partial t}+\frac{1}{\mu}\frac{\partial\mu}{\partial t}\right] +O\left(\lambda^{\theta n}\right).
    \end{equation}
    
    In particular $db/dt$ is a negative number and $b$ is the graph of a smooth function. 
Furthermore, by the previous estimate on the slope of $b$, and by (\ref{dandt}), the curve $b$ have a angle satisfying the following,
\begin{equation}\label{dbminusda}
\frac{db}{dt}=\frac{da_n}{dt}-{n \lambda^{\theta n}}C\left[\frac{\theta }{\lambda} \frac{\partial \lambda}{\partial t}+\frac{1}{\mu}\frac{\partial\mu}{\partial t}\right] +O\left(\lambda^{\theta n}\right), 
\end{equation}
where $\left[\frac{\theta }{\lambda} \frac{\partial \lambda}{\partial t}+\frac{1}{\mu}\frac{\partial\mu}{\partial t}\right]$ is away from zero, see \eqref{eq:thetaconddlambdadmu}. 

 Let $(t_{n_0},a_{n_0})$ be any parameter in $\mathcal{B}_n$ corresponding to a map with a secondary tangency of type $n_0$, $q_{1,n,n_0}$. The proof is exactly the same as the previous case. In particular, by (\ref{dm6y1}), (\ref{dwn}), (\ref{dm6x1}), (\ref{alphaofuse}) (in the estimation of the order term)
\begin{eqnarray}\label{psi121}
\nonumber
\Psi_{1,2}&\geq&D(C-E)  \frac{n}{\mu}\frac{\partial\mu}{\partial t}\left[1+O\left(\frac{1}{n\mu^{\theta n/2}}\right)\right],
\end{eqnarray}
 and by (\ref{dm6yda1}), (\ref{dwn}), (\ref{dm6x1}), Lemma \ref{partialwn} and (\ref{alphaofuse}) (in the estimation of the order term)
\begin{eqnarray}
\nonumber
\Psi_{1,3}&\geq&D(C-E)\left[1+O\left(\frac{1}{\mu^{\theta n/2}}\right)\right].
\end{eqnarray}  
Moreover, by (\ref{dm6y2}), (\ref{dwn}), (\ref{dm6x1}), 
\begin{eqnarray}\label{psi122}
\nonumber
\Psi_{1,2}&\leq& \frac{DT}{\left(\lambda^{\theta}\mu\right)^n\mu^{n_0^*/2}}\frac{n}{\mu}\frac{\partial\mu}{\partial t}\left[1+O\left(\frac{1}{\mu^{n_0^*/2}}\right)\right].
\end{eqnarray}
Since, by Proposition \ref{speed} and Remark \ref{rem:dmudtpositive}, $D,C-E, T, {\partial\mu}/{\partial t}>0$, then $\Psi_{1,2}>0$, $\Psi_{1,3}>0$ and $\Psi_{2,1}> 0$, see Lemma \ref{curvature}. As before, because $D\Psi({0,0,0})$ is onto we get that the set of secondary tangencies $\Psi^{-1}(0)$ is locally the graph of a $\Cd$ function $b$.
Moreover,
\begin{equation}\label{eq:angle2}
\left|\frac{db}{dt}\right|=\left|\frac{\Psi_{1,2}}{\Psi_{1,3}}\frac{1}{\mu^n}\right|=O\left(\frac{n}{\left(\lambda^{\theta}\mu^2\right)^n\mu^{n_0^*/2}}\right),
\end{equation}
which goes to zero exponentially fast, see \eqref{thetacond}. The proof is complete.
 \end{proof}
\subsection{Properties of the loci of secondary tangencies}
We prove now that for any $n_0$ by moving in the direction of the $a$ parameter one can find at the most one tangency of type $n_0$. 
\begin{lem}\label{tangencyunicity}
 For $n$ large enough,
 $$
 \frac{\partial z_{3,y}}{\partial a}-\frac{\partial W_{n-n_0}}{\partial a}>0.
 $$
\end{lem}
\begin{proof}
Using the estimates in Proposition \ref{speed}, one notices that the leading term of $\partial z_{3,y}/{\partial a}$ is proportional to $\left(\lambda^{\theta}\mu^2\right)^n$ which is strictly larger than one, while $\partial W_{n-n_0}/{\partial a}$ is strictly smaller than one, see Lemma \ref{partialwn}, for $n$ large enough. The inequality follows.
\end{proof}

The following lemma shows that one cannot find a secondary tangency of type $n_0$ on the line defining the lower bound of the $\mathcal{B}_n$ strip. 
\begin{lem}\label{tangencylowerbound}
 For all $\epsilon>0$ and for $n$ large enough, there are no secondary tangencies of type $n_0$ along the line $a_n-E{\lambda^{\theta n}}$ where $E=C-\frac{\epsilon}{\mu^{n_0^*/2}}$.    
\end{lem}
\begin{proof}
By contradiction, there exists a $n_0$ tangency on the line $a_n-E{\lambda^{\theta n}}$ with $E=C-\frac{\epsilon}{\mu^{n_0^*/2}}$. This means that $z_{y}^{(3)}\mu^{n_0}=2+O(\lambda^{n_0})$. By Lemma \ref{hnbounds},
 \begin{eqnarray*}
z_{y}^{(3)}(t,a)= QH^2\left[1+O\left(H\right)\right]+z_{y}(t,a)=QH^2\left[1+O\left(H\right)\right]+a.
    \end{eqnarray*}
where $H=C\left(\lambda^{\theta}\mu\right)^n +\left(\mu^n+O(n)\right)\left(a-a_n(t)\right)=(C-E)\left(\lambda^{\theta}\mu\right)^n\left(1+O\left({n}/{\mu^n}\right)\right)$.
Hence, by \eqref{thetacond}, \eqref{eq:n0*def} and Lemma \ref{n0}
\begin{eqnarray}\label{eq:lowerboundarystrip}
z_{y}^{(3)}\mu^{n_0}&=&\nonumber Q(C-E)^2\left(\lambda^{\theta}\mu\right)^{2n}\mu^{n_0}+O\left(\frac{1}{\mu^{n-n_0}}\right)\\
&=&\nonumber Q\frac{\epsilon^2}{\mu^{n_0^*}}\left(\lambda^{\theta}\mu\right)^{2n}\mu^{n_0}+O\left(\frac{1}{\mu^{n-n_0}}\right)\\\nonumber
&=&O\left(\left(\lambda^{\theta}\mu\right)^{2n}\mu^{\alpha n-\frac{1}{2}\alpha n}\right)+O\left(\frac{1}{\mu^{n-n_0}}\right)
\\&=&O\left(\left(\lambda^{3\theta}\mu^{7/4}\right)^{2n}\right).
\end{eqnarray}
Since $\lambda^{3\theta}\mu^{7/4}<1$, see \eqref{thetacond}, this last quantity cannot be close to $2$. Contraddiction.The proof is complete.
\end{proof}
The following lemma shows that one cannot find a $n_0$ tangency on the line $a_n+\frac{T}{\mu^{n+n_0^*/2}}$ with $T$ a large constant.
\begin{lem}\label{tangencyupperbound}
 For $T$ large enough, there are no secondary tangencies of type $n_0$ along the line  $a_n+\frac{T}{\mu^{n+n_0^*/2}}$.    
\end{lem}
\begin{proof}
As in the proof of the previous lemma, By Lemma \ref{hnbounds},
 \begin{eqnarray*}
z_{y}^{(3)}(t,a)= QH^2\left[1+O\left(H\right)\right]+z_{y}(t,a)=QH^2\left[1+O\left(H\right)\right]+a.
    \end{eqnarray*}
where $H=C\left(\lambda^{\theta}\mu\right)^n +\left(\mu^n+O(n)\right)\left(a-a_n(t)\right)=\frac{T}{\mu^{n_0^*/2}}\left(1+O\left(\left(\lambda^{\theta}\mu\right)^n\mu^{n_0^*/2}\right)\right)$.
Hence, by \eqref{thetacond}, and Lemma \ref{n0}
\begin{eqnarray}\label{eq:upperboundarystrip}
z_{y}^{(3)}\mu^{n_0}>\frac{QT^2}{2\mu^{n_0}}\mu^{n_0}+O\left(\frac{1}{\mu^{n}}\mu^{n_0}\right)=\frac{QT^2}{2}+O\left(\frac{1}{\mu^{n-n_0}}\right).
\end{eqnarray}
For $T$ large enough, this last quantity cannot be close to $2$, so in particular one cannot have a $n_0$ tangency. The proof is complete.
\end{proof}
\subsection{Curves of secondary tangencies with uniform size}

In the next proposition we prove the existence of curves of secondary tangencies which are graphs of functions defined over the full domain $[-t_0,t_0]$, see Figure \ref{Fig6}.

\begin{prop}\label{angle} 
There exists a smooth function $b_{n,n_0}:[-t_0,t_0]\to\mathbb R$ such that for all $t\in [-t_0,t_0]$ the parameter point $(t,b_{n,n_0}(t))$ corresponds to a map with a secondary tangency of type $n_0$. Moreover, the graph of $b_{n,n_0}$ is contained in $\mathcal{B}_n$.
\end{prop}
\begin{proof} 
 Let $(t_{n_0}, a_{n_0})\in\mathcal{B}_n$ be a parameter corresponding to a map with a secondary tangency of type $n_0$. By Lemma \ref{bnlocal} there exist $\epsilon>0$ and a smooth curve $$b_{n,n_0}:[t_{n_0}-\epsilon,  t_{n_0}+\epsilon]\mapsto\mathbb R$$ such that, for all $t\in[t_{n_0}-\epsilon, t_{n_0}+\epsilon]$ the parameter $(t,b_{n,n_0}(t))$ corresponds to a map with a secondary tangency of type $n_0$. We want now to prove that this curve does not stop. Suppose it does. Let $(t_-,t_+)$ be the maximal interval in $[-t_0,t_0]$ to which this function extends smoothly such that its graph is a locus of secondary tangency of type $n_0$.
 Suppose that $[t_-,t_+]\subset (-t_0,t_0)$.
 Consider a sequence $t_n\in(t_-,t_+)$ of secondary tangencies of type $n_0$ that converges to a boundary point in $[t_-,t_+]$. Since the set of tangencies is a close set, all limit points are again a tangency. Because of Lemma \ref{tangencyunicity} one cannot have more than one limit point, indeed one can have only one tangency point while moving in the $a$ direction. By lemmas \ref{tangencylowerbound} and \ref{tangencyupperbound} this unique limit point cannot be on the boundary lines that define $\mathcal B_n$. Such a limit point is then contained in $\mathcal B_n$. One can then apply again Proposition \ref{bnlocal} to this point and build a smooth curve of maps with secondary tangencies of type $n_0$. This new courve $b'_{n,n_0}$ coincides with the initial one $b_{n,n_0}$, because of the uniqueness assured by the implicit function theorem used to create the local curve in Proposition \ref{bnlocal}. This contradicts the maximality of the interval  $(t_-,t_+)$.
In particular the function extends on the all domain $[-t_0,t_0]$ without exiting  $\mathcal B_n$.
\end{proof}
In the following lemma we prove that the estimate for the angle between the curve $b_{n,n_0}$ and $a_n$, see Lemma \ref{bnlocal}, persists in a neighborhood of $a_n$ of width $O\left(1/\mu^{2n}\right)$.
 \begin{lem}\label{bnlocalinAn}
For all $t\in [ -t_{0},  t_{0}]$, with $( b_{n,n_0}( t)-a_n(t))=O\left(1/\mu^{2n}\right)$,   
$$\frac{db_{n,n_0}}{dt}=\frac{da_n}{dt}+{n \lambda^{\theta n}}C\left[\frac{\theta }{\lambda} \frac{\partial \lambda}{\partial t}+\frac{1}{\mu}\frac{\partial\mu}{\partial t}\right]+O\left(\lambda^{\theta n}\right).$$
   \end{lem}
   \begin{proof}
  Observe that, by \eqref{thetacond}, \eqref{m5an} holds for all parameters $(t,a)$ with $a=a_n+O\left(1/\mu^{2n}\right)$. As a consequence \eqref{dm6yan} and \eqref{dm6ydan} hold as well and this implies that \eqref{psi12} and \eqref{psi13} also hold. In particular \eqref{dbminusda} holds for all parameters $(t,a)$ with $a=a_n+O\left(1/\mu^{2n}\right)$. The proof is complete.
   \end{proof}
In the next lemma we prove that there exists a neighborhood of the curve of secondary tangencies built above such that the initial family restricted to this neighborhood is an unfolding of the secondary tangency.  
\begin{lem}\label{lem:unfoldingsinneigh}
Let $b_{n,n_0}:[-t_0,t_0]\to\mathbb R$ be a curve of secondary tangencies. There exists a neighborhood $U$ of $b_{n,n_0}$ such that $F_{|U}$ can be reparametrized to become an unfolding.
\end{lem}
\begin{proof}
Let $U$ be a neighborhood of $b_{n,n_0}$ which is strictly contained in $\mathcal{B}_n$. Observe that the new family is a restriction of our original family and the secondary tangency curve $b_{n,n_0}$ describes homoclinic tangencies associated to the original saddle point. The transversal intersection $q_2$ it is still present in the new family. As a consequence, conditions $(f1), (f2), (f3), (f5), (f6), (F1), (F2), (F3), (F5)$ are automatically satisfied. Condition $(f4)$ is a consequence of 
Lemma \ref{curvature}.
By construction, condition $(f6)$ implies $(f7)$ for the secondary tangency $q_{1,n,n_0}$.
\\
 Use the notation from Proposition \ref{speed} and observe that $m_1\in W^u_{\text{loc}}(0)$ close to $q_3$ such that $F^{3N+\theta n+n+n_0}(m_1)=q_{1,n,n_0}$. This proves $(f8)$ for the secondary tangency $q_{1,n,n_0}$.
 \\
For proving $(F4)$ use the $\Cd$ function $b_{n,n_0}$ and observe that the maps on the graph of this curve, $b_{n,n_0}$, have a non degenerate homoclinic tangency.
\\
Observe that $(P1)$ follows by the fact that $b_{n,n_0}$ is the curve of non-degenerate tangencies, see Lemma \ref{curvature} and $(P2)$ from Lemma \ref{tangencyunicity}.
\end{proof}

\section{Newhouse phenomenon}

In this section, we identify parameters corresponding to maps that have infinitely many sinks of arbitrarily high period, a phenomenon referred to as the Newhouse phenomenon. Although this result itself is not new, our approach is constructive in nature and departs fundamentally from earlier methods, which in turn enables us to describe the dynamics on the closure of these sinks establishing a novel and unexpected property of unfoldings. We start by proving the existence of sinks in two-dimensional unfoldings. 
\subsection{Cascades of sinks}
In this section we are going to prove that, for a specific open set of parameters, the corresponding maps have a sink of high period. 
Choose $\epsilon_0>0$ and define, for $n$ large enough,
$$
\mathcal{A}_n=\left\{(t,a)\in [-t_0,t_0]\times [-a_0,a_0] \left|\right. |a-a_n(t)|\leq\frac{\epsilon_0}{\mu(t,a_n(t))^{2n}}\right\}.
$$
The strip $\mathcal{A}_n$ is built to contain the parameters having a sink. The size has to be chosen carefully. The following remark describes the error obtained in the change of the eigenvalues while changing the parameters within $\mathcal{A}_n$.
\begin{rem}\label{muzeroothersforan}
As in Remark \ref{muzeroothers}, for all $(t,\tilde a)\in\mathcal A_n$,
$$
\left(\frac{\mu(t,a_n(t))}{\mu( t,\tilde a)}\right)^n=1+O\left(\frac{n}{\mu(t,a_n(t))^{2n}}\right), \left(\frac{\lambda(t,a_n(t))}{\lambda(\tilde t,\tilde a)}\right)^n=1+O\left(\frac{n}{\mu(t,a_n(t))^{2n}}\right).
$$
\end{rem}

In the following we prove, for a properly chosen $\epsilon_0$, that for all parameters $(t,a)\in \mathcal{A}_n$, $F_{t,a}$ has a sink of period $n+N$ which we call {\it primary sink}. The method used appears already in \cite{ BP, Ro}. Namely, we find an invariant box in the phase space and we prove that the $F_{t,a}^{n+N}$ contracts this box. As a consequence, we get a sink. 

\bigskip

For $n$ large enough, take $(t,a)\in \mathcal{A}_n$ and denote by $c_n(t)=c_{t,a_n(t)}$ and by $z_n(t)=z_{t,a_n(t)}=(z_{n,x}(t),z_{n,y}(t))$. When the choice of $t$ is clear, we just use the notation $c_n$ and $z_n$. For $\delta>0$ we define the box that is going to contain the sink,  
$$
B^n_{\delta}(t,a)=\left\{(x,y)| |x-z_{n,x}(t)|\leq\frac{1}{3},|y-z_{n,y}(t)|\leq\frac{\delta}{\mu(t, a_n(t))^{2n}}\right\}.
$$
 When the choice of $(t,a)$ is clear we just use the notation $B^n_{\delta}$. In the next lemma we prove that $B^n_{\delta}$ returns into itself, see Figure \ref{Fig3}. Let $\tilde Q=\max Q_{x,y}$ where $Q_{x,y}$ is as in Lemma \ref{Nderivatives}. 
\begin{figure}[h]
\centering
\includegraphics[width=0.9\textwidth]{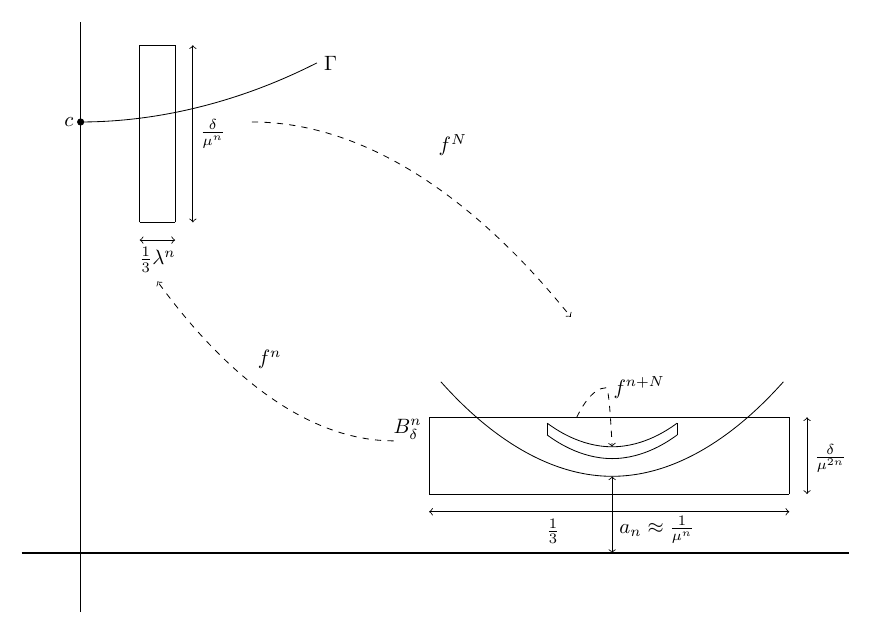}
\caption{Invariant Region}
\label{Fig3}
\end{figure}
\begin{lem}\label{periodicpoint}
Choose $\delta=\frac{1}{4\tilde Q}, \epsilon_0=\frac{1}{32\tilde Q}$. Then, for $n$ large enough and $(t,a)\in\mathcal A_n$,
$$
F^{n+N}_{t,a}\left(B^n_{\delta}\right)\subset B^n_{\delta}.
$$
\end{lem}
\begin{proof}
Fix $(t,a)\in\mathcal A_n$. For $n$ large enough, we write $F^{n+N}_{t,a}$ in coordinates centered at $z_n$, namely $F^{n+N}_{t,a}\left(\Delta x_{z_n},\Delta y_{z_n}\right)=\left(\Delta \tilde x_{z_n},\Delta \tilde y_{z_n}\right)$. Let $\left(\Delta x_{z_n},\Delta y_{z_n}\right)$ be such that if $z_n+\left(\Delta x_{z_n},\Delta y_{z_n}\right)\in B^n_{\delta}$ and $\Delta a=(a-a_n(t))$, then
\begin{equation}\label{deltas}
|\Delta x_{z_n}|\leq\frac{1}{3}, |\Delta y_{z_n}|\leq \frac{\delta}{\mu(t, a_n(t))^{2n}} \text{ and } |\Delta a|\leq \frac{\epsilon_0}{\mu(t, a_n(t))^{2n}}.
\end{equation}
Denote by $\left(\Delta x_{c_{t,a}},\Delta y_{c_{t,a}}\right)=F^n_{t,a}\left(\Delta x_{z_n},\Delta y_{z_n}\right)-c_{t,a}$. Using that $F^{n}_{t,a}$ is linear on $B^n_{\delta}$, see (\ref{Flinear}), we get
\begin{eqnarray*}\label{deltaxprime}
|\Delta x_{c_{t,a}}|\leq 2\lambda(t,a)^n,
\end{eqnarray*}
and 
\begin{eqnarray*}\label{deltayprime}
|\Delta y_{c_{t,a}}|&\leq &\mu(t, a)^{n}\left|\Delta y_{z_n}\right|+\left|\mu(t, a)^{n}z_{n,y}-c(0,t,a)\right|\\&\leq&\mu(t, a)^{n}\left|\Delta y_{z_n}\right|+\left|\left(\left[1+O\left(\frac{n}{\mu(t, a_n(t))^{2n}}\right)\right]\mu(t, a_n(t))^{n}z_{n,y}\right)-1\right|,
\end{eqnarray*}
where we used Remark \ref{muzeroothersforan}. Since $F^n_{t,a_n(t)}\left(z_n\right)\in\Gamma_{t,a_n(t)}$, then $$\mu(t, a_n(t))^{n}z_{n,y}=1+O\left(\lambda(t,a_n(t))^n\right).$$ Hence, by (\ref{deltas}) and by Remark \ref{muzeroothersforan} we get,
\begin{eqnarray*}\label{deltayprime}
|\Delta y_{c_{t,a}}|
&\leq &\mu(t, a)^{n}\left|\Delta y_{z_n}\right|+O\left(\lambda(t, a_n(t))^n\right)+O\left(\frac{n}{\mu(t, a_n(t))^{2n}}\right)\\&\leq &\frac{\delta}{\mu(t, a_n(t))^{n}}\left[1+O\left(\frac{n}{\mu(t, a_n(t))^{2n}}\right)\right]+O\left(\frac{n}{\mu(t, a_n(t))^{2n}}\right).
\end{eqnarray*}
By Lemma \ref{Nderivatives} (center $F^{N}_{t,a}$ in coordinates around $c_{t,a}$) extended to include also the Taylor expansion in $\Delta a$   
we get, for $n$ large enough
\begin{eqnarray*}
|\Delta\tilde x_{z_n}|&=&O\left(\Delta x_{c_{t,a}}\right)+O\left(\Delta y_{c_{t,a}}\right)+O\left(\Delta a\right)\leq\frac{1}{3},
\end{eqnarray*}
and 
\begin{eqnarray*}
|\Delta\tilde y_{z_n}|&\leq &O\left(\lambda(t,a)^n \right)+Q_{x,y}\left|\Delta y_{c_{t,a}}\right|^2+O\left(\frac{1}{\mu(t, a_n(t))^{3n}}\right)+|\Delta a|\\
&\leq &\frac{\tilde Q \delta^2}{\mu(t, a_n(t))^{2n}}+O\left(\frac{n}{\mu(t, a_n(t))^{3n}}\right)+\frac{\epsilon_0}{\mu(t, a_n(t))^{2n}},
\end{eqnarray*}
where we also used $(F3)$ and Remark \ref{muzeroothersforan}.
By our choice of $\epsilon_0$ and $\delta$, for $n$ large enough, the lemma follows.
\end{proof}

We fix $\epsilon_0,$ and $\delta$ such that Lemma \ref{periodicpoint} holds. 
We are now ready to prove the existence of a sink. This is achieved in the next proposition by proving that $F^{n+N}$ contracts the box $B_{\delta}^n$.
\begin{prop}\label{sink}
For $n$ large enough and for all $(t,a)\in \mathcal A_n$, $ B^n_{\delta}(t)$ has a unique periodic point which is a sink of period $n+N$. 
\end{prop}
\begin{proof}
Because $F_{t,a}^n$ is linear on $B_{\delta}^n$, the image $F_{t,a}^n\left(B_{\delta}^n\right)$ is contained in a neighborhood of $c_{t,a}$ of diameter smaller than  $\delta\mu^{-n}\left(1+O\left({n}\mu^{-n}\right)\right)
$. From this and by differentiating (\ref{NstepsTaylorformula}) with respect to $\Delta x$ and $\Delta y$ we get
\begin{eqnarray*}
DF_{t,a}^N&=&\left(\begin{matrix}
O(1) &O(1) \\
O(1) & 2{\delta\mu^{-n} Q_{x,y}}\left(1+O\left({n}{\mu^{-n}}\right)\right)\\
\end{matrix}\right).
\end{eqnarray*}
Note that $N$ is fixed. Using again that $F_{t,a}^n$ is linear on $B_{\delta}^n$ we obtain that 
\begin{eqnarray*}
DF_{t,a}^{n+N}&=&\left(\begin{matrix}
O(1) &O(1) \\
O(1) &{2\delta }{\mu^{-n}}Q_{x,y}\left(1+O\left({n}{\mu^{-n}}\right)\right)
\end{matrix}\right)\left(\begin{matrix}
\lambda^n &0 \\
0& \mu^n
\end{matrix}\right)\\
&=&\left(\begin{matrix}
O\left(\lambda^n\right) & O\left(\mu^n\right) \\
O\left(\lambda^n\right)& 2\delta Q_{x,y}\left(1+O\left({n}{\mu^{-n}}\right)\right)
\end{matrix}\right).
\end{eqnarray*}
Let $D=\left(\begin{matrix}
O\left(\lambda^n\right) & O\left(\mu^n\right) \\
O\left(\lambda^n\right)& \frac{3}{4}
\end{matrix}\right)$ be a positive matrix and $\left(DF_{t,a}^{n+N}\right)^k(\Delta x,\Delta y)=(\Delta x_k,\Delta y_k)$, then by the choice of $\delta$
\begin{eqnarray*}
\left(\begin{matrix}
|\Delta x_{k+1}| \\
|\Delta y_{k+1}| 
\end{matrix}\right)\leq D\left(\begin{matrix}
|\Delta x_{k}| \\
|\Delta y_{k}| 
\end{matrix}\right).
\end{eqnarray*}
Observe that $\text{tr}(D)=\frac{3}{4}+O\left(\lambda^n\right)$ and $\text{det}(D)=O(\left(\lambda\mu\right)^n)$. As consequence, for $n$ large enough, $|\Delta x_{k}|, |\Delta y_{k}|\to 0$ exponentially fast. Hence, the periodic point in $B^n_{\delta}$ is a sink whose basin of attraction contains $B^n_{\delta}$.
\end{proof}

\subsection{Newhouse phenomenon: infinitely many sinks}\label{sec:Newhouse}
We are now ready to set up an induction procedure to prove the existence of maps having infinitely many sinks. This procedure uses ``Newhouse boxes". In the first generation, the Newhouse boxes are essentially rectangles in $\mathcal A_n$ whose boundary are defined by the curves of secondary tangencies, see Figure \ref{Fig6}. The family restricted to the Newhouse boxes of first generation have one sink and it is an unfolding of a new homoclinic tangency. The propositions and the lemmas proved in the previous sections apply then to these families creating Newhouse boxes of second generation. As a consequence, the family restricted to the Newhouse boxes of second generation have two sinks and it is an unfolding of a new homoclinic tangency. We proceed by induction.

Let $(t_{n,n_0},a_{n,n_0})$ be the parameters at the intersection point $ a_n\cap b_{n,n_0}$, see Proposition \ref{angle}, and let $$f_{n,n_0}=F_{t_{n,n_0},a_{n,n_0}}.$$ Recall that $b_{n,n_0}$ is the graph of a function  $b=b_{n,n_0}:[-t_{0},t_{0}]\mapsto\R$.  Consider its restriction to the interval 
$[t^{-}_{n,n_0}, t^{+}_{n,n_0}]$ where $t^{-}_{n,n_0}$ and $t^{+}_{n,n_0}$ are the two intersection points of the curve $b$ with the boundary of the strip $\mathcal A_n$, see Figure \ref{Fig6}. The existence of these intersection points is guaranteed by Lemma \ref{bnlocalinAn}. The domains,
$$
\mathcal {P}_{n,n_0}=\left\{(t,a)\left|\right.  t\in[t^{-}_{n,n_0}, t^{+}_{n,n_0}], |a-a_n(t)|\leq\frac{\epsilon_0}{\mu(t,a_n(t))^{2n}}\right\},
$$
are called the Newhouse boxes of first generation. 
\bigskip

The construction in the previous sections started with a map $f:\mathcal M\to \mathcal M$ with a strong homoclinic tangency and an unfolding $F:\mathcal P\times \mathcal M\to \mathcal M$. The following inductive construction will repeat the discussion of the previous sections starting with the map $f_{n,n_0}:\mathcal M\to \mathcal M$ and an unfolding $F:\mathcal P_{n,n_0}\times \mathcal M\to \mathcal M$ which is the restriction of the original family.
\begin{lem}\label{disjP}
The domains $\mathcal {P}_{n,n_0}$ are pairwise disjoint for $n$ large enough and the diameter goes to zero.
\end{lem}
\begin{proof}
By Proposition \ref{angle} the curve $ b_{n,n_0} $ is the graph of the function $b$ and by the angle estimate in Lemma \ref{bnlocalinAn}, there exists a uniform constant $K>0$ such that
\begin{equation}\label{horizsize}
\frac{1}{K}\frac{1}{n\left(\lambda^{\theta}\mu^2\right)^n}\leq \left| \left\{t\left|\right. (t,b(t))\in\mathcal A_n\right\}\right|\leq K\frac{1}{n\left(\lambda^{\theta}\mu^2\right)^n}.
\end{equation}
Moreover the proof of Proposition \ref{newtangency} gives $\text{dist}\left(f_{n,n_0},f_{n,n_0+1}\right)$ is proportional to ${1}/{n}$. The disjointness follows from this estimates, the fact that $\lambda^{\theta}\mu^2\geq \left(\lambda^{2\theta}\mu^3\right)^{1/2}>1$, see (\ref{thetacond}) and the fact that $\mathcal A_n$ are pairwise disjoint.
\end{proof}
The next proposition ensures that the family restricted to $\mathcal P_{n,n_0}$ is again an unfoldings. This allows an inductive procedure.
\begin{prop}\label{induction}  
For $n$ large enough, the map $f_{n,n_0}$ has a strong homoclinic tangency and the restriction $F:\mathcal P_{n,n_0}\times \mathcal M\to\mathcal M$ can be reparametrized to become an unfoldings. Moreover each map in $\mathcal P_{n,n_0}$ has a sink of period $n+N$.
\end{prop}
\begin{proof}
The proof is given in Lemma \ref{lem:unfoldingsinneigh} and Proposition \ref{sink}.
\end{proof}
Inductively we are going to construct parameters with multiple sinks of higher and higher periods and a strong homoclinic tangency.
Let $\mathfrak N^1$ be the set of labels of the Newhouse boxes of first generation, i.e. all possible pairs $(n,n_0)$ for which there exists a Newhouse box $\mathcal P_{n,n_0}$. As inductive hypothesis assume that there exist sets $\mathfrak N^k\subset \N^2$ such that, each  $\mathfrak N^k$ is the set of labels of the Newhouse boxes of generation $k$. Moreover the natural projections $\mathfrak N^1\leftarrow\mathfrak N^2\leftarrow\dots\leftarrow\mathfrak N^g $ with $\mathfrak N^k\in\N^2$ are countable to $1$ and correspond to inclusion of Newhouse boxes of successive generations. They satisfy the following inductive hypothesis.

 For $\underline n=\left\{(n^{(k)},n_0^{(k)})\right\}_{k=1}^g$, there exist sets $\mathcal P^k_{n^{(k)},n_0^{(k)}}\subset\mathcal P$, where $n^{(k)}$ labels the sink and  $n_0^{(k)}$ labels the secondary tangency, such that
\begin{itemize}
\item $\text{diam}\left(\mathcal P^k_{n^{(k)},n_0^{(k)}}\right)\leq\frac{1}{k}$,
\item $\mathcal P^{k+1}_{n^{(k+1)},n_0^{(k+1)}}\subset \mathcal P^k_{n^{(k)},n_0^{(k)}}$ for $k=1,\dots,g-1$,
\item there exists a map $f^k_{n^{(k)},n_0^{(k)}}\in\mathcal P^k_{n^{(k)},n_0^{(k)}}$ which has a strong homoclinic tangency,
\item the restriction $F:{\mathcal P^k_{n^{(k)},n_0^{(k)}}}\times \mathcal M\to\mathcal  M$ can be reparametrized to become and unfoldings of $f^k_{n^{(k)},n_0^{(k)}}$,
\item every map in $\mathcal P^k_{n^{(k)},n_0^{(k)}}$ has at least $k$ sinks of different periods.
\end{itemize}
By induction, using Proposition \ref{induction},  we get an infinite sequence of sets $\mathfrak N^k$. 
\begin{figure}
\centering
\includegraphics[width=0.73\textwidth]{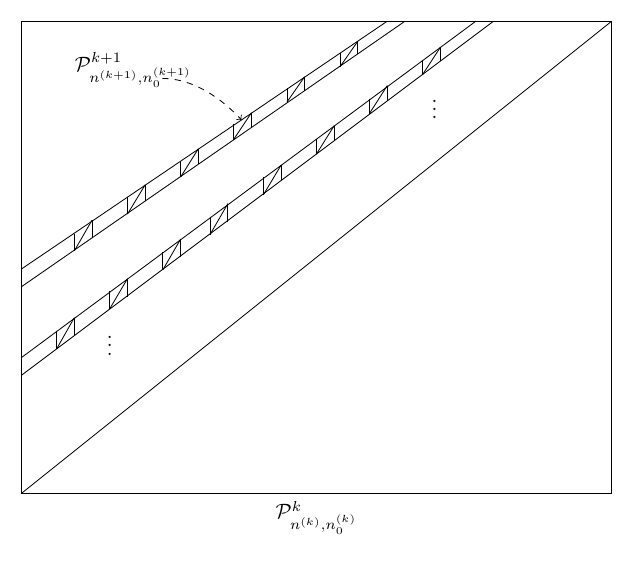}
\caption{Newhouse boxes}
\label{Fig7}
\end{figure}
\begin{defin}
The sets $\mathcal P^g_{n^{(g)},n_0^{(g)}}$ are called Newhouse boxes of generation $g$, see Figure \ref{Fig7}, and 
$$NH=\bigcap_g\bigcup_{\mathfrak N^g}\mathcal P^g_{n^{(g)},n_0^{(g)}}.$$ 
\end{defin}

The set $NH$, consisting of parameters for which the corresponding maps have infinitely many sinks, accumulates on the curve of the original tangency. The inductive construction of these parameters implies that the set $NH$ accumulates on all tangency curves given by $b_{n,n_0}$. We have the following lemma.

\begin{lem}\label{pacc}
\begin{equation*}
\overline{NH}\supset [-t_0,t_0]\times\left\{0\right\}.
\end{equation*}
In particular,
\begin{equation}\label{NHinter}
\overline{NH}\cap \mathcal {P}_{n,n_0}\supset\text{\rm graph}(b_{n,n_0})\cap \mathcal {P}_{n,n_0}.
\end{equation}
\end{lem}
\begin{proof}
Given $(t,0)$, by Proposition \ref{newtangency}, for every $n$ large enough, $a_n$ has a secondary tangency at $(t_n,a_n(t_n))$ with $|t-t_n|=O\left(\frac{1}{n}\right)$. Hence, there exists a sequence of Newhouse boxes in $\mathfrak{N_1}$ accumulating at $(t,0)$. By construction, each box in $\mathfrak{N_1}$ contains points of $NH$. The lemma follows.
\end{proof}
Given any family $F$ of diffeomorphisms, we define the Newhouse set $NH_F$ as the set of parameters having infinitely many sinks.
The upper Minkowski dimension is denoted by $MD$.

%

\begin{theo}\label{Newhousepoints}
Let $F:\mathcal P\times\mathcal M\to\mathcal  M$ be an unfolding then 
\begin{itemize}
\item[-] $NH\subset NH_F$, every map in $NH$ has infinitely many sinks,
\item[-] $NH$ is homeomorphic to $\R\setminus\Q$,
\item[-]$MD(NH)\geq\frac{2}{3}$. 

\end{itemize}
\end{theo}
\begin{proof} The inductive construction, using  Proposition \ref{induction}, implies that all maps in $NH$ have infinitely many sinks. From  Lemma \ref{disjP} we know that 
$$
\bigcup_{\mathfrak N^g}\mathcal P^g_{n^{(g)},n_0^{(g)}},
$$
consists of countably many disjoint boxes. Each box $P^g_{n^{(g)},n_0^{(g)}}$ contains countably many pairwise disjoint boxes of the next generation. Hence, the nested intersection $NH$ is homeomorphic to $\R\setminus\Q$, see \cite{Gaal}. 

For the last property, let $(t^*,0)$ such that ${\log{\lambda(t^*,0)^{-1}}}/{\log\mu(t^*,0)}$ is the maximum of ${\log{\lambda(t,0)^{-1}}}/{\log\mu(t,0)}$. Consider a sequence of first generation Newhouse boxes $\mathcal {P}_{n,n_0}\in\mathfrak{N_1}$ accumulating at $(t^*,0)$. This is possible because of Lemma \ref{pacc}. Choose $\epsilon>0$ and let $n$ be maximal such that $\epsilon\leq {\epsilon_0}/{\mu(t^*,a^*)^{2n}}$. 
 Because (\ref{horizsize}), (\ref{NHinter}) and  the fact that the vertical size of $\mathcal {P}_{n,n_0}$ is ${\epsilon_0}/{\mu(t^*,a^*)^{2n}}$, we need at least ${K}/{\lambda(t^*,a^*)^{\theta n}}$ balls of radius $\epsilon$ to cover $\overline{NH}\cap \mathcal {P}_{n,n_0}$. As consequence $$MD(NH)\geq \frac{\theta}{2}\max_{NH}\left(\frac{\log\frac{1}{\lambda}}{\log\mu}\right),$$
 and $MD(NH)\ge{2}/{3}$, 
  where we used  (\ref{thetacond1}).
\end{proof}
\begin{rem}
Observe that the estimate for the upper Minkowski dimension is not sharp.  Other dimension estimates for maps with infinitely many sinks were obtained in \cite{BDS, TY, W}. 
\end{rem}
%

\section{Dynamical Partitions, Directed Graphs and Strips}
In this section, we introduce the objects and the induction step (see Proposition \ref{Prop:inductionstep}) that will allow to build later the main player of this work: the lamination and the invariant Cantor set.

\subsection{Dynamical Partitions}
We give now the definition of a dynamical partition associated to a map. The reader can refer to Figure \ref{dynamicalpartition}. The definition models the combinatorics as described in Subsection \ref{subsection:combinatorics}.

\begin{figure}
\centering
\includegraphics[width=0.6\textwidth]{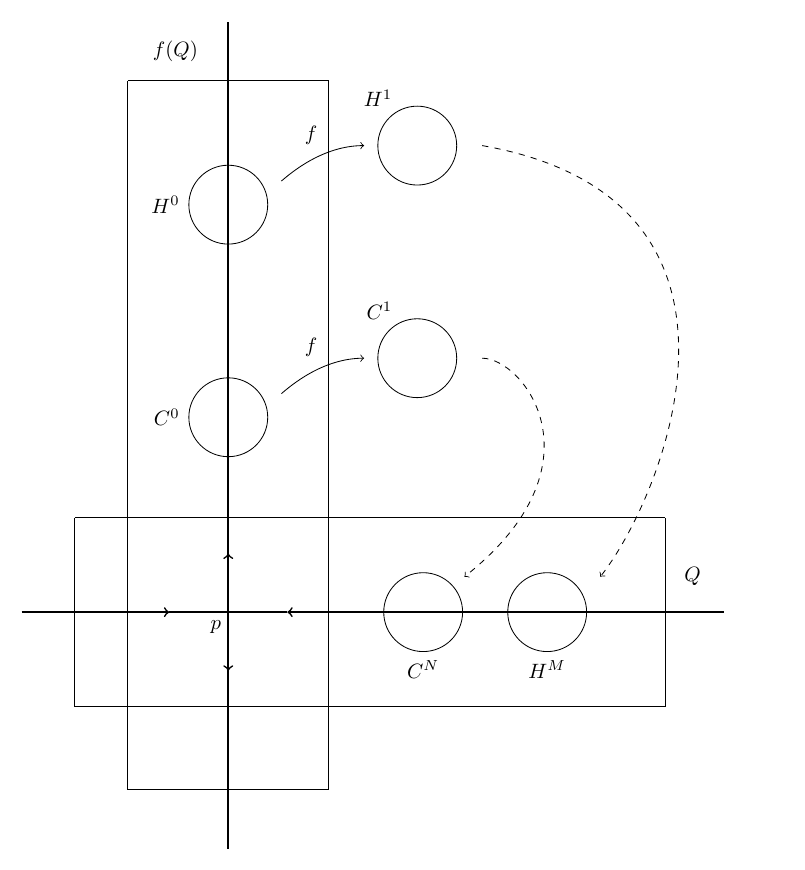}
\caption{$(1,N,M)$ dynamical partition of $f$}
\label{dynamicalpartition}
\end{figure}
Recall the rescaling done in Subsection \ref{RescaledFamilies} while introducing our unfoldings and let $Q$ be the domain of linearization. We denote by $W^u_{\text{loc}}(p)$ the connected component of $W^u(p)\cap Q$ and by $W^s_{\text{loc}}(p)$ the connected component of $W^s(p)\cap Q$ both containing $p$.
\begin{defin}\label{def:dynamicalpartition}
Let $\mathcal  M$ be a $m$-dimensional manifold. A map $f:\mathcal M\to\mathcal  M$ has a $(1,N,M)$ dynamical partition 
$$
\mathfrak{D}=\left\{Q, C^0,C^1,\dots, C^{N-1}, H^0,H^1,\dots,H^{M-1}\right\},
$$
where $N$ and $M$ are positive natural numbers, $Q, C^i, Hj$ are open sets for all $i\in\left\{0,\dots,N-1\right\}$ and $j\in\left\{0,\dots,M-1\right\}$ if the following holds.
\begin{enumerate}
    \item The elements of $\mathfrak{D}$ are pairwise disjoints.
    \item The map $f$ has a saddle point $p\in Q$.
    \item The $C$-sets satisfy the followings:
    \begin{itemize}
        \item $C^0\subset f(Q)\setminus Q$,
        \item $f:C^k\to C^{k+1}$, for all $k\in\left\{0,1,\dots, N-2\right\}$ is a diffeomorphism,
        \item $f:C^{N-1}\to f\left(C^{N-1}\right)=C^N$ is a diffeomorphism,
        \item $C^{N}\subset Q\setminus f(Q)$,
        \item $C^0\cap f\left(W^u_{\text{loc}}(p)\right)\neq\emptyset$,
     \item $C^N\cap W^s_{\text{loc}}(p)\neq\emptyset$.
    \end{itemize}
     \item The $H$-sets satisfy the followings:
    \begin{itemize}
        \item $H^0\subset f(Q)\setminus Q$,
        \item $f:H^k\to H^{k+1}$, for all $k\in\left\{0,1,\dots, M-2\right\}$ is a diffeomorphism,
        \item $f:H^{M-1}\to f\left(H^{M-1}\right)=H^M$ is a diffeomorphism,
        \item $H^{M}\subset Q\setminus f(Q)$,
        \item $H^0\cap f\left(W^u_{\text{loc}}(p)\right)\neq\emptyset$,
     \item $H^{M}\cap W^s_{\text{loc}}(p)\neq\emptyset$,
     \item $f^M\left(H_0\cap f\left(W^u_{\text{loc}}(p)\right)\right)$ intersects transversally $W^u_{\text{loc}}(p)$ in the point $q_2$.
    \end{itemize}
\end{enumerate}
\end{defin}
\begin{defin}
    Let $\mathcal M$ be a $m$-dimensional manifold, $f:\mathcal M\to \mathcal M$ a map and let $\mathfrak{D}$ be the $(1,N,M)$ dynamical partition of $f$. We define the \it{mesh} of $\mathfrak{D}$ as
    $$\text{mesh}\left(\mathfrak{D}\right)=\max_{U\in\mathfrak{D}}\left(\text{diam(U)}\right), $$
    where $\text{diam(U)}$ is the diameter of $U$. Moreover we denote by $\mathcal{D}$ the union of elements in $\mathfrak{D}$. Namely,
    $$\mathcal{D}=\bigcup_{U\in\mathfrak{D}}U.$$
   
\end{defin}

\subsection{Directed Graphs}
Fixed two positive natural numbers, $N$ and $M$, we define the $(1,N,M)$ directed graph.
The reader can refer to Figure \ref{directedgraph}.  As before this models the combinatorics as described in Subsection \ref{subsection:combinatorics}.

\begin{defin}
   Given two positive natural numbers, $N$ and $M$, a directed graph $X$ is called the $(1,N,M)$ directed graph if it has $1+N+M$ vertices, $q,c^0,\dots,c^{N-1}, h^0,\dots,h{M-1}$ and unique edges:
   \begin{eqnarray*}
        \overrightarrow{qq}&&\\ \overrightarrow{qh^0}& \overrightarrow{h^kh^{k+1}} \text{ for all }k\in\left\{1,\dots,M-2\right\} & \overrightarrow{h^{M-1}q}
        \\\overrightarrow{qc^0}& \overrightarrow{c^kc^{k+1}} \text{ for all }k\in\left\{1,\dots,N-2\right\} & \overrightarrow{c^{N-1}q}
        \end{eqnarray*}
        The sub graph formed by $q$ is called the $q$-loop and is denote by $\lambda^q$.\\
        The sub graph formed by $q,h^0,\dots, h^{M-1}$ is called the $h$-loop and is denote by $\lambda^h$.\\
        The sub graph formed by $q,c0,\dots, c^{M-1}$ is called the $c$-loop and is denote by $\lambda^c$.\\
        The shortest path from the vertex $v_1$ to the vertex $v_2$ is denoted by $ \overrightarrow{v_1v_2}$.
\end{defin}

\begin{rem}
    Observe that $q$ has three ingoing and outgoing edges. All other vertices have exactly one ingoing and one outgoing edge. In particular each vertex other than $q$ has a unique predecessor and a unique successor.
\end{rem}
   Let $X$ be the $(1,N,M)$ directed graph. A loop $\lambda$ is a path in $X$ which starts and ends in $q$. The loops $\lambda^q$, $\lambda^h$ and $\lambda^c$ are called the prime loops of $X$, respectively $q$-loop, $h$-loop and $c$-loop. Observe that a loop $\lambda$ is a concatenation of prime loops $\gamma_i$, namely,
   $$
   \lambda=\gamma_u\ast\cdots\ast\gamma_1.
   $$
The winding of $\lambda$ is the word
$$
\omega(\lambda)\in\left\{q,h,c\right\}^n,
$$
such that $\gamma_i=\lambda_{\omega(\lambda)(i)}$ for all $i\in\left\{1,\dots,u\right\}$.

\begin{figure}
\centering
\includegraphics[width=0.6\textwidth]{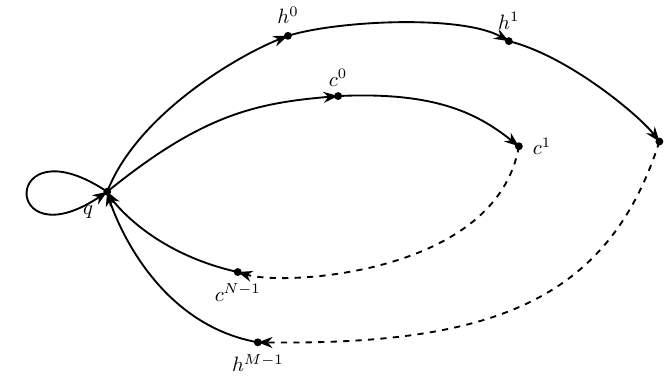}
\caption{$(1,N,M)$ directed graph}
\label{directedgraph}
\end{figure}
\paragraph{Graph Morphisms.} Given $N,M,\tilde N,\tilde M$ positive natural numbers, consider the $(1,N,M)$ directed graph $X$, the $(1,\tilde N,\tilde M)$ directed graph $\tilde X$ and assume there is a graph morphism
$$w:\tilde X\mapsto X.$$
Observe that, because $q$ is the only vertex in $X$ which has an edge from itself to itself then $w(\tilde q)=q$ and that the image of every loop in $\tilde X$ is a loop in $X$. In particular, the graph morphism $w$ is completely determined by the windings of the primary loops, 
$$
\underline\omega(w)=\left(\omega\left(w(\lambda_{\tilde q})\right),\omega\left(w(\lambda_{\tilde h})\right), \omega\left(w(\lambda_{\tilde c})\right)\right).$$
\paragraph{Projections.}

In the following lemma we introduce the {\it projection} of $D$ onto $X$. 
\begin{lem}
Let $f$ be a map. Consider $\mathfrak{D}$ the $(1,N,M)$ dynamical partition of $f$ and $X$ the $(1,N,M)$ directed graph $X$. There exists a map $\pi:\mathcal{D}\mapsto X$ satisfying the following properties:
\begin{enumerate}
    \item for all vertices $v\in X$, $\pi^{-1}(v)\in\mathfrak{D}$,
    \item if $x\in \mathcal{D}$ and $f(x)\in \mathcal{D}$ $\implies$ there exists an edge $\overrightarrow{\pi(x)\pi\left(f(x)\right)}$,
    \item if there is an edge $\overrightarrow{v_1v_2}$ $\implies$ there exists $x\in\pi^{-1}(v_1)$ such that $f(x)\in\pi^{-1}(v_2)$.
 
\end{enumerate}
The map $\pi$ is called the projection of $\mathcal{D}$ on $X$. 
\end{lem}
\subsection{Strips}
We introduce strips in the plane that will enable us later to construct, inductively, the object of primary interest to us: the lamination. Observe that the strip are an abstract generalization of the $\mathcal B_n$ strips built earlier to contain the curves of secondary tangencies. 
We start by defining $r$-tangency strips, with $r$ a positive number; see Figure \ref{strips}.

\begin{defin}
    Let $r>0$, $r\in\mathbb N$ and let $F$ be an unfolding. An $r$-tangency strip $S$ is an open region in $[-t_0,t_0]\times[-a_0,a_0]$ satisfying the following:
    \begin{itemize}
        \item $S$ is bounded by the graph of two functions $s_l,s_u:[-t_0,t_0]\to\mathbb R$ such that, for all $t\in[-t_0,t_0]$,
        \begin{eqnarray*}
            \left|s_l(t)-s_u(t)\right|\leq C\rho^{r},
        \end{eqnarray*}
        where $0<\rho<1$, $\rho<\max_t\left\{\frac{1}{\mu(t)}\right\}$, and $C$ is a positive constant.
        \item $S$ contains $r$ curves $\left\{b_k\right\}_{k=1}^r$, which are graph of functions $b_k:[-t_0,t_0]\mapsto\mathbb R$ for all $k$, such that the following holds: 
        \begin{itemize}
            \item[(i)] $\forall$ $k\in\left\{1,\dots,r\right\}$ and for all $t\in[-t_0,t_0]$, $F_{t,b_k(t)}$ is a map with a strong homoclinic tangency,
            \item[(ii)] $\forall$ $k\in\left\{1,\dots,r\right\}$ there exists a neighborhood of $b_k$, $U_k$, such that $F_{|U_k}$ is an unfolding of $F_{t,b_k(t)}$,
            \item[(iii)] $\forall$ $k,k'\in\left\{1,\dots,r\right\}$, $U_k\cap U_{k'}=\emptyset$,
            \item[(iv)] $\forall$ $k\in\left\{1,\dots,r\right\}$ there exist positive numbers $N_k,M_k$ such that, for all $f\in U_k$, the map $f$ has a $(1,N_k,M_k)$ dynamical partition $\mathfrak{D}_k(f)$ which depends continuously on the parameters.

        \end{itemize}
    \end{itemize}
    For all $k\in\left\{1,\dots,r\right\}$, the pair $\left\{b_k,U_k\right\}$ is called a $(1,N_k,M_k)$ descendant of $S$.
\end{defin}
\begin{rem}\label{rem:renormaps}
Let $f\in S$, then $f$ has a $(1,N,M)$ dynamical partition. We identify $f$ with its restriction to this dynamical partition. If in particular $f\in U_k$ for some $k$, then it has a finer dynamical partition $(1,N_k,M_k)$ which describes the dynamics on a smaller scale.  For this reason we call the restriction of $f$ to the finer dynamical partition the renormalization of $f$. 
\end{rem}

\begin{figure}
\centering
\includegraphics[width=1\textwidth]{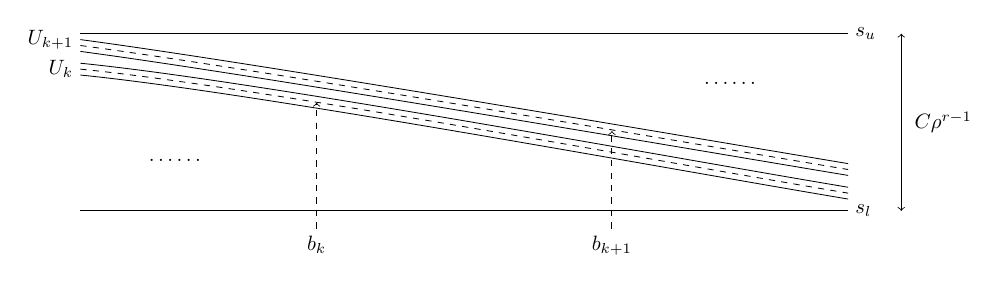}
\caption{A $r$-tangency strip}
\label{strips}
\end{figure}
\subsection{The induction step}
In this subsection we present the induction step that will allow later to build the lamination and the invariant Cantor set.

\begin{prop}\label{Prop:inductionstep}
    Let $S$ be a $r$-tangency strip and let $\left\{b,U\right\}$ be a $(1,N,M)$ descendant of $S$. Then there exist $C>0$ and $n\in\mathbb N$, $n>>r$ large enough and a countable sequence of regions $\left\{S_j\right\}_{j=1}^{\infty}\subset U\subset S$ satisfying the following:
    \begin{itemize}
        \item[(i)] $\forall j$, $S_j$ is a $(\sqrt{n+j})$-tangency strip,
        \item[(ii)]$\forall j,j'$ with $j\neq j'$, $S_j\cap S_j'=\emptyset$.    \end{itemize}
        Moreover, for all $j$ and for all $k\in\left\{1,\dots,\sqrt{n+j}\right\}$ let $\left\{b_k^{(j)}, U_k^{(j)} \right\}$ be a $\left(1, N_k^{(j)}, M_k^{(j)}\right)$ descendant of $S_j$. Then the following holds:
        \begin{itemize}
            \item[(iii)] $\forall t\in\left[-t_0,t_0\right]$,
            \begin{eqnarray*}
               \left|b_k^{(j)}(t)-b(t)\right|\leq\frac{C}{\mu(t)^{n+j}},
            \end{eqnarray*}
            \item[(iv)]$\forall t\in\left[-t_0,t_0\right]$,
            \begin{equation}\label{eq:angleall}
\left|\frac{db_k^{(j)}}{dt}(t)-\frac{db}{dt}(t)\right|\leq C\frac{n+j}{\left(\lambda(t)^{\theta}\mu(t)^2\right)^{n+j}},  \end{equation}           
\item[(v)] $\forall f\in U_k^{(j)}\subset U$, let $\mathfrak{D}^{(j)}_k(f)$ be the $\left(1, N_k^{(j)}, M_k^{(j)}\right)$ dynamical partition of $f$ and let $\mathfrak{D}(f)$ be the $\left(1, N, M\right)$ dynamical partition of $f$. Then $\mathfrak{D}^{(j)}_k(f)$ is a refinement of $\mathfrak{D}(f)$ and 
            \begin{eqnarray*}
                \text{mesh}\left(\mathfrak{D}^{(j)}_k(f)\right)<\frac{1}{2}\cdot\text{mesh}\left(\mathfrak{D}(f)\right).
            \end{eqnarray*}
        \end{itemize}
        Furthermore, if $X$ is the $\left(1,N,M\right)$ directed graph, $X_k^{(j)}$ is the $\left(1,N_k^{(j)}, M_k^{(j)}\right)$ directed graph, $i:\mathcal{D}_k^{(j)}\mapsto \mathcal{D}$ is the inclusion, $\pi:\mathcal{D}\mapsto X$ is the projection of $\mathcal{D}$ on $X$ and $\pi_k^{j}:\mathcal{D}_k^{(j)}\mapsto X_k^{(j)}$ is the projection of $\mathcal{D}_k^{(j)}$ on $X_k^{(j)}$, there exists a unique graph morphism $w_k^{(j)}:  X_k^{(j)}\mapsto X$ such that 
        \begin{eqnarray*}
            \pi\circ i_k^{(j)}=w_k^{(j)}\circ\pi_k^{(j)}.
        \end{eqnarray*}
        
      \end{prop}
\begin{proof}
Without loss of generality, we can assume that the $r$-strip $S$ is bounded by straight lines. In particular, we can make a reparametrization  to ensure that the curve $b$ corresponds to a straight line in the $t$-direction.
For all $j$, consider the strips $\mathcal B_{n+j}$ as defined in \eqref{eq:Bn} and take $S_j=\mathcal B_{n+j}$. Observe that for $n$ large enough $S_j\subset U\subset  S$ and by construction $S_j\cap S_{j'}=\emptyset$ for all $j\neq j'$. Moreover Section $4$ and Section $5$ ensure also that $\mathcal B_{n+j}$ contains $\sqrt{n+j}$ secondary tangency curves. The descendants are the pairs $\left\{b_{k}^{(j)}, U_{k}^{(j)}\right\}$ with $k\in\left\{1,\dots,\sqrt{n+j}\right\}$ where $b_{k}^{(j)}$ is the curve of secondary tangencies as constructed in Proposition \ref{angle}  and $U_{k}^{(j)}$ is a neighborhood of $b_{k}^{(j)}$ as built in Lemma \ref{lem:unfoldingsinneigh}. In particular, the family $F_{|U_{k}^{(j)}}$ is an unfolding of the homoclinic tangency $b_{k}^{(j)}$. Given a descendent of $S_j$, $\left\{b_{k}^{(j)}, U_{k}^{(j)}\right\}$, observe that  $a_{n+j}, b_{k}^{(j)}\subset S_j=\mathcal B_{n+j}$ and the width of $\mathcal B_{n+j}$ is smaller than the distance between $a_n$ and the previous tangency $b$ that is indeed of then order of $1/\mu^{n+j}$, see \eqref{munan}. This proves $(iii)$. Condition $(iv)$ on the derivative comes from \eqref{eq:angle2}. 
The final step is to construct the corresponding dynamical partitions. Let $f\in b_k^{(j)}$. In particular there is a point $m_1\in W^u_{\text{loc}}(p)$ and a time $s$ such that at $f^{s+M}(m_1)$ there is a tangency of $ W^s(p)\cap  W^u_{\text{loc}}(p)$, see Section $4$. Let $Q_k^{(j)}\subset Q$ such that 
$$
\overline{Q_k^{(j)}}\cap\left\{m_1,f(m_1),\dots,f^{s+M}(m_1)\right\}=\emptyset.
$$
Let $\overline{\ell}_k^{(j)}$ be minimal such that 
$$
f^{-\overline{\ell}_k^{(j)}+1}(q_2), f^{-\overline{\ell}_k^{(j)}+1}(m_1)\in\overline{Q_k^{(j)}},
$$
and let $\underline{\ell}_k^{(j)}$ be minimal such that 
$$
f^{\underline{\ell}_k^{(j)}}(f^M(q_2)), f^{\underline{\ell}_k^{(j)}}(f^{s+M}(m_1)). 
$$
Let $(C^0)^{(j)}_k$ be a small enough neighborhood of $f^{-\overline{\ell}_k^{(j)}+1}(m_1)$ and let $(H^0)^{(j)}_k$ be a small enough neighborhood of $f^{-\overline{\ell}_k^{(j)}+1}(q_2)$ such that 
\begin{eqnarray*}
\mathfrak D_k^{(j)}(f)&=&\left\{\overline Q_k^{(j)}\right\}\cup\\&&\left\{\overline {f^i\left((C^0)^{(j)}_k \right)}\left|\right.i\in\left\{0,\dots, \overline{\ell}_k^{(j)}+s+M+\underline{\ell}_k^{(j)}-1\right\}\right\}\cup\\&&\left\{\overline {f^i\left((H^0)^{(j)}_k \right)}\left|\right.i\in\left\{0,\dots, \overline{\ell}_k^{(j)}+M+\underline{\ell}_k^{(j)}-1\right\}\right\}
\end{eqnarray*}
forms a pairwise disjoint collection,
            \begin{eqnarray*}
                \text{mesh}\left(\mathfrak{D}^{(j)}_k(f)\right)<\frac{1}{3}\cdot\text{mesh}\left(\mathfrak{D}(f)\right).
            \end{eqnarray*}
            and 
            \begin{equation}\label{eq:disttoq2}
            \text{dist}\left({f^{\overline{\ell}_k^{(j)}+s}\left((C^0)^{(j)}_k \right)},q_2\right)<\lambda^{n_{0,k}^{(j)}/2},
            \end{equation}
            where $n_{0,k}^{(j)}$ is defined in Subsection \ref{subsection:combinatorics}.
Observe that the construction of $\mathfrak{D}^{(j)}_k(f)$ can be made continuous along the curve $b_k^{(j)}$. Moreover by restricting the neighborhood $U_k^{(j)}$ we can also make the construction of $\mathfrak{D}^{(j)}_k(f)$ varying continuously over the neighborhood still satisfying condition \eqref{eq:disttoq2}. In particular, $\mathfrak{D}^{(j)}_k(f)$ is a dynamical partition with $N^{(j)}_k=\overline{\ell}_k^{(j)}+s+\underline{\ell}_k^{(j)}-1$ and $M^{(j)}_k=\overline{\ell}_k^{(j)}+M+\underline{\ell}_k^{(j)}-1$.
This proves condition $(v)$ and the last statement is consequence of this construction.
\end{proof}
  \begin{rem}\label{rem:nextloops}
    Observe that, using the notation as in the previous proposition and by using the construction in  Subsection \ref{subsection:combinatorics}, there exist $n_{k}^{(j)},\underline{\ell}_k^{(j)}, \overline{\ell}_k^{(j)},n_{0,k}^{(j)}$ such that 
    \begin{eqnarray*}
        \underline\omega\left(w_k^{(j)}\right)=\left(q, {q}^{\underline{\ell}_k^{(j)}}{h}{q}^{n_{0,k}^{(j)}}{c}{q}^{n_{k}^{(j)}}{c}{q}^{\theta n_k^{(j)}}{c}{q}^{\overline{\ell}_k^{(j)}},  {q}^{\underline{\ell}_k^{(j)}}{h}{q}^{\overline{\ell}_k^{(j)}}\right).
    \end{eqnarray*}
\end{rem}
\section{Finite Collet-Eckmann Condition}

In this section we describe the induction step needed to prove later the existence of maps with a Collet-Eckmann point. More in detail, we prove the existence of maps having a point satisfying the finite Collet-Eckmann condition. Roughly speaking a point $x$ is a $T$-Collet-Eckmann (T-CE)  point if there are vectors $v$ in the tangent space at $x$ such that the derivative of the map along the orbit of $x$ in the direction of $v$ grows at the least exponentially fast up to time $T$.   Below is the formal definition.

As before, without lose of generality, we consider a diffeomorphism $f:\mathcal M\to\mathcal M$ where the manifold $\mathcal M$ is an open set in $\mathbb R^2$.
Let $\varphi>0$. Denote by $\text{T}_{x}\mathcal M$ the tangent space at $x$, and by $$K_{\varphi}(x)=\left\{(v_1,v_2)\in \text{T}_{x}\mathcal M\left|\right. |v_2| \tan\varphi> |v_1| \right\},$$ the cone at $x$ of angle ${\varphi}$.

\begin{defin}
Given constants  $0<C<1$, $\rho>1$, we say that  $x\in\mathcal M$ is a $T$-Collet-Eckmann (T-CE) point if the following holds. For all vectors $v\in K_{\varphi}(x)$, 

$$
\left| Df^{s}_xv\right|\geq C\rho^s \left| v\right| \text{, \hspace{.49cm}        }\forall  \text{        } 0\leq s\leq T.
$$
\end{defin}

From now on we are going to assume an extra condition on $\theta$, namely we ask that,
\begin{eqnarray}\label{thetacond2}
\lambda^{\theta}\mu^{1+\alpha}>  \lambda^{\alpha},
\end{eqnarray}
where $\alpha$ is defined in  \eqref{alphaofuse}. By \eqref{thetacond1} and \eqref{alphaofuse}, there exists $\theta\in(\theta_0,\theta_1)$ satisfying \eqref{thetacond2}, see Appendix. Take constants $\varphi>0$, $0<C<1$ and $\rho>1$ such that,
 \begin{eqnarray}\label{eq:varphicond}
 &&\text{ (o) } \cos\varphi >\frac{\rho}{\mu}\nonumber \\ \\&& \text{ (oo) } 1<\rho<\lambda^{\theta}\mu^2<\mu\nonumber .
   \end{eqnarray}
   
We introduce now some notation. 
 Let $(x,y)$ be a point and let $\Gamma$ be a curve which is a graph over the $x$-axes. If there is a point $(x,y_0)\in\Gamma$ then denote by $\text{d}_{\text{v}}\left(x, \Gamma\right)=|y-y_0|$ the vertical distance between a point $x$ and the curve $\Gamma$. The angle between two vectors is denoted by $\angle$.
 
 \bigskip
 
  Let $S$ be a tangency strip and let $\left\{b,U\right\}$ be a $(1,N,M)$ descendant of $S$. For all $f\in U$, let $\mathfrak{D}(f)$ be the $\left(1, N, M\right)$ dynamical partition of $f$. Given $T\in\left\{1, \dots,N-1\right\}$, one can define a vector field $\mathfrak {X}^{N-T}$ on $C^{N-T}$ as 
 $$
 Df^T(x)v(x)=e_1 \text{ for all } x\in C^{N-T}.
 $$   
 Observe that, since $f^T:C^{N-T}\to C^N$ is a diffeomorphism, then $\mathfrak {X}\left(f^T\right)$ is well defined.

 We are now ready to present the induction step needed to prove the existence of laminations of maps with a Collet–Eckmann point. We use the notation introduced above.
    
\begin{prop}\label{Prop:inductionstepCE}
    Let $S$ be a $r$-tangency strip and let $\left\{b,U\right\}$ be a $(1,N,M)$ descendant of $S$.     Suppose that for all $f\in U$, there exist a positive constant $K$ and $T\in\left\{1, \dots,N-1\right\}$ such that the $T$-Collet-Eckmann Condition holds,
    \begin{itemize}
    \item[$\text{IH}_Q:$] $f_{|Q}$ is a linear map, i.e. 
    $$
    f_{|Q}=\left(\begin{matrix}\lambda& 0\\0&\mu
    \end{matrix}\right).
    $$
     \item[$\text{IH}_{H^0}:$] The family  $
  \left\{  f^{i}_{|H^0}\right\}_{i=0}^{M}
    $
    is a compact $\Cd$-family.
     \item[$\text{IH}_{C^0}:$]  \begin{itemize}
     \item[(i)] There exists a curve $\Gamma$ in $C^0$ which is a graph over the $x$-axis and 
     $$
     Df_x^{N-T}e_2\in\text{span}\left(\mathfrak X^{N-T}\right)\iff x\in\Gamma.
     $$
      \item[(ii)] For all $x\in C^0$,  
           $$
     \frac{1}{K}{{d}_\text{v}}\left(x, \Gamma\right)\leq\angle\left(Df_x^{N-T}e_2, \mathfrak X^{N-T}\right)\leq K d_{\text{v}}\left(x,\Gamma\right).     
          $$
      \item[(iii)]     The family  $
  \left\{  f^{i}_{|C^0}\right\}_{i=0}^{N-T}
    $
    is a compact $\Cd$-family.     
          \end{itemize}
  \item[$\text{IH}_{C^{N-T}}:$]  \begin{itemize}
     \item[(i)] For all $x\in C^{N-T}$ and for all $v\in K_{\varphi}(x)\subset \text{T}_x$,
     $$ 
    \left |Df_x^s v\right|\geq C\rho^s\left|v\right|, \text{            }\forall \text{  } 0\leq s\leq T.
        $$
     
           \item[(ii)] For all $x\in C^{N-T}$ and for all $v\in K_{\varphi}(x)\subset \text{T}_x$, 
           $$
     Df_x^{T}v \in K_{\varphi}\left(f^T(x)\right)\subset \text{T}_{f^T(x)}.         
     $$
      \item[(iii)]     The family  $
  \left\{  f^{i}_{|C^{N-T}}\right\}_{i=0}^{T}
    $
    is a compact $\Cd$-family.     
          \end{itemize}
    
             \end{itemize}
     Then, for every $\tilde r>r$, there exists a $\tilde r$-tangency strip $\tilde S\subset U\subset S$. Each $(1,\tilde N,\tilde M)$ descendant $\left\{\tilde b,\tilde U\right\}$, satisfies the following. For all $f\in \tilde U$, there exist a positive constant $\tilde K$ and $\tilde T\in\left\{1, \dots,\tilde N-1\right\}$, $\tilde T>2T$, with $$\tilde C^{\tilde N-\tilde T}\subset C^{ N- T}$$such that the $\tilde T$-Collet-Eckmann Condition holds. Namely,
    \begin{itemize}
    \item[$\text{IH}_{\tilde Q}:$] $f_{| \tilde Q}$ is a linear map, i.e. 
    $$
    f_{|\tilde Q}=\left(\begin{matrix}\lambda& 0\\0&\mu
    \end{matrix}\right).
    $$
     \item[$\text{IH}_{\tilde H^0}:$] The family  $
  \left\{  f^{i}_{|\tilde H^0}\right\}_{i=0}^{\tilde M}
    $
    is a compact $\Cd$-family.
     \item[$\text{IH}_{\tilde C^0}:$]  \begin{itemize}
     \item[(i)] There exists a curve $\tilde \Gamma$ in $\tilde C^0$ which is a graph over the $x$-axis and 
     $$
     Df^{\tilde N-\tilde T}(x)e_2\in\text{span}\left({\mathfrak X}^{\tilde N-\tilde T} \right)\iff x\in\tilde \Gamma.
     $$
      \item[(ii)] For all $x\in \tilde C^0$,  
           $$
     \frac{1}{\tilde K}{\text{d}_{v}}\left(x, \tilde \Gamma\right)\leq\angle\left(D f_x^{\tilde N-\tilde T}e_2, {\mathfrak X}^{\tilde N-\tilde T}\right)\leq \tilde K d_{\text{v}}\left(x,\tilde \Gamma\right).     
          $$
      \item[(iii)]     The family  $
  \left\{  f^{i}_{|\tilde C^0}\right\}_{i=0}^{\tilde N-\tilde T}
    $
    is a compact $\Cd$-family.     
          \end{itemize}
  \item[$\text{IH}_{\tilde C^{\tilde N-\tilde T}}:$]
    \begin{itemize}
     \item[(i)] For all $x\in \tilde C^{\tilde N-\tilde T}$ and for all $v\in K_{\varphi}(x)\subset \text{T}_x$,
     $$ 
    \left |D f_x^s v\right|\geq C\rho^s\left|v\right| \text{  }\forall \text{  } 0\leq s\leq \tilde T.
        $$
     
           \item[(ii)] For all $x\in \tilde C^{\tilde N-\tilde T}$ and for all $v\in K_{\varphi}(x)\subset \text{T}_x$, 
           $$
     D f_x^{\tilde T}v \in K_{\varphi}\left(f^{\tilde T}(x)\right)\subset \text{T}_{f^{\tilde T}(x)}.         
     $$
      \item[(iii)]     The family  $
  \left\{   f^{i}_{|\tilde C^{\tilde N-\tilde T}}\right\}_{i=0}^{\tilde T}
    $
    is a compact $\Cd$-family.     
          \end{itemize}
    
             \end{itemize}        
    
     \end{prop}
\begin{proof}
 Let $\tilde S=S_j\subset U\subset S$ be a tangency strip and let $\left\{\tilde b,\tilde U\right\}=\left\{b_{k}^{(j)}, U_{k}^{(j)}\right\}$ be a  $\left(1,\tilde N,\tilde M\right)=\left(1,N_k^{(j)}, M_k^{(j)}\right)$ descendant of $\tilde S$ as in Proposition \ref{Prop:inductionstep}. In each step of the construction we might have to increase the $j$. Consider the orbit from $\tilde C^0$ to $\tilde C^{\tilde N}$. By construction there exist exactly three moments, $s_1$, $s_2$, $s_3$ smaller than $\tilde N$ such that $\tilde C^{s_i}\subset C^{N-T}$. Using the notation from  Remark \ref{rem:nextloops} we get 
 \begin{eqnarray*}
 s_1&=&\overline{\ell}_k^{(j)}+N-T,\\
 s_2&=&s_1+T+\theta n_k^{(j)}+N-T,\\
  s_3&=&s_2+T+n_k^{(j)}+N-T.
 \end{eqnarray*}
 Define $\tilde T$ such that $\tilde N-\tilde T=s_2$ corresponds to the second passage\footnote{The second passage is an essential choice. Using the third passage will not lead to a Cantor set, see Lemma \ref{lem:omegaO2} and its proof. Using the first passage will not give enough expansion to ensure the CE condition. } trough $C^{N-T}$, see Figure \ref{Passage}.
 \begin{figure}
\centering
\includegraphics[width=0.9\textwidth]{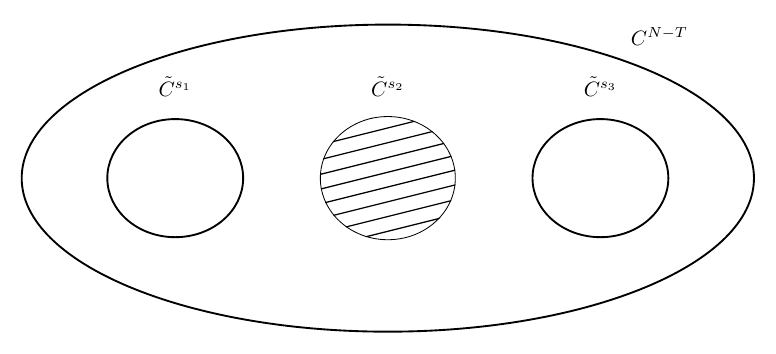}
\caption{Passages in $C^{N-T}$}
\label{Passage}
\end{figure}
 In particular,
 $$\tilde T= T+n_k^{(j)}+N-T+T+n_{0,k}^{(j)}+M+\underline{\ell}_k^{(j)}.$$ 
 For all $\tilde f\in\tilde U$, let $\tilde K$ be a positive constant describing the bounds as required in $\text{IH}_{\tilde H^0}$,  $\text{IH}_{\tilde C^0}$-(iii) and $\text{IH}_{\tilde C^{\tilde N-\tilde T}}$-(iii). The existence of such a constant is guaranteed by the fact that this conditions consider only a finite number of iterates of the original diffeomorphism.  Observe that $\text{IH}_{\tilde Q}$ is trivially satisfies. Let $\tilde\Gamma$, be the curve associated to the critical point $\tilde c$, then  $\text{IH}_{\tilde C^0}$-(i) and  $\text{IH}_{\tilde C^0}$-(ii) are satisfied by construction.

 The rest of the proof is devoted to prove  $\text{IH}_{\tilde C^{\tilde N-\tilde T}}$-(i) and   $\text{IH}_{\tilde C^{\tilde N-\tilde T}}$-(ii).  This will be achieved by steps. Each step corresponds to a term in the definition of $\tilde T$. In order to lighten the notation, we denote $\tilde C=\tilde C^{\tilde N-\tilde T}$. Let $x\in\tilde C$ and let $v\in K_{\varphi}(x)\subset \text{T}_x$. Observe that $\tilde C\subset C^{N-T}$.
 \begin{itemize}
\item[{\bf Step 1:}] We begin by proving  $\text{IH}_{\tilde C^{\tilde N-\tilde T}}$-(i) from time $0$ up to time $T$. By $\text{IH}_{ C^{ N-T}}$-(i), we get,

$$
  \text{IH}_{\tilde C}\text{-(i)P1:    }  \left |D f_x^s v\right|\geq C\rho^s\left|v\right| \text{  }\forall \text{  } 0\leq s\leq  T.
     $$

                Moreover, by $\text{IH}_{ C^{N-T}}$-(ii) 

 $$   D f_x^{T}v \in K_{\varphi}\left(f^T(x)\right)\subset \text{T}_{f^{ T}(x)}.    
 $$ 
    
       In particular,  if we denote by 
       $$ v^{(1)}=Df^{T}_xv=\left(\begin{matrix}
       v^{(1)}_1\\v^{(1)}_2
       \end{matrix}\right),$$
         it follows that,
         
         \begin{equation}\label{eq:v1}
         v^{(1)}_2>\cos\varphi \left|v^{(1)}\right|
         \text{ and } \frac{v^{(1)}_1}{v^{(1)}_2}<\tan\varphi.
         \end{equation}
       
        \item[{\bf Step 2:}] Now we prove  $\text{IH}_{\tilde C^{\tilde N-\tilde T}}$-(i) from time $T+1$ up to time $T+n_k^{(j)}$. Indeed, by $\text{IH}_{Q}$, $\text{IH}_{\tilde C}$-(i)P1 and \eqref{eq:v1}, for all $1\leq s\leq n$, we get,

           \begin{eqnarray*}
     \text{IH}_{\tilde C}\text{-(i)P2:   }     \left|Df^{T+s}_xv\right|&=&\left|\left(\begin{matrix}\lambda^s&0\\0&\mu^s\end{matrix}\right)
          \left(\begin{matrix}v^{(1)}_1\\v^{(1)}_2\end{matrix}\right)          \right|>\mu^sv^{(1)}_2\geq\mu^s\cos\varphi\left|v^{(1)}\right| \\
          &=&\mu^s\cos\varphi \left|Df^T_xv\right|\geq\mu^s\cos\varphi C\rho^T\left|v\right| \geq  C\rho^{T+s}\left|v\right|.    
                 \end{eqnarray*}   
               
                 Observe that in the last inequality of $\text{IH}_{\tilde C}$-(i)P2 we used that 
               $$
               \mu^s\cos\varphi \geq \rho^{s},\text{    }\forall 1\leq s\leq n.     
               $$
               This follow from the condition imposed in \eqref{eq:varphicond}-(o). 
               In particular, $\text{IH}_{\tilde C}$-(i)P2 proves 
               \begin{eqnarray}\label{eq:IHC(i)Step2Final}
                 \left|Df^{T+n_k^{(j)}}_xv\right|\geq\mu^{n_k^{(j)}}\cos\varphi C\rho^T\left|v\right|.         
           \end{eqnarray}
           Denote by 
       $$ v^{(2)}=Df^{T+n_k^{(j)}}_xv=\left(\begin{matrix}
       v^{(2)}_1\\v^{(2)}_2
       \end{matrix}\right).$$
         By $\text{IH}_{Q}$ and by \eqref{eq:v1}, 
         \begin{equation}\label{eq:v2}
         \frac{v^{(2)}_1}{v^{(2)}_2}=\left(\frac{\lambda}{\mu}\right)^n  \frac{v^{(1)}_1}{v^{(1)}_2}<\left(\frac{\lambda}{\mu}\right)^n \tan\varphi.
         \end{equation}

     \item[{\bf Step 3:}]  Now we prove  $\text{IH}_{\tilde C^{\tilde N-\tilde T}}$-(i) from time $T+n_k^{(j)}+1$ up to time $T+n_k^{(j)}+N-T$. Indeed, by $\text{IH}_{C^0}$-(iii) and \eqref{eq:IHC(i)Step2Final}, for all $1\leq s\leq N-T$, we get,

  \begin{eqnarray*}
  \text{IH}_{\tilde C}\text{-(i)P3:  }   \left|Df^{T+n_k^{(j)}+s}_xv\right|\geq \frac{1}{K} \left|Df^{T+n_k^{(j)}}_xv\right|>\frac{1}{K}\mu^{n_k^{(j)}}\cos\varphi\cdot C\rho^T\left|v\right| \geq  C\rho^{T+n_k^{(j)}+s}\left|v\right|.    
                 \end{eqnarray*}

               Observe that in the last inequality of $\text{IH}_{\tilde C}$-(i)P3 we used that 
               $$
              \frac{1}{K} \left(\frac{\mu}{\rho}\right)^{n_k^{(j)}}\cos\varphi \geq \rho^{s},\text{    }\forall 1\leq s\leq N-T     
               $$
               which is equivalent to ask 
                $$
              \frac{1}{K} \left(\frac{\mu}{\rho}\right)^{n_k^{(j)}}\cos\varphi \geq \rho^{N-T}.     
               $$
               
       This follow from the condition imposed in \eqref{eq:varphicond}-(oo) and taking $n_k^{(j)}$ large enough (which is equivalent to take $j$ large enough). 
               In particular, $\text{IH}_{\tilde C}$-(i)P3 proves 
               \begin{eqnarray}\label{eq:IHC(i)Step3Final}
                 \left|Df^{T+n_k^{(j)}+N-T}_xv\right|\geq\frac{C}{K}\cos\varphi\cdot\mu^{n_k^{(j)}}\ \rho^T\left|v\right|.         
           \end{eqnarray}

   Moreover by $\text{IH}_{C^0}$-(ii) and Lemma  \ref{ymaxymin},
      \begin{equation*}\label{eqinL1}
      \angle\left(Df_{f^{T+n_k^{(j)}}(x)}^{N-T}e_2, \mathfrak X^{N-T}\right)\geq\frac{1}{K}{\text{d}_\text{v}}\left(f^{T+n_k^{(j)}}(x), \Gamma\right)\geq\frac{1}{K^2}\left(\lambda^{\theta}\mu\right)^{n_k^{(j)}},
      \end{equation*}
     and by \eqref{eq:v2},
       \begin{equation*}\label{eqinL2}
      \tan\left(v^{(2)},e_2\right)\leq\left(\frac{\lambda}{\mu}\right)^n\tan\varphi.
      \end{equation*}         
      The last two inequalities together with $\text{IH}_{C^0}$-(iii), gives that, for $n_k^{(j)}$ large enough,
      \begin{equation}\label{IHtildeCP3}
      \angle\left(Df_x^{T+n_k^{(j)}+N-T}v, \mathfrak X^{N-T}\right)\geq\frac{1}{2K^2}\left(\lambda^{\theta}\mu\right)^{n_k^{(j)}}.
      \end{equation}

           \item[{\bf Step 4:}] 
         Now we prove  $\text{IH}_{\tilde C^{\tilde N-\tilde T}}$-(i) from time $T+n_k^{(j)}+N-T+1$ up to time $T+n_k^{(j)}+N-T+T$. Indeed, by $\text{IH}_{C^{N-T}}$-(iii) and \eqref{eq:IHC(i)Step3Final}, for all $1\leq s\leq T$, we get,

  \begin{eqnarray*}
 \text{IH}_{\tilde C}\text{-(i)P4:  }     \left|Df^{T+n_k^{(j)}+N-T+s}_xv\right|&\geq& \frac{1}{K} \left|Df^{T+n_k^{(j)}+N-T}_xv\right|\\&\geq& \frac{C}{K^2}\cos\varphi \cdot\mu^{n_k^{(j)}}\rho^T\left|v\right| \geq  C\rho^{T+n_k^{(j)}+N-T+s}\left|v\right|.    
                 \end{eqnarray*}

               Observe that in the last inequality of $\text{IH}_{\tilde C}$-(i)P4 we used that 
               $$
              \frac{1}{K^2}\cos\varphi \left(\frac{\mu}{\rho}\right)^{n_k^{(j)}}\rho^T\geq \rho^{N+s},\text{    }\forall 1\leq s\leq T     
               $$
               which is equivalent to ask 
                $$
              \frac{1}{K^2} \cos\varphi\left(\frac{\mu}{\rho}\right)^{n_k^{(j)}} \geq \rho^{N}.     
               $$
               
       This holds for $n_k^{(j)}$ large enough from the condition imposed in \eqref{eq:varphicond}-(oo). 
               In particular, $\text{IH}_{\tilde C}$-(i)P4 proves 
               \begin{eqnarray}\label{eq:IHC(i)Step4Final}
                 \left|Df^{T+n_k^{(j)}+N}_xv\right|\geq\frac{C}{K^2}\cos\varphi\cdot\mu^{n_k^{(j)}}\ \rho^T\left|v\right|.         
           \end{eqnarray}

Equation \eqref{IHtildeCP3}, $\text{IH}_{C}$-(iii) and the definition of the vector field $\mathfrak X$ give that, for $n_k^{(j)}$ large enough, 
  
             \begin{equation}\label{IHtildeCP4}
      \angle\left(Df_x^{T+n_k^{(j)}+N}v, e_1\right)\geq\frac{1}{2K^3}\left(\lambda^{\theta}\mu\right)^{n_k^{(j)}}.
      \end{equation}
    In particular, if we denote by 
       $$ v^{(4)}=Df^{T+n_k^{(j)}+N}_xv=\left(\begin{matrix}
       v^{(4)}_1\\v^{(4)}_2
       \end{matrix}\right), $$
         then, by \eqref{IHtildeCP4}, 
         \begin{equation}\label{eq:v4}
         \frac{v^{(4)}_2}{\left|v^{(4)}\right|}\geq\sin \frac{1}{2K^3}\left(\lambda^{\theta}\mu\right)^{n_k^{(j)}}\geq \frac{1}{4K^3}\left(\lambda^{\theta}\mu\right)^{n_k^{(j)}}
         \end{equation}   
         
\item[{\bf Step 5:}]  Now we prove  $\text{IH}_{\tilde C^{\tilde N-\tilde T}}$-(i) from time $T+n_k^{(j)}+N+1$ up to time $T+n_k^{(j)}+N+n_{0,k}^{(j)}$. Indeed, by $\text{IH}_{Q}$, \eqref{eq:v4} and \eqref{eq:IHC(i)Step4Final}, for all $1\leq s\leq n_{0,k}^{(j)}$, we get,

  \begin{eqnarray*}
 \text{IH}_{\tilde C}\text{-(i)P5:  }   \left|Df^{T+n_k^{(j)}+N+s}_xv\right|&\geq& \mu^s v^{(4)}_2\geq \mu^s\frac{1}{4K^3}\left(\lambda^{\theta}\mu\right)^{n_k^{(j)}}\left|Df^{T+n_k^{(j)}+N}_xv\right|\\&\geq &\frac{C}{4K^5}\cos\varphi \cdot\mu^{s+n_k^{(j)}}\rho^T\left(\lambda^{\theta}\mu\right)^{n_k^{(j)}}\left|v\right|\\ &\geq&  C\rho^{T+n_k^{(j)}+N+s}\left|v\right|.    
                 \end{eqnarray*}

               Observe that in the last inequality of $\text{IH}_{\tilde C}$-(i)P5 we used that 
               $$
              \frac{1}{4K^5}\cos\varphi \left(\frac{\mu}{\rho}\right)^{n_k^{(j)}+s}\left(\lambda^{\theta}\mu\right)^{n_k^{(j)}}\geq \rho^{N},\text{    }\forall 1\leq s\leq n_{0,k}^{(j)}     
               $$
               which is equivalent to ask 
                $$
             \frac{1}{4K^5}\cos\varphi \left(\frac{\mu}{\rho}\right)\left(\lambda^{\theta}\mu^2\right)^{n_k^{(j)}}\geq \rho^{n_k^{(j)}+N}.
               $$
               
       This follow from the condition imposed in \eqref{eq:varphicond}-(oo) and taking $n_k^{(j)}$ large enough. 
               In particular, $\text{IH}_{\tilde C}$-(i)P5 proves 
               \begin{eqnarray}\label{eq:IHC(i)Step5Final}
                 \left|Df^{T+n_k^{(j)}+N+n_{0,k}^{(j)}}_xv\right|\geq\frac{C}{4K^5}\cos\varphi\cdot \mu^{n_{0,k}^{(j)}}\rho^T\left(\lambda^{\theta}\mu^2\right)^{n_k^{(j)}}\left|v\right|.         
           \end{eqnarray}
           Moreover, by \eqref{IHtildeCP4}, 
             $$
      \tan\angle\left(Df_x^{T+n_k^{(j)}+N}v, e_2\right)\leq\frac{2K^3}{\left(\lambda^{\theta}\mu\right)^{n_k^{(j)}}},
      $$
      which together with \eqref{thetacond2} and by taking $n_k^{(j)}$ large enough imply that, 
      
     \begin{equation}\label{IHtildeCP5}
      \tan\angle\left(Df_x^{T+n_k^{(j)}+N+n_{0,k}^{(j)}}v, e_2\right)\leq\left(\frac{\lambda}{\mu}\right)^{n_{0,k}^{(j)}}\frac{2K^3}{\left(\lambda^{\theta}\mu\right)^{n_k^{(j)}}}<<1.
       \end{equation} 
      \item[{\bf Step 6:}]  We start proving  $\text{IH}_{\tilde C^{\tilde N-\tilde T}}$-(i) from time $T+n_k^{(j)}+N+n_{0,k}^{(j)}+1$ up to time $T+n_k^{(j)}+N+n_{0,k}^{(j)}+M$. Indeed, by $\text{IH}_{H}$ and \eqref{eq:IHC(i)Step5Final}, for all $1\leq s\leq M$, we get,
                   
  \begin{eqnarray*}
 \text{IH}_{\tilde C}\text{-(i)P6:   }  \left|Df^{T+n_k^{(j)}+N+n_{0,k}^{(j)}+s}_xv\right|&\geq& \frac{1}{K}  \left|Df^{T+n_k^{(j)}+N+n_{0,k}^{(j)}}_xv\right|\\&\geq& \frac{C}{4K^6}\cos\varphi\cdot \mu^{n_{0,k}^{(j)}}\rho^T\left(\lambda^{\theta}\mu^2\right)^{n_k^{(j)}}\left|v\right| \\&\geq&  C\rho^{T+n_k^{(j)}+N+n_{0,k}^{(j)}+s}\left|v\right|.    
                 \end{eqnarray*}

               Observe that in the last inequality of $\text{IH}_{\tilde C}$-(i)P6 we used that 
               $$
              \frac{1}{4K^6}\cos\varphi \left(\frac{\mu}{\rho}\right)^{n_k^{(j)}+n_{0,k}^{(j)}}\left(\lambda^{\theta}\mu\right)^{n_k^{(j)}}\geq \rho^{N+s},\text{    }\forall 1\leq s\leq M     
               $$
               which is equivalent to ask 
                $$
             \frac{1}{4K^6}\cos\varphi \left(\frac{\mu}{\rho}\right)^{n_{0,k}^{(j)}}\left(\frac{\lambda^{\theta}\mu^2}{\rho}\right)^{n_k^{(j)}}\geq \rho^{N+M}.
               $$
               
       This follow from the condition imposed in \eqref{eq:varphicond}-(oo) and taking $n_k^{(j)}$ large enough. 
               In particular, $\text{IH}_{\tilde C}$-(i)P6 proves 
               \begin{eqnarray}\label{eq:IHC(i)Step6Final}
                 \left|Df^{T+n_k^{(j)}+N+n_{0,k}^{(j)}+M}_xv\right|\geq\frac{C}{4K^6}\cos\varphi\cdot\mu^{n_{0,k}^{(j)}}\rho^T\left(\lambda^{\theta}\mu^2\right)^{n_k^{(j)}}\left|v\right|.         
           \end{eqnarray}
  Observe that, there exists $\beta>0$ such that 
  \begin{equation*}\label{eq·step6angle1}
\angle\left(Df^M_{q_2}e_2,e_1\right)>\beta
 \end{equation*}
   and by \eqref{eq:disttoq2},
         \begin{equation*}\label{eq·step6angle2}
         \text{dist}\left(f^{T+n_k^{(j)}+N+n_{0,k}^{(j)}}(x),q_2\right)<<1.
   \end{equation*}
    The previous two inequalities, \eqref{IHtildeCP5} together with $\text{IH}_{H}$ gives 
     \begin{equation*}\label{eq:step6angle}
\angle\left(Df^{T+n_k^{(j)}+N+n_{0,k}^{(j)}+M}_{x}v,e_1\right)>\frac{\beta}{2}.
\end{equation*}
In particular, if we denote by 
       $$ v^{(6)}=Df^{T+n_k^{(j)}+N+n_{0,k}^{(j)}+M}_xv=\left(\begin{matrix}
       v^{(6)}_1\\v^{(6)}_2
       \end{matrix}\right),$$   
    it follows that
       \begin{equation}\label{IHtildeCP6}
      \frac{v^{(6)}_2}{\left|v^{(6)}\right|}\geq\sin \frac{\beta}{2}.          
       \end{equation}
       
         \item[{\bf Step 7:}]  Now we prove $\text{IH}_{\tilde C^{\tilde N-\tilde T}}$-(i) from time $T+n_k^{(j)}+N+n_{0,k}^{(j)}+M+1$ up to time $\tilde T=T+n_k^{(j)}+N+n_{0,k}^{(j)}+M+\underline\ell_k^{(j)}$. Indeed, by $\text{IH}_{Q}$, \eqref{IHtildeCP6} and \eqref{eq:IHC(i)Step6Final}, for all $1\leq s\leq \underline\ell_k^{(j)}$, we get,

  \begin{eqnarray*}
   \text{IH}_{\tilde C}\text{-(i)P7:  }  \left|Df^{T+n_k^{(j)}+N+n_{0,k}^{(j)}+M+s}_xv\right|&\geq& \mu^s\sin\frac{\beta}{2}  \left|Df^{T+n_k^{(j)}+N+n_{0,k}^{(j)}+M}_xv\right|\\&\geq& \frac{C}{4K^6}\cos\varphi\cdot\sin\frac{\beta}{2}\mu^{n_{0,k}^{(j)}+s}\rho^T\left(\lambda^{\theta}\mu^2\right)^{n_k^{(j)}}\left|v\right| \\&\geq&  C\rho^{T+n_k^{(j)}+N+n_{0,k}^{(j)}+M+s}\left|v\right|.    
                 \end{eqnarray*}

               Observe that in the last inequality of $\text{IH}_{\tilde C}$-(i)P7 we used that 
         
                $$
             \frac{1}{4K^6}\cos\varphi\cdot \sin\frac{\beta}{2}\left(\frac{\mu}{\rho}\right)^{n_{0,k}^{(j)}+s}\left(\frac{\lambda^{\theta}\mu^2}{\rho}\right)^{n_k^{(j)}}\geq \rho^{N+M}.
               $$
               
       This follow from the condition imposed in \eqref{eq:varphicond}-(oo) and taking $n_k^{(j)}$ large enough. 
              Denote by 
       $$ v^{(7)}=Df^{T+n_k^{(j)}+N+n_{0,k}^{(j)}+M+\underline\ell_k^{(j)}}_xv=\left(\begin{matrix}
       v^{(7)}_1\\v^{(7)}_2
       \end{matrix}\right),$$   
       then, by $\text{IH}_{Q}$ and by \eqref{IHtildeCP6}, we get 
       
       $$
        \frac{v^{(7)}_1}{v^{(7)}_2} = \left(\frac{\lambda}{\mu}\right)^{\underline\ell_k^{(j)}} \frac{v^{(6)}_1}{v^{(6)}_2}\leq  \left(\frac{\lambda}{\mu}\right)^{\underline\ell_k^{(j)}} \frac{1}{\sin\beta/2}<\tan\varphi,       $$        
    with the last inequality achieved by taking $j$ large enough which implies that ${\underline\ell_k^{(j)}}$ is large enough. 
       
       \end{itemize} 
       In conclusion, the inequality above proves  $\text{IH}_{\tilde C}$-(ii) 
    and $\text{IH}_{\tilde C}$-(i) is achieved by putting together $\text{IH}_{\tilde C}$-(i)P1, $\text{IH}_{\tilde C}$-(i)P2, $\text{IH}_{\tilde C}$-(i)P3, $\text{IH}_{\tilde C}$-(i)P4, $\text{IH}_{\tilde C}$-(i)P5, $\text{IH}_{\tilde C}$-(i)P6 and $\text{IH}_{\tilde C}$-(i)P7.  \end{proof}
   \section{The lamination and the invariant Cantor set}
In this section we build a lamination in parameter space whose maps have an invariant Cantor set. 
\subsection{The lamination}
Consider the $1$-tangency strip $S$ given by $S=[-t_0,t_0]\times[-a_0,a_0]$ with tangency line $b=\left\{a=0\right\}$. Take $U=S$. In Subsection \ref{familyofunfoldings}, we defined a family $F$ which is an unfolding of the tangency $b$ and each map $f\in S$ has a $(1,N_1,M_1)$ dynamical partition with $N_1=N$ and $M_1=M$, where $N$ and $M$ are defined in Definition \ref{stronghomtang}, property $(f8)$.

Denote by $$\mathfrak{S}_1=\left\{S\right\},$$ and by $$\mathfrak{S}_2=\left\{S_{j,2}\right\}_{j=1}^{\infty},$$ where each $S_{j,2}$ is a $n_2+j$-tangency strip contained in $S$ whose existence is guaranteed by Proposition \ref{Prop:inductionstep}. By applying Proposition \ref{Prop:inductionstep} inductively, for all generations $g\geq 3$ one can define a sequence of collections of tangency strips
$$\mathfrak{S}_g=\left\{S_{j,g}\right\}_{j=1}^{\infty},$$ satisfying the following properties:
\begin{itemize}
    \item $\forall j,j'$ with $j\neq j'$, $S_{j,g}\cap S_{j',g}=\emptyset$, 
         \item $\forall S_{j,g}\in\mathfrak{S}_g$ there exists a unique $S_{i,g-1}\in\mathfrak{S}_{g-1}$ with $S_{j,g}\subset S_{i,g-1}$,
          \item $\forall j,$ $S_{j,g}$ is bounded by the graph of two functions $s_{j,g}^+,s_{j,g}^-:[-t_0,t_0]\to\mathbb R$ such that,
        \begin{eqnarray*}
             \left\lVert  s_{j,g}^+-s_{j,g}^-\right\rVert_0\leq\frac{1}{2^g},
        \end{eqnarray*} 
        and if $S_{j,g}\subset S_{i,g-1}$, then 
         \begin{eqnarray*}
           \left\lVert \frac{ds_{j,g}^{\pm}}{dt}-\frac{ds_{i,g-1}^{\pm}}{dt}\right\rVert_0\leq\frac{1}{2^g}.
        \end{eqnarray*} 
\end{itemize}
This last property relies on \eqref{eq:angleall}.
Observe that $\cup\mathfrak{S}_g\subset\cup\mathfrak{S}_{g-1}$.
We are now ready to define the lamination 
\begin{eqnarray}\label{eq:lamination}
\mathfrak{L}=\bigcap_{g=1}^{\infty}\bigcup_{j=1}^{\infty} S_{j,g}\subset [-t_0,t_0]\times [0,a_0].
\end{eqnarray}
The connected components of $\mathfrak{L}$ are called leafs of the lamination. Namely, a leaf of the lamination has the form 
\begin{eqnarray}\label{eq:leavelamination}
\ell=\bigcap_{g=1}^{\infty} S_{g},
\end{eqnarray}
where, for all $g$, $S_{g}\in\mathfrak{S}_g$ and $S_g\subset S_{g-1}$.
Observe that $\mathfrak{L}$ is a $\Cuno$ lamination, i.e. each leaf is the graph of a $\Cuno$ function, $\lim_{q\to\infty}s^+_{g}$ where the graph of $s^+_{g}$ is the upper boundary of $S_g$.

\bigskip

The proof is a consequence of the construction of the lamination and of the fact that the strips of the same generation are pairwise disjoint. 
\begin{prop}\label{Prop:longleavesandtransversalCantorset}
Every leaf of $\mathfrak L$ is the graph of a smooth function $\ell:[-t_0,t_0]\mapsto [-a_0,a_0]$. Morever if $t\in [-t_0,t_0]$, the closure of the set $L_t=\left\{(t,a)\in\mathfrak L\right\}$ is a Cantor set. 
\end{prop}
\begin{rem}\label{rem:NMlarge}
As in the proof of Proposition \ref{Prop:inductionstepCE} by taking all $n_g^{(j)}$ large enough we may also assume that, for all $g$, $N_{g+1},M_{g+1}\ge 32 \max\{N_g, M_g\}$ and $N_g, M_g\ge 3$.
\end{rem}

\begin{rem}\label{rem:NewhouseInstability}
Observe that, by construction, given an open neighborhood of $f\in\mathfrak{L}$, it contains in parameter space a curve $b$ of secondary tangencies. Hence, by Remark \ref{NHinter}, such open neighborhood contains maps with infinitely many sinks.
\end{rem}
\subsection{The invariant Cantor set}
Given a leaf $\ell=\cap S_{g}$ of the lamination $\mathfrak L$ and a map $f\in\ell$, one can associate a sequence $\left\{\left(b_g, U_g,\mathfrak{D}_g\right)\right\}_{g=1}^{\infty}$ and a sequence of numbers $\left\{\left(1, N_g,M_g\right)\right\}_{g=1}^{\infty}$ such that, for all $g$, $(b_g,U_g)$ is the $(1, N_g, M_g)$ descendant of $S_{g-1}$ and $\mathfrak{D}_g$ is the $(1, N_g, M_g)$ dynamical partition of $f$. 
The sequence $\left\{\left(1, N_g,M_g\right)\right\}_{g=1}^{\infty}$ is called the type of $\ell$.

Moreover for each $g$, $\mathfrak{D}_g$ is a refinement of $\mathfrak{D}_{g-1}$, with mesh going to zero when $g$ goes to infinity. We can now define the set associated to $f$
$$
A(f)=\bigcap_{g=1}^{\infty} \mathcal{D}_g=\bigcap_{g=1}^{\infty}\bigcup_{D\in\mathfrak{D}_g} {D}
$$
where $\mathcal{D}_g$ is the union of the elements of $\mathfrak{D}_g$. This set is going to play a fundamental role. Let us begin by studying its properties. Following Proposition \ref{Prop:inductionstep}, for all $g$, let $X_g$ be the $(1, N_g,M_g)$ directed graph, let $w_g:X_{g+1}\to X_g$ be the projection and
$w_{g'g}: X_g\to X_{g'}$ be 
$$
w_{g'g}=w_{g'}\circ \cdots w_{g-1}\circ w_g.
$$ 
\subsection{Properties of the set $A$}
In this subsection we collect properties of the set $A$. 
\begin{lem}\label{lem:ACantorset}
Let $\ell$ be a leaf of the lamination $\mathfrak{L}$ and let $f\in\ell$. Then the set $A(f)$ is a Cantor set.
\end{lem}
\begin{proof}
The set $A(f)$ is compact because it is the intersection of a nested sequence of compact sets and it is totally disconnected because the $\text{mesh}(\mathfrak{D}_g)$ goes to zero when $g$ goes to infinity. It is left to prove that $A(f)$ is perfect. First observe that, by Remark \ref{rem:nextloops}, each set  $D_g\in\mathfrak{D}_g$ contains at least two sets $D_{g+1}, D'_{g+1}\in\mathfrak{D}_{g+1}$, see Figure \ref{Passage}.

Now take $x\in A(f)$, then for every $g$ there exists $D_g\in\mathfrak{D}_g$ such that $x\in D_g$. Fix $\epsilon>0$, since $\text{mesh}(\mathfrak{D}_g)$ goes to zero, there exists $g$ with $\text{diam}(D_g)<\epsilon$. Let $D_{g+1}\in\mathfrak{D}_{g+1}$ with $x\in D_{g+1}\subset D_g$. By the previous observation, there exists $D'_{g+1}\in\mathfrak{D}_{g+1}$ such that $D'_{g+1}\subset D_g$ and $D_{g+1}\cap D'_{g+1}=\emptyset$. Take $y\in A(f)\cap D'_{g+1}$, then $y\neq x$ and $\text{dist}(x,y)\leq\text{dist}(D_g)\leq\epsilon$. The proof is now complete.
\end{proof}

In the following lemma we describe the possible limit behavior of points in $A$. In particular, it implies that the set $A$ is non-hyperbolic.

We denote by $\omega(x)$, the $\omega$-limit set of the point $x$.
Recall that $q_2$ is the point with a transversal homoclinic intersection. Let
$$
O_2=\left\{f^t(q_2)\left|\right.t\in\mathbb Z\right\}\cup \left\{p\right\}.
$$
The proof of the following lemmas relies on the directed graphs introduced before. We will need the following notation.
Denote the corresponding prime loops in $X_g$ respectively by $\lambda ^q_g$, $\lambda ^h_g$, and $\lambda ^c_g$. We will use frequently the following observation: if a loop in $X_g$ consists purely of $h-$ and $q-$loops of $X_g$ then
 $w_{g'g}(\lambda)$ consists purely of $h-$ and 
 $q-$loops of $X_{g'}$.

\begin{lem}\label{lem:Adicotomy}
Let $\ell$ be a leaf of the lamination $\mathfrak{L}$ and let $f\in\ell$. Each $x\in A(f)$ satisfies one of the following:
\begin{itemize}
    \item $\omega(x)=\left\{p\right\}$,
    \item $\omega(x)=O_2$,
    \item $\omega(x)=A(f)$.
\end{itemize} 
\end{lem}
\begin{proof}
Let $x\in A(f)$, $x\neq p$, in particular $x=\bigcap_{g}D_g$ where, for all $g$, $D_g\in\mathfrak{D}_g$ and $D_g\subset D_{g+1}$. Fix $g$ and denote by $x_g=\pi_g(D_g)\in X_g$ where $\pi_g$ is the projection of $D_g$ on $X_g$. Observe that $w_{g-1}(x_{g})=x_{g-1}$. We claim that, if for infinitely many $g$,  
\begin{equation}\label{eq:infinitelymanylambdag}
w_{g-1} \left(\overrightarrow{x_gq_g}\right)\supset \lambda^c_{g-1} 
\end{equation}
 then $\omega(x)=A(f)$. 
 Indeed, take an open set $U\subset A(f)$ and let $g_0$ be such that there exists $\tilde D_{g_0}\in\mathfrak{D}_{g_0}$ with $\tilde D_{g_0}\subset U$ and  $$
 w_{g_0+1}\left(\overrightarrow{x_{g_0+2}q_{g_0+2}}\right)\supset \lambda^c_{g_0+1}.
 $$   
 Observe, 
 $$
 w_{g_0 }(\lambda^c_{g_0+1})=X_{g_0}.
 $$
  Hence, 
   there exists $t$ such that $f^t\left(D_{g_0+2}\right)\subset \tilde D_{g_0}$ and in particular $f^t(x)\in \tilde D_{g_0}\subset U $. In particular we showed that under if \eqref{eq:infinitelymanylambdag} holds for infinitely many $g$ then $\omega(x)=A(f)$.

Now consider the opposite situation when there exists $g_0\in \mathbb{N}$ such that
$
w_{g-1} \left(\overrightarrow{x_gq_g}\right)\subset X_{g-1}$ does not contain $\lambda^c_{g-1} $ for every $g\ge g_0$. We will prove in this case that
$$
\omega(x)=O_2.
$$
The proof relies on the following claim. Let $g_0\le g'$. There exists $t_{g'}\ge 0$ such that the point $y=f^{t_{g'}}(x)$ has the property that for every $g>g'$,
$$
w_{g'g}(\overrightarrow{y_{g }q_{g }})=
\lambda_{g'}^q\cup \lambda_{g'}^h.
$$
In other words,
the path $w_{g'g}(\overrightarrow{y_{g }q_{g }})$ is a concatenation of multiple times the $q-$loop   and $h-$loop of $X_{g'}$. The proof of this Claim is by induction. First observe
$$
w_{g' }(\overrightarrow{x_{g'+1 }q_{g'+1 }})=
\lambda\gamma,
$$
where $\gamma=w_{g'}(\overrightarrow{x_{g'+1 }y_{g'+1 }})\subset \lambda^c_{g'}$ is a path of length $t_{g'}$ and $y=f^{t_{g'}}(x)$. Moreover,
$\lambda$ is a loop which is a concatenation of 
multiple $\lambda_{g'}^q$ and one $\lambda_{g'}^h$.

 Suppose the Claim is proved for $g>g'$. Observe,
 $$
 w_{g  }(\overrightarrow{y_{g +1 }q_{g +1 }})=
 \overrightarrow{y_{g  }q_{g  }}\lambda,
 $$
where $\lambda$ is a loop which is a concatenation of 
multiple $\lambda_{g }^q$ and one $\lambda_{g}^h$. This is true because  $w_{g  }(\overrightarrow{y_{g +1 }q_{g +1 }})\subset w_{g  }(\overrightarrow{x_{g +1 }q_{g +1 }})$ which cannot contain the $c-$loop of $X_g$, we are in the second case. 
The induction  assumption implies that
 $$
w_{g'g}(\overrightarrow{y_{g }q_{g }})=\lambda_{g'}^q\cup \lambda_{g'}^h
 $$
Moreover,  using that projections of $q-$ and $h-$loops are concatenations of $q-$ and $h-$loops, we get
$$
w_{g'g}( \lambda)=\lambda_{g'}^q\cup \lambda_{g'}^h.
 $$
This finishes the proof of the Claim.

The Claim implies that for all $g'\ge g_0$,
$$
 \text{Orbit}(f^{t_{g'}}(x))\subset   Q_{g'}\cup\bigcup_{j=1}^{M_{g'}-1}H_{g'}^j .
 $$
Hence,
 
 $$
 \omega(x)\subset \bigcap_{g\ge g_0}\left(Q_g\cup\bigcup_{j=1}^{M_g-1}H_g^j\right)
 $$
 where for all $g$ and for all $j$, $Q_g$ and $H_g^j$ are the sets in $\mathfrak{D}_g$ as described in Definition \ref{def:dynamicalpartition}. 
In order to complete the proof one should establish that 
$$
\bigcap_{g}\left(Q_g\cup\bigcup_{j=1}^{M_g-1}H_g^j\right)=O_2. 
$$
One inclusion is trivial, since each $h$-loop contains $q_2$. Let us concentrate on the other one and for this purpose let us take a point $$
x\in\bigcap_{g}\left(Q_g\cup\bigcup_{j=1}^{M_g-1}H_g^j\right).
$$  
If $x=p$ then $x\in O_2$. Assume $x\neq p$, then from a certain moment on  
$$
x\in \bigcap_{g\geq \tilde g}\bigcup_{j=1}^{M_g-1}H_g^j.
$$  
Here we are using the fact that if $x_g\neq q_g$, then the same holds for all next points in the sequence defining the point $x$. Say, for $g \ge \tilde{g}$
$$
x\in H_g^{j_x}.
$$
 In particular, there exists a number $t_g\in\mathbb Z$ such that $$f^{t_g}(q_2)\in  H_g^{j_x}.$$   
 Because, $f^t(q_2)\to p$  when $|t|\to \infty$ we have that $t_g$ is bounded. 
  Observe also that $f^{t_g}(q_2)$ is close to $x$ because the diameter of $H_g^{j_x}$  goes to zero when $g$ goes to infinity. It follows that $f^t(q_2)=x$ for some $t\in \mathbb{Z}$. 
Hence, $x\in O_2$ 
and the proof is complete.
\end{proof}
A natural question at this point is: does a point with dense orbit exist? The answer is yes. 
\begin{lem}\label{lem:Adense}
Let $\ell$ be a leaf of the lamination $\mathfrak{L}$ and let $f\in\ell$. There exists a point $c\in A(f)$ such that  $\omega(c)=A(f).$
\end{lem}
\begin{proof}
Chose $C_1\in\mathfrak{D}_1$ such that $c_1=\pi_1(C_1)\in\lambda^{c}_1$ where $\pi_1$ is the projection of $\mathcal{D}_1$ on $X_1$. We want to prove that for all $g$ there is an element $c_g$ in the $c$-loop of $X_g$ such that 
$$
w_{g-1} \left(\overrightarrow{c_gq_g}\right)\supset \lambda^c_{g-1} $$
and $w_{g-1}(c_g)=c_{g-1}$.

 Suppose we have constructed $c_1,\dots,c_{g-1}$. Take the first vertex $c_g$ in the $c$-loop of $X_g$ with the property that $w_{g-1}(c_g)=c_{g-1}$. Observe that $$w_{g-1}\left(\overrightarrow{c_gq_g}\right)=\lambda\overrightarrow{c_{g-1}q_{g-1}}$$ where $\lambda$ is a loop and, by Remark \ref{rem:nextloops} $\lambda$ contains all primary loops of $X_{g-1}$. In particular $
w_{g-1} \left(\overrightarrow{c_gq_g}\right)\supset \lambda^c_{g-1} $. 

For all $g$, consider the set of $C_g$ defined by $\pi_g(C_g)=c_g$, where $\pi_g$ is the projection of $\mathcal D_g$ on $X_g$. Our selected point is 
$$
c=\bigcap_g C_g\in A(f).
$$
Observe that $c$ satisfies \eqref{eq:infinitelymanylambdag} for all $g$. Hence, $\omega(c)=A(f)$.
\end{proof}
 \begin{lem}\label{lem:omegaO2}
Let $\ell$ be a leaf of the lamination $\mathfrak{L}$ and let $f\in\ell$. There exists a point $b\in A(f)$ such that  $\omega(b)=O_2.$
\end{lem}
\begin{proof} Observe, for every $x_{g-1}\in \lambda^c_{g-1}\subset X_{g-1}$ there exists a unique $x_{g }\in \lambda^c_{g }\subset X_{g }$ such that
$$
w_{g-1}(\overrightarrow{x_gq_g})=
\lambda\overrightarrow{x_{g-1}q_{g-1}},
$$
where $\lambda$ is a concatenation of  multiple $\lambda^q_{g-1}$ and only one $\lambda^h_{g-1}$.

Choose $x_1\in \lambda^c_1$, $x_1\ne q_1$, and use the above observation to construct the sequence $x_g$, $g\ge 1$. Let $D_g\in \mathfrak{D}_g$ such that
$$
x_g=\pi_g(D_g) 
$$
and $x=\bigcap_g D_g$.

Given $g_0$, let $t_{g_0}$ be such that the point $y=f^{t_{g_0}}(x)$ satisfies 
$$
w_{g_0}(\overrightarrow{y_{ g_0+1}q_{ g_0+1}})=
 \lambda,
$$
where $\lambda$ is a concatenation of  multiple $\lambda^q_{g_0}$ and only one $\lambda^h_{g_0}$.
Now we claim that for all $g_0<g$
$$
w_{g_0g}(\overrightarrow{y_gq_g})=\lambda,
$$
where $\lambda$ is a concatenation of  multiple $\lambda^q_{g_0}$ and   $\lambda^h_{g_0}$.

The proof of the Claim is by induction. It holds for $g=g_0+1$. By construction we have
$$
w_g(\overrightarrow{y_{g+1}q_{g+1}})=
\lambda\overrightarrow{y_gq_g},
$$
where $\lambda$ is a concatenation of  multiple $\lambda^q_{g }$ and  one  $\lambda^h_{g }$. Then
$$
w_{g_0g}(\overrightarrow{y_{g+1}q_{g+1}})=
w_{g_0g}(\lambda)w_{g_0g}(\overrightarrow{y_{g }q_{g }}).
$$
The induction hypothesis implies that $w_{g_0g}(\overrightarrow{y_{g }q_{g }})$  is a concatenation of  multiple $\lambda^q_{g_0 }$ and    $\lambda^h_{g_0 }$. As used before we have 
$w_{g_0g}(\lambda)$ is a concatenation of  multiple $\lambda^q_{g_0 }$ and    $\lambda^h_{g_0 }$. This finishes the proof of the Claim.

The Claim implies for all $g\ge 1$
$$
\omega(x)\subset Q_g\cup\bigcup_{j=1}^{M_g-1}H_g^j
$$
and
$$
\omega(x)\cap Q_g\ne \emptyset,
$$
and for all $j\le M_g-1$
$$
\omega(x)\cap H_g^j\ne \emptyset.
$$
This means that $\omega(x)=O_2$ where we use that
$$
\bigcap_{g}\left(Q_g\cup\bigcup_{j=1}^{M_g-1}H_g^j\right)=O_2. 
$$
\end{proof}

In the next proposition we prove that $A$ is an invariant set.
\begin{prop}
    Let $\ell$ be a leaf of the lamination $\mathfrak{L}$ and let $f\in\ell$. Then the map $f:A(f)\mapsto A(f)$ is an homeomorphism.   
\end{prop}
\begin{proof}
    
    Since $f$ is injective, its restriction to $A$ is injective as well. We only need to prove that $f(A(f))=A(f)$. Observe that for all $x$, $f(\omega(x))=\omega(x)$. In particular, when $x=c$, $f(\omega(c))=\omega(c)$ and by Lemma \ref{lem:Adense} we obtain $f(A(f))=A(f)$. The proof is complete. 
\end{proof}
\section{Topological Equivalence along  the Leaves} 
The aim of this section is to prove that two maps which are in the same leaf of the lamination  are conjugate on their respective Cantor sets, see Theorem \ref{Th:conjonleaves}. In order to do that, we need to build a topological model for the invariant Cantor sets.  

\subsection{Inverse limits}
Take a leaf $\ell$ of the lamination $\mathfrak{L}$ as constructed in \eqref{eq:lamination} and \eqref{eq:leavelamination}. By the construction following Proposition \ref{Prop:inductionstep} we know that for all $g$ there is a $(1,N_g,M_g)$-directed graph and a graph morphism  $w_g:X_{g+1}\mapsto X_g$. 
We can then define the inverse limit as
$$
X=\underleftarrow{\lim}X_g
$$
with the corresponding topology. A sequence $\left\{\underline x_n\right\}_{n\in\mathbb{N}}=\left\{\left(x_{g,n}\right)\right\}_{n\in\mathbb{N}}\subset X$ converges to a point $\underline x=\left(x_g\right)\in X$, namely  
\begin{eqnarray*}
   \left\{\underline x_n\right\}_{n\in\mathbb{N}}\mapsto \underline x\iff\forall g\text{ }\exists n_0\in\mathbb N\text{ s.t }\forall n\geq n_0\text{, }x_{g,n}=x_g.
\end{eqnarray*}
In particular $X$ is a Cantor set. 
\begin{lem}\label{lem:allprevareqs}
Let $\underline{x}\in X$, $\underline x=(x_g)$. If there exists $g_0$ such that $x_{g_0}=q_{g_0}$ then for all $g\leq g_0$, $x_g=q_g$. 
\end{lem}
\begin{proof}
    Since for all $g$, $w_{g-1}(q_g)=q_{g-1}$, then $x_{g_0-1}=w_{g_0-1}(x_{g_0})=w_{g_0-1}(q_{g_0})=q_{g_0-1}$. The same can be repeated for all $g\leq g_0$.
\end{proof}
In the following proposition we prove the existence of an homeomorphism on the inverse limit defined by the directed graphs. 
\begin{prop}
There exists a unique homeomorphism $\phi:X\to X$ with the following property.  For all sequences $\underline x=(x_g)\in X$,  if $\phi(\underline x)=\underline y=(y_g)\in X$ then $\overrightarrow{x_{g}y_{g}}$ is an edge in $X_g$ for all $g$. 
\end{prop}
\begin{proof}
    Let $\underline x=(x_g)\in X$. By Lemma \ref{lem:allprevareqs}, either $\underline x=\underline q=(q_g)$ or there exists $g_0$ such that, for all $g\geq g_0$, $x_g\neq q_g$. For the case when $\underline x=\underline q$, define $\phi(\underline x)=\underline q$. In the other case define $\phi(\underline x)=\underline y=(y_g)$ where $y_g=q_g$ for all $g<g_0$ and $y_g$ is the vertex in the unique outgoing edge starting in $x_g$. Since $w_g(x_{g+1})=x_g$ and since $x_g$ has a unique outgoing edge, if $w_g(\overrightarrow{x_{g+1}y_{g+1}})=\overrightarrow{x_{g}z}$, then $\overrightarrow{x_{g}z}=\overrightarrow{x_{g}y_g}$ and $w_g(y_{g+1})=y_g$. Hence $\phi(\underline x)=\underline y=(y_g)\in X$.
    
    It is left to prove that $\phi$ is an homeomorphism. We start proving that if $\underline{x}\neq\underline{q}$ then $\phi(\underline{x})\neq \underline{q}$. If not, $\phi(\underline{x})=\underline y=(y_g)= \underline{q}=(q_g)$ and for all $g$, $y_g=q_g$. Since $\underline{x}\neq\underline{q}$, there exists $g_0$ such that, for all $g\geq g_0$ $x_g\neq q_g$. In particular, for all $g\geq g_0$, $x_g$ either is the predecessor of $q_g$ in the $c$-loop or $x_g$ is the predecessor of $q_g$ in the $h$-loop. Denote this loop by $\lambda_g$. Take $g\geq g_0+1$, then there exists $z_g\in\lambda_g\setminus q_g$ such that $\overrightarrow{z_gq_g}\subset w_{g-1}^{-1}(q_{g-1})\cap\lambda_g$. In particular we can choose  $z_g$ with the property that the length of $\overrightarrow{z_gq_g}$ is $\underline\ell_g$ which is larger than $2$. Hence, the predecessor of $q_g$ in $\lambda_g$ is contained in this path, i.e. $x_g$ is in $\overrightarrow{z_gq_g}$. This implies that $w_{g-1}(x_g)=q_{g-1}=x_{g-1}$. Contradiction. We proved that if $\underline{x}\neq\underline{q}$, then $\phi(\underline{x})\neq\underline{q}$. 
    
 Now we are ready to prove that the map $\phi$ is injective. Let $\underline{x}\neq \underline{\tilde x}$ and suppose that $\phi(\underline{x})=\phi(\underline{\tilde x})$. By Lemma \ref{lem:allprevareqs}, there exists $g_0$ such that for all $g\geq g_0$, $x_g\neq\tilde{x}_g$ and by assumption $\phi(\underline{x})=\phi(\underline{\tilde x})$, $y_g=\tilde{y}_g$ for all $g$. Since $x_g$ and $\tilde{x}_g$ are predecessors of $y_g$ and $\tilde{y}_g$, respectively, and since they are the same, the only possibility is that $y_g=\tilde{y}_g=q_g$ for all $g\geq q_0$. By Lemma \ref{lem:allprevareqs} $y_g=\tilde{y}_g=q_g$ for all $g$ and $\phi(\underline{x})=\phi(\underline{\tilde x})=\underline q$. By the previous observation, this leads to the contradiction $\underline{x}=\underline{\tilde x}=\underline q$. The injectivity is then proved. 

    In order to prove that $\phi$ is a surjection, let $\underline {y}=(y_g)\in X $. If $\underline {y}=\underline {q}$, then $\phi(\underline {q})=\underline {y}$. Suppose $\underline {y}\neq \underline {q}$. By Lemma \ref{lem:allprevareqs}, there exists $g_0$ such that, for all $g\geq g_0$, $y_g\neq q_g$ and $y_g$ has a unique predecessor $x_g$. Let $w_g(\overrightarrow{x_{g+1}y_{g+1}})=\overrightarrow{z_gy_g}$. Then, by the uniqueness of the predecessor, $z_g=x_g$ and $w_g(x_{g+1})=x_g$. Hence, the sequence $\underline x=(x_g)\in X$ and $\phi(\underline x)=\underline y$. 

    We are now going to prove that $\phi$ is continuous. 

    Take a sequence $\underline{x}_n\in X$ converging to an element $\underline x\in X$. 

    Case 1. If $\underline x\neq\underline q$ then, by Lemma \ref{lem:allprevareqs} there exists $g_0$ such that for all $g\geq g_0$ $x_g\neq q_g$ and $\phi(x_g)=y_g$ is the unique successor of $x_g$. Moreover, since $\underline{x}_n\to\underline{x}$, there exists $n_0$ such that, for all $n\geq n_0$, $x_{g,n}=x_g$. This says that for all $n\geq n_0$, $y_g$ is also the successor of $x_{n,g}$ and $\phi(x_{g,n})=y_{g,n}=y_g$. Hence, $\phi(\underline{x}_n)\to\underline{y}=\phi(\underline{x})$.

    Case 2. If $\underline x=\underline q$, then for all $g$, $x_g=q_g$. Choose $g$. Since $\underline{x}_n\to\underline{x}$, there exists $n_0$ such that, for all $n\geq n_0$, $x_{g+1,n}=q_{g+1}=x_{g+1}$. Let $y_{g+1,n}$ be the successor of $x_{g+1,n}=q_{g+1}$. Then $y_{g+1,n}\in Q_{g+1}\bigcup C_{g+1}^1\bigcup H_{g+1}^1\subset w_g^{-1}(q_g)$. This implies that $y_{g,n}=w_{g+1}(y_{g+1,n})=q_g$. Hence, for all $n\geq n_0$, $y_{g,n}=q_g$. In conlcusion $\phi(\underline{x}_n)\to\underline{q}$.
    
   The construction shows that the homeomorphism is unique.  
\end{proof}
As for the directed graphs, if for all $g$, $\mathfrak{D}_g$ is the $(1, N_g, M_g)$ dynamical partition of a map $f\in\ell$, $\mathcal{D}_g$ is the union of the elements of $\mathfrak{D}_g$ and $i_g:\mathcal{D}_{g+1}\to\mathcal{D}_g$ is the inclusion, we can define the inverse limit as
$$
\mathcal{D}=\underleftarrow{\lim}\mathcal{D}_g.
$$
\subsection{Diagram}
In this section we build a diagram as in Figure \ref{Fig:diagramstructure}. This is the the core of the proof of Theorem \ref{Th:conjonleaves}.
\begin{figure}[h]
\centering
\includegraphics[width=0.9\textwidth]{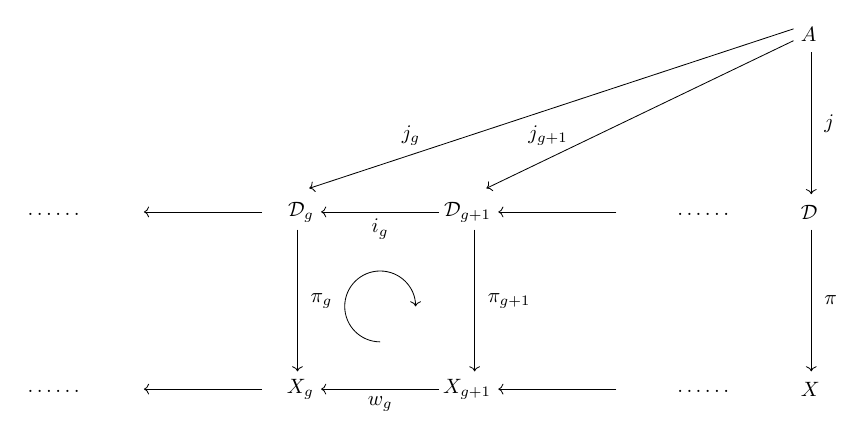}
\caption{Diagram Structure}
\label{Fig:diagramstructure}
\end{figure}
Observe that, for all $g$, $\pi_{g+1}\circ w_g=i_g\circ\pi_g$.
Moreover, for all $g$, let $j_g:A\to\mathcal{D}_g$ be the inclusion and denote the corresponding limit by $j:A\to\mathcal D$. By construction the map $j$ is an homeomorphism. Similarly define the projection $\pi:\mathcal D\to X$.
\begin{lem}
The map $\pi:\mathcal D\to X$ is an homeomorphism. 
\end{lem}
\begin{proof}
Observe that $\pi$ is continuous by construction and since $\mathcal D$ is compact it is enough to show that $\pi$ is injective and surjective. 

Let $\underline{x}=(x_g)\in X$ and for all $g$, let $\pi^{-1}_g(x_g)=U_g\subset\mathcal D_g$. Since for all $g$, $\pi_{g+1}\circ w_g=i_g\circ\pi_g$, it follows that $\dots\subset U_{g+1}\subset U_g$. Let ${a}\in\bigcap U_g$. Then $a\in\mathcal D$ and $\pi(a)=\underline x$. This proves that $\pi$ is a surjective map. Moreover, since $\text{mesh}(U_g)$ tends to zero, then $\left\{a\right\}=\bigcap U_g$. Hence $\pi^{-1}(\underline x)=(a)$. The proof is complete.
\end{proof}
We are now ready to define an homeomorphism $h:A\to X$ as 
$$h=\pi\circ j.$$
\begin{lem}\label{conjugationdiagram}
The following diagram commutes, 
\begin{center}

\begin{tikzpicture}
\coordinate [label=left :$A$]  (x_3) at (-1,1);
\draw[black,  thin, ->]  (-1,1) --  (0.7,1);
\coordinate [label=above :$f$]  (x_3) at (-0.2,1);

\coordinate [label=right :$A$]  (x_3) at (0.7,1);

\draw[black,  thin, ->]  (-1.3,0.7) --  (-1.3,-1);
\coordinate [label=below :$X$]  (x_3) at (-1.3,-1);

\draw[black,  thin, ->]  (1,0.7) --  (1,-1);
\coordinate [label=above :$\phi$]  (x_3) at (-0.2,-1.3);
\coordinate [label=below :$X$]  (x_3) at (1,-1);
\draw[black,  thin, ->]  (-1,-1.3) --  (0.7,-1.3);
\coordinate [label=right :$h$]  (x_3) at (1,0);
\coordinate [label=left :$h$]  (x_3) at (-1.3,0);
\end{tikzpicture}
\end{center}

Namely $h\circ f=\phi\circ h$.
\end{lem}
\begin{proof}
Let $a\in A$. If $a=p$, then $h(f(p))=h(p)=q$ and $\phi(h(p))=\phi(q)=q$.  Assume $a\neq p$, $a=\bigcap_g U_g$ where for all $g$, $U_g$ is a set in $\mathcal D_g$. Let $x_g=\pi_g(U_g)$. By definition $h(a)=\underline{x}=(x_g)$. Since $a\neq p$, there exists $g_0$ such that, for all $g\geq g_0$, $U_g\neq Q_g$ and $x_g\neq q_g$. In particular, $x_g$ has a unique successor $y_g$. Let $V_g=\pi^{-1}(y_g)$. Since $\overrightarrow{x_gy_g}$ is an edge, $f(U_g)\subset V_g$. Hence $\phi(h(a))=\phi(\underline x)=(y_g)=f(h(a))$.
\end{proof}

\begin{theo}\label{Th:conjonleaves}
     If $f$ and $\tilde {f}$ are in the same leaf $\ell$ of the lamination $\mathfrak{L}$ then they are conjugate on their respective Cantor sets.  
\end{theo}
\begin{proof} The dynamical partitions of both maps $f$ and $\tilde f$ give rise to the same   $(1,N_g,M_g)$-directed graph and graph morphism  $w_g:X_{g+1}\mapsto X_g$, for all $g\ge 1$. According to Lemma \ref{conjugationdiagram} both $A(f)$ and $A(\tilde f)$ are conjugate to the corresponding $X$. 
\end{proof}

\begin{rem} Probably, the Cantor set of maps on different leaves are not conjugated. 
    
\end{rem}
\subsection{Codimension-1 Stability}\label{Subsec:renormalization}
In this section, following the Remark \ref{rem:renormaps}, we prove that the leaves of the lamination constructed above are the loci of infinitely renormalizable maps in the following sense. Let $\ell=\cap S_g$ be a leaf of $\mathfrak L$. Recall, from the proof of Lemma \ref{bnlocal}, there exists a smooth function 
$$m_1:S_g\to\mathbb R^2,$$
such that 
\begin{itemize}
    \item $m_1\in W^u_{\text{loc}}(p)\cap C_g^0$,
    \item $m_7=F^{3N_g+(1+\theta_g)n_g+n_{0,g}}(m_1)\in W^u(p)
    $
    \end{itemize}
    and $W^u(p)$ has horizontal tangency at $m_7$.
    \begin{defin}
Let $\ell=\cap S_g$ be a leaf of type $\left\{\left(1, N_g,M_g\right)\right\}_{g=1}^{\infty}$. A map $f\in S_g$ is renormalizable of type $(1,N_{g+1}, M_{g+1})$ if $m_7(f)\in H_g^0(f)$.
        Moreover, $f$ is infinitely renormalizable of type $\left\{\left(1, N_g,M_g\right)\right\}_{g=1}^{\infty}$ if it is renormalizable of type $(1,N_{g+1}, M_{g+1})$ for all $g\geq 1$. 
        \end{defin}

\begin{defin} Let $f$ be an infinitely renormalizable map  of type $\left\{\left(1, N_g,M_g\right)\right\}_{g=1}^{\infty}$. The Cantor set $A(f)$ associated to  $f$ is codimension-$k$ stable if there exists a codimension-$k$ manifold $W$ in parameter space containing  $f $ such that  every map in $W$ is infinitely renormalizable  of type $\left\{\left(1, N_g,M_g\right)\right\}_{g=1}^{\infty}$ and there does not exists a manifold of codimension-$k-1$ with the same property. The manifold $W$ is called the manifold of stability.
        \end{defin}

        \begin{theo}\label{Theo:Acod1stable}
            Let $\ell$ be a leaf of type $\left\{\left(1, N_g,M_g\right)\right\}_{g=1}^{\infty}$. Every  map $f\in \ell$ is infinitely renormalizable of type $\left\{\left(1, N_g,M_g\right)\right\}_{g=1}^{\infty}$    and $A(f)$ is codimension-$1$ stable.  Moreover, $\ell$ is the manifold of stability.       \end{theo}
\begin{proof}
    Clearly if $f\in\ell=\cap S_g$, then $f$ is infinitely renormalizable of type $\left\{\left(1, N_g,M_g\right)\right\}_{g=1}^{\infty}$.
    Consider a map   $\tilde{f}\in  \partial S_{g}$. Observe that, by the construction in Proposition \ref{Prop:inductionstep}, $S_g$ is a strip of the form $\mathcal B_n$. In particular, Lemma \ref{tangencylowerbound} and Lemma \ref{tangencyupperbound} apply to $S_g$. By \eqref{eq:lowerboundarystrip} and \eqref{eq:upperboundarystrip} we have that $m_7(\tilde{f})\notin H_g^0(\tilde{f})$. Hence,    $\tilde{f}$ is not renormalizable. In particular, every neighborhood of $f$ contains maps which are not renormalizable of type $\left\{\left(1, N_g,M_g\right)\right\}_{g=1}^{\infty}$ because $\partial S_g\to \ell$ when $g\to \infty$ .
    The map $f\in \ell$ is codimension-$1$ stable.    
    \end{proof}

\section{Invariant Measures}
In this section we prove that each set $A$ has a unique invariant probability measure. 
\begin{theo}\label{Theo:uniquemeasure}
 For each function $f\in\ell$, the set $A(f)$ has a unique invariant probability measure. This invariant measure is the point mass on the fixed point in $A(f)$ .   
\end{theo}
\begin{cor}\label{Cor:AHyperbolic}
    The set $A$ is not hyperbolic.
\end{cor}
\begin{proof}
If $A$ is hyperbolic, then $A$ is equivalent to a subshift of finite type. Since $A$ has a dense orbit, this subset of finite type, is irreducible, i.e. it has a continuum of ergodic invariant measures. This contradicts Theorem \ref{Theo:uniquemeasure}. (Another reason is that $A$ has only one periodic orbit.) 
\end{proof}
The proof of Theorem \ref{Theo:uniquemeasure} requires some preparation. Let $X$ be the $(1,N,M)$ directed graph.  We denote by $\delta^q$ the measure on $X$ which assigns unit mass to the vertex of the $q$-loop, $\lambda^q$. Similarly, we denote by $\delta^c$ the measure on $X$ which assigns unit mass to each vertex of the $c$-loop, $\lambda^c$ and by $\delta^h$ the measure which assigns unit mass to each vertex of the $h$-loop, $\lambda^h$. 
In follows that $\delta^q$ is a probability measure, $\delta^c(X)=N$ and $\delta^h(X)=M$.

Consider the span of the measures defined above, $$\mathcal M=\text{span}\left\{\delta^q,\delta^c,\delta^h\right\}.$$
Observe that $\mathcal M$ consists of invariant measures in the sense that, if $\mu\in\mathcal M$, then for all $x\in X\setminus\{q, c_0, h_0\}$, 
$$
\mu(x)=\sum_{\overrightarrow{yx}}\mu(y). 
$$
Similarly, for all $g$, if $X_g$ is the $(1,N_g,M_g)$ directed graph, we denote the span of the measures $\delta_g^q, \delta_g^c, \delta_g^h$ by $\mathcal M_g.$ In particular, $\left\{\delta_g^q,\delta_g^c,\delta_g^h\right\}$ is called the standard basis for $\mathcal M_g$.

Moreover the map $w_g:X_{q+1}\to X_g$, induces, by pushing forward the measure, a linear map 
$$
w_{ g* }:\mathfrak{M}_{g+1}\to\mathfrak{M}_{g}.
$$
As in Remark \ref{rem:nextloops}, the winding of $w_g$ is given by 
  \begin{eqnarray*}
        \underline \omega\left(w_g\right)=\left(q, {q}^{\underline{\ell}_g}{c}{q}^{\theta n_g}{c}{q}^{n_{g}}{c}{q}^{n_{0,g}}{h}{q}^{\overline{\ell}_g},  {q}^{\underline{\ell}_g}{h}{q}^{\overline{\ell}_g}\right).
    \end{eqnarray*}
Then the matrix for $w_{g*}$, in the standard basis of $\mathfrak{M}_g$ and $\mathfrak{M}_{g+1}$ can be easily read off from $ \underline \omega\left(w_g\right)$ and it is 
$$
 W_g=\left(\begin{matrix}
    1 &\underline{\ell}_g+\theta n_g+n_g+n_{0,g}+\overline{\ell}_g&\underline{\ell}_g+\overline{\ell}_g\\
    0&3&0\\
    0&1&1
\end{matrix}\right).
$$
The simplex $\mathcal{P}_g\subset \mathfrak{M}_g$ consisting of probability measures is
$$
\mathcal{P}_g=\{\mu=x\delta^q_g+y\delta^c_g+z\delta^h_g| x,y,z\ge 0 \text{ and }\mu(X_g)=1\}.
$$
Observe,
$$
 W_g: \mathcal{P}_{g+1 } \to  \mathcal{P}_{g   }. 
$$
The inverse limit defines  probability measures on $X=\underleftarrow{\lim} X_g$,
$$
\mathcal{P}=\underleftarrow{\lim} \mathcal{P}_g.
$$
Let $\mathcal{I}$ be the set of invariant probability measures of $f:X\to X$. If $\mu\in \mathcal{I}$ then the restriction of the measure to the $\sigma$-algebra of $X_g$ is denoted by
$$
\mu_g=w_{g,\infty *}(\mu)\in  \mathcal{P}_g.
$$
Observe, $w_{g*}(\mu_{g+1})=\mu_g$. Hence, $\mu=(\mu_g)\in \mathcal{P}$. By construction, the map
\begin{equation}\label{eq:invariantprobmeasures}
\mathcal{I}\ni \mu\mapsto (\mu_g)\in \mathcal{P}
\end{equation}
is a bijection. For more details on this construction see \cite{GM}.

\bigskip

Denote the vertices of $\mathcal{P}_g$ by
$$
\underline{\delta}^c_g=\frac{1}{N_g}{\delta}^c_g,\text{     } \underline{\delta}^h_g=\frac{1}{M_g}{\delta}^h_g, \text{     }\underline{\delta}^q_g= \delta ^q_g.
$$
Observe,
$$
\mathcal{P}_g=\text{hull}\{\underline{\delta}^c_g, \underline{\delta}^h_g , \underline{\delta}^q_g \}\subset \mathfrak{M}_g.$$ 
Define the norm $|\mu |$ of  $\mu\in\mathfrak{M}_g$ by
$$
|\mu |=\max\{|x|, |y|, |z|\},
$$
where
$$
\mu=x\underline{\delta}^c_g+y\underline{\delta}^h_g +z \underline{\delta}^q_g.
$$
Observe, if $\mu=x\underline{\delta}^c_g+y\underline{\delta}^h_g +z \underline{\delta}^q_g$ then
$$
\mu(\lambda^q_g)=\frac{x}{N_g}+\frac{y}{M_g}+z.
$$
$$
\mu(\lambda^c_g\setminus \lambda^q_g)=x\cdot \frac{N_g-1}{N_g} .
$$
$$
\mu(\lambda^h_g\setminus \lambda^q_g)=y\cdot \frac{M_g-1}{M_g} .
$$

The map $
 W_g: \mathcal{P}_{g+1 } \to  \mathcal{P}_{g   }
$
contracts towards $\underline{\delta}^q_g$. Namely,

\begin{lem}\label{contractionofmesurelem}If  $N_{g+1},M_{g+1}\ge 32 \max\{N_g, M_g\}$ and $N_g, M_g\ge 3$
 then for all $\mu\in  \mathcal{P}_{g+1 }$
$$
|W_g\mu-\underline{\delta}^q_g |\le \frac12 |\mu -\underline{\delta}^q_{g+1 }|.
$$
 
\end{lem}

\begin{proof} Observe, $W_g\underline{\delta}^q_{g+1}=\underline{\delta}^q_g$. Hence,
it suffices to prove the estimate for $\mu=\underline{\delta}^c_g$ and $\mu= \underline{\delta}^h_g$. We will only present the estimate for $\mu=\underline{\delta}^c_g$, the other is similar. Observe,
$$
W_g\underline{\delta}^c_{g+1}(\lambda^q_g)=\frac{\underline{\ell}_g+\theta n_g+n_g+n_{0,g}+\overline{\ell}_g}{N_{g+1}}.
$$
Because,
$$
N_{g+1}=\underline{\ell}_g+N_g+\theta n_g+N_g+ n_g+N_g+n_{0,g}+M_g+\overline{\ell}_g
$$
we get
$$
W_g\underline{\delta}^c_{g+1}(\lambda^q_g)=\frac{ N_{g+1}-3N_g-M_g}{N_{g+1}}\ge \frac{7}{8}
$$
when $N_{g+1}\ge 32 \max\{N_g, M_g\}$. Hence,
$$
W_g\underline{\delta}^c_{g+1}(\lambda^c_g\setminus\lambda^q_g), W_g\underline{\delta}^c_{g+1}(\lambda^h_g\setminus\lambda^q_g)\le \frac{1}{8}.
$$
Use the coordinates of
$$
 W_g\underline{\delta}^c_{g+1}=x\underline{\delta}^c_g+y\underline{\delta}^h_g +z\underline{\delta}^q_g. 
 $$
 Recall, $x,y,z\ge 0$ and $x+y+z=1$.
 Then
 $$
 x=\frac{N_g}{N_g-1} W_g\underline{\delta}^c_{g+1}(\lambda^c_g\setminus\lambda^q_g)\le \frac{3}{16},
 $$
 when $N_g\ge 3$. Moreover,
 $$
 y=\frac{M_g}{M_g-1} W_g\underline{\delta}^c_{g+1}(\lambda^h_g\setminus\lambda^q_g)\le \frac{3}{16},
 $$
 when $M_g\ge 3$. Because $x+y+z=1$ we get
 $$
 z\ge \frac58.
 $$
The coordinates of 
$
 W_g\underline{\delta}^c_{g+1}-\underline{\delta}^q_g 
$ 
are $(x,y,z-1)$. Hence,
$$
| W_g\underline{\delta}^c_{g+1}-\underline{\delta}^q_g|\le \frac12. 
$$  
\end{proof}

\noindent 
{\it Proof of Theorem \ref{Theo:uniquemeasure}:} Observe that the conditions in Lemma \ref{contractionofmesurelem} are guaranteed by Remark \ref{rem:NMlarge}. Let $\mu\in \mathcal I$ be an invariant probability measure of $f$. In particular,  $\mu$ corresponds to $(\mu_g)\in\mathcal P$, see \eqref{eq:invariantprobmeasures}. We will show that $\mu_g=\underline{\delta}^q_g $, for all $g\ge 1$. The previous lemma implies
that for all $s\ge 1$
$$
|\mu_g-\underline{\delta}^q_g|\le \frac{1}{2^s}|\mu_{g+s}-\underline{\delta}^q_{g+s}|\le \frac{1}{2^s}.
$$
Hence, $\mu_g=\underline{\delta}^q_g $, for all $g\ge 1$. This means that the only invariant measure of $f:A\to A$ is the unit mass on its fixed point. \qed

 \section{Collet-Eckmann Lamination and Newhouse points}
 In the previous sections we observed many different phenomena. In this section we show that all these phenomena occur in one and the same lamination. In particular there exists a lamination with a stable non-hyperbolic Cantor set which contains a Collet-Eckmann point with a dense orbit. Moreover each leave has also a Newhouse point. At this Newhouse points there is an intriguing coexistence of stable and unstable dynamics. This section confirms Theorem B and Theorem C. 
 
 Let us start by giving the formal definition of a Collet-Eckmann point.  
 \begin{defin}\label{def:CE}
Let $f:\mathcal M\to \mathcal M$ be a map and let $x\in \mathcal M$. We say that $x$ is a  Collet-Eckmann (CE) point if there exists a vector $v$, in the tangent space at $x$, such that 
$$
\left| Df^{s}_xv\right|\geq C\rho^s \left| v\right| \text{, \hspace{.49cm}        }\forall  \text{        } s\geq 0
$$
where $C$ and $\rho$ are two constants, $C>0$ and $\rho>1$.
\end{defin}
The following Theorem confirms Theorem B. 
\begin{theo}\label{Theo:CE}
There exists a sub-lamination $\mathfrak{L}_{CE}\subset \mathfrak{L}$ satisfying the following property. For each function $f\in\ell\subset \mathfrak{L}_{CE}$, the invariant non-hyperbolic Cantor set $A(f)$ contains a Collet-Eckmann point $c$ such that  $\omega(c)=A(f)$. Moreover, every leaf of $\mathfrak{L}_{CE}$ is the graph of a smooth function $\ell:[-t_0,t_0]\mapsto [-a_0,a_0]$ and if $t\in [-t_0,t_0]$, the closure of the set $L^{CE}_t=\left\{(t,a)\in\mathfrak L_{CE}\right\}$ is a Cantor set.
\end{theo}
\begin{proof}The proof goes by induction on the generations $g$ defining the lamination. The induction step is contained in Proposition \ref{Prop:inductionstepCE}. Using the notation in that proposition we prove the base of induction, i.e., we prove that the conditions are verified at generation $g=1$. 

Consider the $1$-tangency strip $S$ given by $S=[-t_0,t_0]\times[-a_0,a_0]$ with tangency line $b=\left\{a=0\right\}$. Take $U=S$. In Subsection \ref{familyofunfoldings}, we defined a family $F$ which is an unfolding of the tangency $b$ and each map $f\in S$ has a $(1,N_1,M_1)$ dynamical partition with $N_1=N$ and $M_1=M$, where $N$ and $M$ are defined in Definition \ref{stronghomtang}, property $(f8)$. Define $T=0$ and for all $ f\in U$ let $K$ be a positive constant assuring that $\text{IH}_{H^0}$,  $\text{IH}_{C^0}$-(iii) and $\text{IH}_{C^{N}}$-(iii) hold. The existence of such a constant is guaranteed by the initial conditions on the family. 
 Observe that $\text{IH}_{Q}$, $\text{IH}_{H^0}$, $\text{IH}_{C^0}$-(i) and $\text{IH}_{C^0}$-(ii) are consequence of the definition of our family, while $\text{IH}_{C^N}$ is guaranteed since $T=0$. As for $\text{IH}_{C^0}$-(ii), the lower bound comes from the fact that $b$ is a non-degenerate tangency, while the upper bound is a consequence of the fact that the family is made of $\Cd$ diffeomorphisms. This concludes the base of induction. As said above, the induction step is Proposition \ref{Prop:inductionstepCE}. 
 We can now draw our conclusions. Let 
 $$
 c=\bigcap_gC^{N_g-T_g}\in A(f).
 $$
 Since $T_g$ goes to infinity, $c$ is a Collet-Eckmann point. The proof of the density is the same as the proof of Lemma \ref{lem:Adense}.
  \end{proof}

Even more surprisingly, there exists a sub-lamination of $\mathfrak{L}_{CE}$ whose leaves contains a map with infinitely many sinks of higher and higher order accumulating at the orbit of the Collet-Eckmann point. The following Theorem confirms Theorem C. 
\begin{theo}\label{Theo:CE+N}
There exists a sub-lamination $\mathfrak{L}_{N}\subset\mathfrak{L}_{CE}\subset \mathfrak{L}$ satisfying the following property. Each leaf $\ell\subset \mathfrak{L}_{N}$ contains a map $f_{\ell}$ with infinitely many attracting periodic points, $p_g$ of arbitrarily high  period, such that,
$$
\bigcap_n\overline{\bigcup_{g\geq n}\text{Orb}(p_g)}=A(f_{\ell}).
$$ 
Moreover, every leaf of $\mathfrak{L}_{N}$ is the graph of a smooth function $\ell:[-t_0,t_0]\mapsto [-a_0,a_0]$ and if $t\in [-t_0,t_0]$, the closure of the set $L^{N}_t=\left\{(t,a)\in\mathfrak L_{N}\right\}$ is a Cantor set. 
  \end{theo}
\begin{proof} The theorem is a consequence of Theorem \ref{Theo:CE} and Theorem \ref{Newhousepoints}. More in details, using the terminology from Subsection \ref{sec:Newhouse}, given a point $(t,a)\in NH$ there exists a sequence of Newhouse boxes such that 
$$
\left\{(t,a)\right\}=\bigcap_{g\geq 1}\mathcal P^g_{n^{(g)},n_0^{(g)}}.
$$
Associated to each Newhouse box $\mathcal P^g_{n^{(g)},n_0^{(g)}}$ there is a secondary tangency curve which is the graph of $b^g_{n^{(g)},n_0^{(g)}}$. By taking $n^{(g)}$ large enough we can assure that these curves converge in $\Cuno$ to a leaf $\ell_{(t,a)}$ which is also the graph of a $\Cuno$ function over $ [-t_0,t_0]$. In particular $(t,a)\in\ell_{(t,a)}$.
Let 
$$
\mathfrak L_N=\bigcup_{(t,a)\in NH}\ell_{(t,a)}.
$$
By taking $n^{(g)}$ large enough we can assure that $\mathfrak{L}_{N}\subset\mathfrak{L}_{CE}$.

Let $p_g$ be the periodic sink of generation $g$. By construction the orbit of $p_g$ is contained in the $c$- and $q$-loop of the dynamical partition $\mathfrak D_g$. Hence, the sinks accumulate at the Cantor set $A$. \end{proof}
\section{Appendix}
We prove that there exists a $\theta$ satisfying \eqref{thetacond2}.
\begin{lem}
There exists $\theta\in\left(\theta_0,\theta_1\right)$ such that 
$$
\lambda^{\theta}\mu^{1+\alpha}>  \lambda^{\alpha}.
$$
\end{lem}
\begin{proof}
By \eqref{alphaofuse1}, 
$$
\alpha=2\theta\frac{\log \lambda}{\log\mu}+3.
$$
We need to prove that 
$$
\alpha\left(\log\mu-\log\lambda\right)+\theta\log\lambda+\log\mu>0.
$$
We substitute the value of $\alpha$ and we get 
$$
-3\log\lambda+4\log\mu>\theta\left(2\frac{(\log\lambda)^2}{\log\mu}-3\log\lambda\right).
$$ 
The inequality is satisfied if one substitute the value of $\theta_0=-4\log\mu/3\log\lambda$. Hence, for a value $\theta$ larger then $\theta_0$ and sufficiently close to it, the inequality is still satisfied. The lemma is then proved. 
\end{proof}

 \end{document}